\documentclass{amsart}[11pt]

\usepackage{
	amsmath,
	amssymb,
	amsthm,
	mathdots,
	url,
	lmodern
}
\usepackage[retainorgcmds]{IEEEtrantools}
\usepackage{ enumerate,pifont}
\usepackage{float}
\usepackage{tikz,tikz-cd,ifthen}
\usetikzlibrary{calc,decorations.pathreplacing}
\usepackage{hyperref}

\usepackage[square,numbers]{natbib}

\numberwithin{equation}{section}
\numberwithin{figure}{section}

\newtheorem{theorem}{Theorem}[section]
\newtheorem{corollary}[theorem]{Corollary}
\newtheorem{lemma}[theorem]{Lemma}
\newtheorem{proposition}[theorem]{Proposition}

\newtheorem{question}{Question}

\theoremstyle{definition}
\newtheorem{definition}[theorem]{Definition}
\newtheorem{example}[theorem]{Example}
\newtheorem{remark}[theorem]{Remark}
\newtheorem{observation}[theorem]{Observation}

\newcommand\nc\newcommand
\nc\rc\renewcommand
\nc\dc\DeclareMathOperator

\nc\m\mathbb
\rc\c\mathcal

\newcommand\R{\mathbb{R}}
\newcommand\Q{\mathbb{Q}}
\newcommand\C{\mathbb{C}}
\newcommand\Z{\mathbb{Z}}
\newcommand\N{\mathbb{N}}

\DeclareMathOperator{\id}{id}
\DeclareMathOperator{\im}{im}

\DeclareMathOperator{\coker}{coker}

\DeclareMathOperator{\Hom}{Hom}
\DeclareMathOperator{\Ext}{Ext}
\DeclareMathOperator{\Tor}{Tor}
\DeclareMathOperator{\Crit}{Crit}
\DeclareMathOperator{\std}{std}

\newcommand\ip[1]{\left \langle #1 \right \rangle}

\newcommand\flower{\text{\ding{96}}}

\dc{\CLF}{\textsc{clf}}
\dc{\Ann}{Ann}
\dc{\Fix}{Fix}
\rc\tilde\widetilde

\newcommand{\cat}[1]{{\sffamily\upshape{\textbf{#1}}}}
\newcommand\undermat[2]{%
  \makebox[0pt][l]{$\smash{\underbrace{\phantom{%
    \begin{matrix}#2\end{matrix}}}_{\text{$#1$}}}$}#2}

\tikzset{bd/.style={circle, fill, inner sep=0pt, minimum size=2mm},
	int/.style={circle, draw, fill = white, inner sep=0pt, minimum size=2mm}}

\makeatletter
\renewcommand*\env@matrix[1][*\c@MaxMatrixCols c]{%
  \hskip -\arraycolsep
  \let\@ifnextchar\new@ifnextchar
  \array{#1}}
\makeatother

\begin{document}

\title[Algebraic Properties of Graph Laplacians]{Algebraic Properties of Generalized Graph Laplacians: Resistor Networks, Critical Groups, and Homological Algebra}
  
\author{David Jekel, Avi Levy, Will Dana, Austin Stromme, Collin Litterell}

\address{David Jekel, Dept.\ of Mathematics, University of California,
  Los Angeles, CA}
\email{davidjekel@math.ucla.edu}
\urladdr{\url{http://www.math.ucla.edu/~davidjekel/}}

\address{Avi Levy, Dept.\ of Mathematics, University of Washington,
  Seattle, WA}
\email{avius@uw.edu}
\urladdr{\url{http://www.math.washington.edu/~avius}}

\address{Will Dana, Dept.\ of Mathematics, University of Washingon, Seattle, WA}
\email{danaw6@uw.edu}

\address{Austin Stromme, Dept.\ of Math.\ and Dept.\ of Computer Science, University of Washington, Seattle, WA}
\email{astromme@uw.edu}

\address{Collin Litterell, Dept.\ of Mathematics, University of Washington, Seattle, WA}
\email{collindl@uw.edu}

\date{\today}
\thanks{All the authors acknowledge support of the NSF grant DMS-1460937 during summer 2015.}


\begin{abstract}
We propose an algebraic framework for generalized graph Laplacians which unifies the study of resistor networks, the critical group, and the eigenvalues of the Laplacian and adjacency matrices.  Given a graph with boundary $G$ together with a generalized Laplacian $L$ with entries in a commutative ring $R$, we define a generalized critical group $\Upsilon_R(G,L)$.  We relate $\Upsilon_R(G,L)$ to spaces of harmonic functions on the network using the $\Hom$, $\Tor$, and $\Ext$ functors of homological algebra.

We study how these algebraic objects transform under combinatorial operations on the network $(G,L)$, including harmonic morphisms, layer-stripping, duality, and symmetry.  In particular, we use layer-stripping operations from the theory of resistor networks to systematize discrete harmonic continuation.  This leads to an algebraic characterization of the graphs with boundary that can be completely layer-stripped, an algorithm for simplifying computation of $\Upsilon_R(G,L)$, and upper bounds for the number of invariant factors in the critical group and the multiplicity of Laplacian eigenvalues in terms of geometric quantities.
\end{abstract}

\maketitle


\section{Introduction}

Motivated by questions from several contexts, we study algebraic properties of a generalized critical group.  We relate spaces of harmonic functions to the generalized critical group using homological algebra.  We study how these algebraic objects transform under modifications of the network, including harmonic morphisms, layer-stripping, duality, and symmetry.  \footnote{The previous draft was submitted to the SIAM Journal of Discrete Mathematics in Apr.\ 2016.  This revised version was submitted on Jan.\ 9, 2017.}

\subsection{Layer-stripping for Resistor Networks} \label{subsec:motivationlayering}

Our first motivation comes from the theory of resistor networks developed by \cite{CIM,CM,dVGV,LPcyl,WJ,layering}.  A {\bf graph with boundary} or {\bf $\partial$-graph} is a graph $(V,E)$ together with a specified partition of $V$ into a set $\partial V$ of {\bf boundary vertices} and a set $V^\circ$ of {\bf interior vertices}.  The boundary vertices are the vertices where we will allow a net flow of current into or out of the network.  A {\bf resistor network} is an edge-weighted $\partial$-graph, where each weight or conductance $w(e)$ is strictly positive.  An {\bf electrical potential} is a  function $u \colon V \to \R$.  The {\bf net current} at a vertex $x$ is given by the {\bf weighted Laplacian}
\[
Lu(x) = \sum_{y \sim x} w(x,y)[u(x) - u(y)].
\]
A potential function is {\bf harmonic} if the net current vanishes at each interior vertex.

The discrete electrical inverse problem studied by \cite{CIM,dVGV,WJ,LPcyl,layering} asks whether the conductances of a network can be recovered by performing boundary measurements of harmonic functions.  We measure how the potentials $u|_{\partial V}$ and net currents $Lu|_{\partial V}$ relate for a harmonic function $u$, and we encode this information in a {\bf response matrix} $\Lambda$ (for precise definition, see \cite[\S 3.2]{CM}).  The inverse problem asks whether we can uniquely determine $w$ knowing only $G$ and $\Lambda$.  In other words, for fixed $G$, we want to reverse the transformation $w \mapsto \Lambda$.

The electrical inverse problem cannot be solved for all graphs, but many graphs can be recovered via {\bf layer-stripping}, a technique in which the edge weights are recovered iteratively, working inwards from the boundary.  At each step, one recovers the conductance of a near-boundary edge, then removes that edge by a {\bf layer-stripping operation} of deletion or contraction, and thus reduces the problem to a smaller graph \cite[\S 6.5]{CM}.

Layer-stripping operations have intrinsic algebraic and combinatorial interest as well.  For instance, if a $\partial$-graph can be completely layer-stripped to nothing, then one can construct its response matrix iteratively through simple transformations corresponding to the layer-stripping operations \cite[\S 6]{CM}.  This process parametrizes the response matrices associated for resistor networks which are circular planar (i.e.\ able to be embedded in the disk).  Furthermore, as observed by \cite{LP}, the action of layer-stripping operations on these response matrices generates a group isomorphic to the symplectic group.  In \cite{ALT}, circular planar networks (up to $Y$-$\Delta$ equivalence) are given the structure of a poset with $G' \geq G$ if $G'$ can be layer-stripped down to $G$.

Let us call a $\partial$-graph {\bf layerable} if it can be completely layer-stripped to the empty graph.  In this paper, we will construct an algebraic invariant to test layerability.  Our strategy is to replace edge weights in $\R_+$ with edge weights in an arbitrary commutative ring $R$.  Then we consider the weighted Laplacian as an operator on functions $u: V \to M$, where $M$ is a given $R$-module.  We examine algebraic properties of the module $\c U(G,L,M)$ of harmonic functions and the module $\c U_0(G,L,M)$ of harmonic functions such that $u$ and $Lu$ both vanish on the boundary of $G$.

We show that if a $\partial$-graph is layerable and we assign edge weights which are units in a ring $R$, then $\c U_0(G,L,M) = 0$.  That is, if $u$ is harmonic with $u|_{\partial V} = 0$ and $Lu|_{\partial V} = 0$, then $u$ is identically zero.  The idea is to start with the values on the boundary and work one's way inward following the sequence of layer-stripping operations.  At each step, we deduce that another edge has zero current or that another vertex has zero potential.  In essence, this is a discrete version of harmonic continuation.

The condition that $\c U_0(G,L,M) = 0$ for all $L$ and $M$ does not quite characterize layerable $\partial$-graphs.  If we have a $\partial$-graph $G$ and $\c U_0(G,L,M) = 0$ for every Laplacian $L$ obtained by assigning unit edge weights in any ring $R$, then $G$ may not be layerable.  However, it must be {\bf completely reducible}, that is, it can be reduced to nothing using layer-stripping and another operation which splits apart two subgraphs that are glued together at a common boundary vertex (see Theorem \ref{thm:reducibilitycharacterization}).

Moreover, we can characterize layerability algebraically by generalizing $L$ to allow arbitrary diagonal entries.  As shown in Theorem \ref{thm:layerabilitycharacterization}, $G$ is layerable if and only if $\c U_0(G,L,M) = 0$ for every generalized Laplacian $L$ of the form $D - A$, where $D$ is a diagonal matrix and $A$ is the adjacency matrix weighted by units in $R$.

\begin{remark}
It is important to point out that our invariants do \emph{not} test whether the inverse problem can be solved.  Solving the inverse problem would require not only deleting and contracting a sequence of edges, but also being able to determine the weight of each edge from the boundary behavior of the network.  For a treatment of the inverse problem through layer-stripping, see \cite{WJ,layering}.
\end{remark}

\begin{remark}
It is straightforward to test whether a specific $\partial$-graph is layerable by repeatedly iterating over the boundary vertices searching for edges that can be removed, and this can be done in polynomial time.  The advantage of an algebraic invariant is that it can be used to test layerability for whole classes of networks by relating it to other more global properties (see e.g.\ Proposition \ref{prop:boundaryinteriorbipartite}). It also gives us significant information about non-layerable graphs with boundary.
\end{remark}

\subsection{Harmonic Functions and the Critical Group} \label{subsec:harmoniccritical}

The modules $\c U(G,L,M)$ and $\c U_0(G,L,M)$ of harmonic functions turn out to be related to another $R$-module, which we call the {\bf fundamental module $\Upsilon$}.  The module $\Upsilon$ is a generalization of the \emph{critical group} of a graph (also known as the \emph{sandpile group}, \emph{Jacobian}, or \emph{Picard group}), which has received significant attention from physicists, combinatorialists, probabilists, algebraic geometers, and number theorists.

The critical group can be produced through several different combinatorial models.  The Abelian sandpile model was introduced in statistical physics by Dhar \cite{Dhar}, who was motivated by the study of self-organized criticality. Grains of sand are placed on the vertices of a graph.  If a vertex has at least as many grains of sand as its degree, the vertex is allowed to \emph{topple} by sending one grain of sand to each of its neighbors. The elements of the sandpile group are the critical configurations of sand \cite[\S 14]{GodsilRoyle}, \cite{Biggs2}.  There are other combinatorial models which produce the same group:  Extending work of \cite{Spencer} on the balancing game, \cite{BLS} introduced the chip-firing game and uncovered its connection to greedoids. The dollar game appeared in \cite{Biggs} and was analyzed extensively using the methods of algebraic potential theory.

Sandpile theory has since expanded into other areas of combinatorics, graph theory, and even algebraic geometry.  Graph theorists study the sandpile group in the guise of the quotient of the chain group by the submodule generated by cycles and bonds \cite[\S 26-29]{Biggs}.  Probabilists study the abelian sandpile model due to its intimate connections with generating uniformly random spanning trees \cite{HLMPPW,levine}; the sandpile group acts freely and transitively on the set of spanning trees of the graph \cite[\S 7]{HLMPPW}, \cite[Theorem 7.3]{Biggs2}.  Viewing sand configurations as divisors on the graph, \cite{Lor2,bakerNor2} interpreted the sandpile group as the Jacobian variety of a degenerate curve and proved a Riemann-Roch theorem for graphs.

For such a fruitful object with deep and diverse connections, the critical has a surprisingly simple algebraic characterization.  For a connected graph $G$, if $\Z V$ is the group of $0$-chains on the vertices and $L: \Z V \to \Z V$ is the graph Laplacian, then $\coker(L) \cong \Crit(G) \oplus \Z$ (see \cite[Theorem 4.2]{Biggs2}).  This construction of $\Crit(G)$ easily generalizes to $\partial$-graphs with edge weights in an arbitrary ring.  In the general case, we define $\Upsilon(G,L)$ as the cokernel of the generalized Laplacian $L$ viewed as a map from $0$-chains on the interior vertices to $0$-chains on $V$.  Similar constructions for graphs without boundary appear in \cite{BakerFaber,DKM}.

We relate $\Upsilon$ with harmonic functions by observing that
\[
\c U(G,L,M) = \Hom(\Upsilon(G,L), M);
\]
in other words, $\Upsilon(G,L)$ is the \emph{representing object} for the functor $\c U(G,L,-)$ on $R$-modules.  We also show that $\c U_0(G,L,M) = \Tor_1(\Upsilon(G,L), M)$ (for non-degenerate networks).  In other words, $M$-valued harmonic functions that are not detectable from boundary measurements indicate \emph{torsion} of the fundamental module $\Upsilon$.  These algebraic facts lead to several equivalent algebraic characterizations of layerability and complete reducibility in terms of $\Upsilon$ (Theorems \ref{thm:layerabilitycharacterization} and \ref{thm:reducibilitycharacterization}).

As a special case of our theory, for a graph without boundary with edge weights $1$, we have $\Upsilon \cong \Crit(G) \oplus \Z$.  Moreover,
\[
\c U(G,L,\R/\Z) \cong \Hom_\Z(\Crit(G), \R/\Z) \times \R/\Z \cong \Crit(G) \times \R/\Z,
\]
where the isomorphism $\Hom_\Z(\Crit(G), \R/\Z) \cong \Crit(G)$ follows from Pontryagin duality because $\Crit(G)$ is a finite abelian group.  Thus, we recover the observation of \cite[\S 2]{Solomyak} \cite[p.\ 11]{HLMPPW} that $\Crit(G)$ is isomorphic to the group of $\R/\Z$-valued harmonic functions modulo constants.

This harmonic-function perspective makes the computation of $\Upsilon$ (and hence $\Crit(G)$) accessible to the powerful technique of discrete harmonic continuation, which has proved extremely useful to the resistor network community -- for instance, see \cite[\S 4.1 - 4.5]{CM} \cite[\S 4]{CMdn} \cite{WJ} \cite[\S 2.3]{RK}.   We illustrate this technique in \S \ref{sec:CLF}, using it to compute $\Upsilon$ for a family of $\partial$-graphs embedded on the cylinder.

In \S \ref{sec:layering} we present a systematic approach which uses layer-stripping as a geometric model for harmonic continuation.  As an application, we have the following result (a special case of Theorem \ref{thm:explicitalgorithm}):  Suppose $G$ is a graph without boundary and $G'$ is obtained from $G$ by assigning $s$ vertices to be boundary vertices.  If $G'$ is layerable, then $\Crit(G)$ has at most $s - 1$ invariant factors.  In fact, these invariant factors can be found from the Smith normal form of an $s \times s$ matrix computed explicitly from the sequence of layer-stripping operations.

\subsection{Discrete Differential Geometry and Complex Analysis} \label{subsec:discretedifferentialgeometry}

The generalized critical group $\Upsilon$ serves as a link between the combinatorial properties of a $\partial$-graph and the algebraic properties of harmonic functions, not unlike the way that homology links the topology of a Riemannian manifold to harmonic differential forms.  In light of Hodge theory, harmonic differential forms on a manifold represent elements of the de Rham cohomology groups. On the other hand, these groups are characterized as $\Hom_{\Z}(H^n, \R)$ by the de Rham Theorem, where $H^n$ is the homology of a chain complex defined using formal linear combinations of simplices.  In a similar way, the module $\c U(G,L,M)$ of harmonic functions on an $R$-network can be represented as $\Hom(\Upsilon, M)$, where $\Upsilon$ is obtained by considering formal linear combinations of vertices.

In fact, the analogy between Riemannian geometry and weighted graphs can be made quite precise.  We shall sketch the connection here in a similar way to \cite[\S 2.1]{BakerFaber} and \cite{Mercat}.  We remark also that \cite{DKM} has generalized the critical group to higher dimensions using an analogue of the Hodge Laplacian.

Given a $\partial$-graph $G$ and a commutative ring $R$, we define chain groups
\[
C_0:=RV,\qquad C_1:=RE/\{-e=\bar e\}_{e\in E},
\]
that is, the free $R$-modules on the vertex and edge sets respectively, after identifying the negative of an oriented edge with its reverse orientation.  Dual to chains, we have modules $\Omega^j(G,L,M)$ consisting of $M$-valued $j$-forms:
\[
\Omega^j(G,L,M):=\Hom(C_j, M),\qquad j=0,1.
\]
The boundary map $\partial\colon C_1 \to C_0$ given by $\partial e = e_+ - e_-$ induces the discrete gradient $d\colon \Omega^0 \to \Omega^1$ given by $df(e) = f(e_+) - f(e_-)$.  The coboundary map $\partial^*\colon C_0 \to C_1$ given by $x \mapsto \sum_{e: e_+ = x} e$ induces the discrete divergence $d^*\colon \Omega^1 \to \Omega^0$ given by $d^* \omega(x) = \sum_{e: e_+ = x} \omega(e)$.  The weighted chain Laplacian $\partial w \partial^*\colon C_0 \to C_0$ induces the weighted Laplacian on cochains or functions $d^* w d\colon \Omega^0 \to \Omega^0$.  (Here $w$ denotes the map $e \mapsto w(e)e$.)

For a graph without boundary, the module $\c U(G,L,M)$ arises from (weighted) cohomology theory as the kernel of $d^* w d\colon \Omega^0 \to \Omega^0$.  On the other hand, $\Upsilon(G,L)$ arises from (weighted) homology theory as the cokernel of $\partial w \partial^*\colon C_0 \to C_0$.  In this discrete setting, the de-Rham-like duality $\c U(G,L,M) \cong \Hom_R(\Upsilon(G,L),M)$ follows immediately from properties of quotient modules (see Lemma \ref{lem:hom}).

There is an even better developed analogy between graphs and Riemann surfaces \cite{bakerNor1, Urakawa, Mercat, Perry, BobenkoGunther}, and we will continue to draw inspiration from complex analysis and topology even as we build a purely combinatorial and algebraic theory.  We shall describe analogues of holomorphic maps (\S \ref{subsec:categorydgraphs}), harmonic continuation (\S \ref{sec:layering}), and harmonic conjugates (\S \ref{sec:duality}).

\subsection{Overview}

The paper will be organized as follows:

{\bf 2. The Fundamental Module $\Upsilon(G,L)$:}  This section will define the generalized critical group $\Upsilon(G,L)$ and interpret $\Hom(\Upsilon(G,L),-)$ and $\Tor_1(\Upsilon(G,L),-)$ in terms of harmonic functions (Lemma \ref{lem:hom} and Proposition \ref{prop:tor}).  We give applications to the special case of principal ideal domains and $\Crit(G)$.

{\bf 3. Chain Link Fence Networks:}  We compute $\Upsilon$ for an infinite family of networks with nontrivial boundary which played a key role in the electrical inverse problem \cite{LPcyl}.  This computation illustrates and motivates ideas we develop systematically later (harmonic continuation, covering spaces, sub-$\partial$-graphs).  As a preview of the computation, Figure \ref{fig:randomHF} shows a $(\Z/64)$-valued harmonic function on one of the chain-link fence networks.  This function has $u = 0$ and $L u = 0$ on the boundary of the network.  One interesting corollary of our analysis is that the $\Z$-module of such harmonic functions breaks up into the direct sum of harmonic functions which are zero on the first column of vertices and harmonic functions which are zero on the second column.

\begin{figure}
        \centering
        \includegraphics[width=10cm]{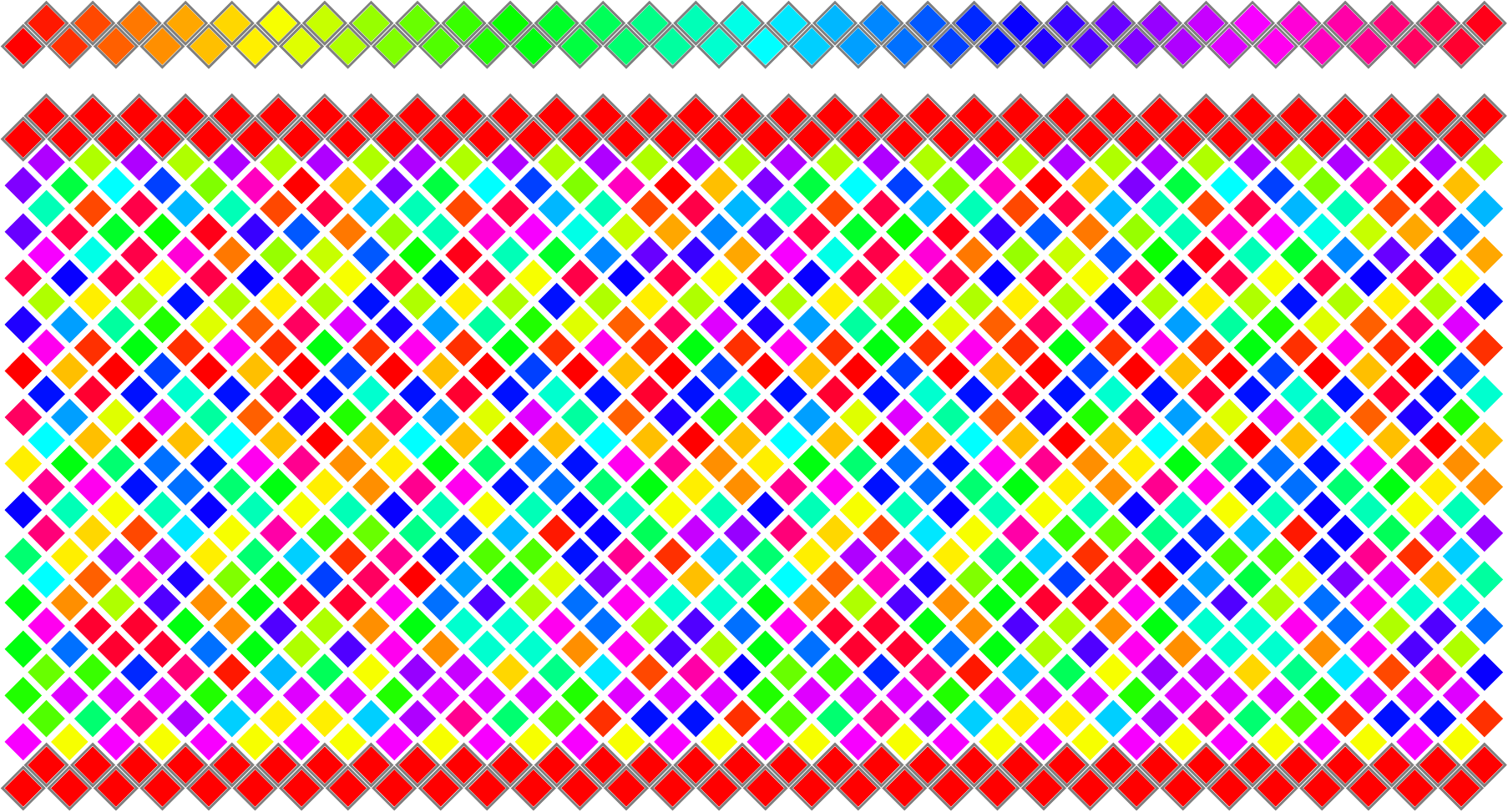}
        \caption{A $\left(\Z/64\right)$-valued harmonic function on a graph we will study in \S \ref{sec:CLF}.  The squares represent vertices, and edges exist between squares which share a side.  The squares on the right and left sides are identified.  The color represents the value of the function; the top bar lists the colors for $0$ through $63$ from left to right.}
        \label{fig:randomHF}
\end{figure}

{\bf 4. The Categories of $\partial$-graphs and $R$-Networks:} We describe categories of $\partial$-graphs and $R$-networks, adapted from the ideas of \cite{Urakawa,bakerNor2,Treumann}.  We show $\Upsilon$ is a covariant functor from $R$-networks to $R$-modules.  We give applications to the critical group and eigenvectors of the Laplacian.

{\bf 5. Layering and Harmonic Continuation:}  We describe the process of layer-stripping a network borrowed from the theory of resistor networks \cite{CIM,dVGV,LPcyl,layering}.  This leads to an algebraic characterization of the finite $\partial$-graphs that can be completely layer-stripped (Theorem \ref{thm:layerabilitycharacterization}).  We use layer-stripping as a geometric model for harmonic continuation.  A systematic approach to harmonic continuation leads to an algorithm for simplifying the computation of $\Upsilon$ (Theorem \ref{thm:explicitalgorithm}).  Corollaries include bounds on the number of invariant factors in the sandpile group and the multiplicity of eigenvalues for graph Laplacians.

{\bf 6. Functorial Properties of Layer-Stripping:}  We relate layer-stripping with the morphisms of $\partial$-graphs from \S \ref{sec:transformations}.  We show that if $f \colon G' \to G$ is a $\partial$-graph morphism and if $G$ can be completely layer-stripped to the empty graph, then so can $G'$ (Lemma \ref{lem:layerabilityPullback}).  The ability to pull back the layer-stripping process leads to a clean description of how far layer-stripping operations can simplify a finite $\partial$-graph in the general case (Theorem \ref{thm:flowerfunctor}).

{\bf 7. Complete Reducibility:} Section \ref{sec:completereducibility} defines a class of completely reducible networks which can be reduced to nothing by layer-stripping operations and splitting apart networks that are glued together at a common boundary vertex.  We prove an algebraic characterization of complete reducibility which is analogous to the one for layerability and apply our theory to boundary-interior bipartite networks.

{\bf 8. Network Duality:} We show that a network and its dual (as in \cite{Perry}) have isomorphic fundamental modules $\Upsilon$, generalizing an earlier duality result for the critical group \cite{Biggs}, \cite{CoriRossin}.  The corresponding statement for harmonic functions is that every harmonic function on $G$ has an essentially unique harmonic conjugate on $G^\dagger$.  Harmonic conjugates provide an alternative approach to a critical group computation of \cite{Biggs} for a simple family of wheel graphs.

{\bf 9. Covering Maps and Symmetry:} We sketch potential applications of symmetry and group actions for understanding the algebraic structure of $\Upsilon$ with special focus on the torsion primes of the critical group.

{\bf 10. Open Problems:} The concluding section hints at further possible applications and generalizations.

\section{The Fundamental Module $\Upsilon(G,L)$} \label{sec:algebra}

We shall generalize the graph Laplacian and critical group in several ways, adapting existing ideas in the literature, especially those of \cite{DKM}, \cite{BakerFaber}, and \cite{CIM}:  Briefly, we will work over an arbitrary ring $R$ rather than $\Z$ or $\R$, assign weights in $R$ to the edges, modify the diagonal terms of $L$ arbitrarily, and choose some \emph{boundary} vertices at which we will not enforce harmonicity.  We will give our general definitions and then describe the examples we have in mind.

We assume familiarity with basic terminology for graphs, categories, rings, and modules, as well as basic homological algebra.  For background, refer to \cite{AtMac}, \cite[Chapters 1-3]{Weibel},\cite[Chapters I, II, III, V]{MacLane}, \cite{Vermani}.  We shall also use theory of modules over a principal ideal domain, including the classification of finitely generated modules and the Smith normal form for morphisms from $R^n \to R^m$ (see \cite[\S 12]{DummitandFoote}).

\subsection{Definitions: $\partial$-graphs, Generalized Laplacians, and the Module $\Upsilon$}

We will take the word {\bf graph} to mean a countable, locally finite, undirected multi-graph. We write $V$ or $V(G)$ for the vertex set of the graph $G$ and $E$ or $E(G)$ for the set of {\bf oriented} edges.  If $e$ is an oriented edge, $e_+$ and $e_-$ refer to its starting and ending vertices, and $\overline{e}$ refers to its reverse orientation.  We use the notation $\mathcal{E}(x) = \{e: e_+ = x\}$ for the set of oriented edges exiting $x$.  The {\bf degree} of a vertex $x$ is the number of such edges, that is, $\deg(x) = |\mathcal{E}(x)|$.

A {\bf graph with boundary} (abbreviated to {\bf $\partial$-graph}) is a graph with a specified partition of $V$ into two sets $V^\circ$ and $\partial V$, called the {\bf interior} and {\bf boundary vertices} respectively.  We will use the letter $G$ to denote $\partial$-graphs as well as graphs.  We will sometimes view a graph without boundary as a $\partial$-graph by taking $V^\circ = V$ and $\partial V = \varnothing$.

Let $G$ be a $\partial$-graph and $R$ a commutative ring.  Then $RV$ will denote the free $R$-module with basis $V$; in the language of topology, $RV$ is the module of $0$-chains or formal $R$-linear combinations of vertices.  Similarly, $RV^\circ$ will denote the free module with basis $V^\circ$, which is a submodule of $RV$.

\begin{definition}
A {\bf generalized Laplacian for $G$ over $R$} is an $R$-module morphism $L: RV \to RV$ of the form
\[
L x = d(x) x + \sum_{e \in \mathcal{E}(x)} w(e)(x - e_-) = \text{ for } x \in V,
\]
where $w$ is a function $E \to R$ satisfying $w(\overline{e}) = w(e)$ and $d$ is a function $V \to R$.
\end{definition}

Here $d(x)$ does not represent the diagonal entry of $L$ at $x$, but rather the difference of the diagonal entry from the standard weighted Laplacian.   Note that $Lx$ can also be written
\[
Lx = \left( d(x) + \sum_{e \in \mathcal{E}(x)} w(e) \right) x - \sum_{e \in \mathcal{E}(x)} w(e) e_-.
\]
Our usage of the term ``generalized Laplacian'' is consistent with \cite[\S 13.9]{GodsilRoyle}.

\begin{definition}
An {\bf $R$-network} is a pair $(G,L)$, where $G$ is a $\partial$-graph and $L$ is an associated generalized Laplacian over $R$.  We call $(G,L)$ an {\bf $R^\times$-network} if $w(e)$ is in the group of units $R^\times$ for every edge $e$; note we do not assume $d(x) \in R^\times$.
\end{definition}

\begin{definition}
For an $R$-network $(G,L)$, we define the {\bf fundamental $R$-module} $\Upsilon_R(G,L)$ as the $R$-module
\[
\Upsilon(G,L) = RV / L(RV^\circ).
\]
When it is helpful to emphasize the ring $R$, we will write $\Upsilon_R(G,L)$.
\end{definition}

\begin{example}
For a $\partial$-graph $G$, the standard graph Laplacian $L_{\std}$ over $R = \Z$ corresponds to the case where $d(x) = 0$ and $w(e) = 1$.  Let $G$ be a finite connected graph (without boundary) considered as a $\partial$-graph by setting $V^\circ = V$.  Then $\Upsilon_{\Z}(G,L_{\std})$ is the cokernel of $L_{\std}: \Z V \to \Z V$, which is known to be isomorphic to $\Crit(G) \oplus \Z$. See \cite[Theorem 4.2]{Biggs2}, \cite[Theorem 14.13.3]{GodsilRoyle}, and \cite[\S 2]{DKM}.
\end{example}

\begin{example}
Lorenzini \cite[pp.\ 481-481]{Lor2} considers the generalized critical group of \emph{arithmetical graphs} constructed by taking $R = \Z$, taking $w(e) = 1$, and choosing $d(x)$ such that the diagonal entries of $L$ are positive and $\ker L$ contains some vector $r: V \to \N$ with positive entries.  Such graphs arise in algebraic geometry.
\end{example}

\begin{example}
For a weighted graph or resistor network as in \cite[\S 3.1]{CM} \cite{CIM}, we consider $R = \R$, let $w(e) > 0$ be the conductance of the edge, and let $d(x) = 0$.  Then $L$ represents the linear map from a potential function in $\R^V$ to the function giving the net current induced at each vertex.  If $G$ is connected and has at least one boundary vertex, then the submatrix of $L$ with rows and columns indexed by the interior vertices will be invertible \cite[Lemma 3.8]{CM}.  Therefore, $L: \R V^\circ \to \R V$ has the maximal rank $|V^\circ|$, so $\Upsilon(G,L)$ will be a vector space over $\R$ of dimension $|\partial V|$.
\end{example}

\begin{example} \label{ex:characteristicpolynomial}
Let $G$ is a finite graph without boundary and $R = \C[z]$.  Then $zI - L_{\std}$ is obtained by taking $w(e) = -1$ and $d(x) = z$.  We can relate $\Upsilon_{\C[z]}(G, zI - L_{\std})$ to the eigenspaces and characteristic polynomial of $L_{\std}$ as follows:  Recall that $L_{\std}$ is symmetric and hence can be written as $S\Lambda S^{-1}$, where $S$ is unitary and $\Lambda$ is a diagonal matrix with diagonal entries given by the eigenvalues $\lambda_1$, \dots, $\lambda_n$ of $L_{\std}$.  Then we have
\begin{align*}
\Upsilon_{\C[z]}(G, zI - L_{\std}) &= \coker_{\C[z]}(zI - L_{\std}) \\
&= \coker_{\C[z]}(S(zI - \Lambda)S^{-1}) \\
&\cong \bigoplus_{j=1}^n \C[z] / (z - \lambda_j).
\end{align*}
The summands $\C[z] / (z - \lambda_j)$ correspond to the eigenspaces of $L_{\std}$.  In the theory of modules over a principal ideal domain, this is an elementary-divisor decomposition of $\Upsilon_{\C[z]}(G, zI - L_{\std})$ over $\C[z]$ (for further algebraic explanation see \cite[\S 12]{DummitandFoote}).  The product of the elementary divisors $(z - \lambda_j)$ is the characteristic polynomial $\det(zI - L_{\std})$.   The characteristic polynomial of the adjacency matrix relates to our theory in a similar way.  We will develop this example further in Example \ref{ex:characteristicpolynomial2} and Proposition \ref{prop:characteristicpolynomial}.
\end{example}

\subsection{Duality between $\Upsilon(G,L)$ and Harmonic Functions}

We mentioned earlier that $\R / \Z$-valued harmonic functions for $L_{\std}$ are related to $\Crit(G)$ through Pontryagin duality.  This will easily generalize to our setting, allowing us to interpret $\Hom(\Upsilon,-)$ and $\Tor_1(\Upsilon,-)$ in terms of harmonic functions.

Recall that if $M$ and $N$ are $R$-modules, then $\Hom_R(M,N)$ is the set of $R$-module morphisms $M \to N$.  For an $R$-module $M$, let $M^V$ be the $R$-module of functions $V \to M$ (or $0$-cochains in the language of topology).  Recall $M^V$ is naturally isomorphic to $\Hom_R(RV, M)$.  The map $L: RV \to RV$ induces a map in the reverse direction $L^* = \Hom(L,M): M^V \to M^V$.  Explicitly, if $u\colon V \to M$, then $L^* u\colon V \to M$ is given by
\[
L^*u(x) = u(L x) = d(x) u(x) + \sum_{e \in \mathcal{E}(x)} w(e) (u(x) - u(e_-)).
\]
Observe that $L_{\std}^*u$ corresponds to the standard Laplacian on functions $V \to \Z$.

If we express $L$ with a matrix using the standard basis for $RV$, then the $L^*: M^V \to M^V$ is given by the transposed matrix.  However, the matrix of $L$ in the standard basis is symmetric, and thus $L^*: M^V \to M^V$ is given by the same matrix as $L$.  Hence, for a finite $\partial$-graph, if we identify $R^V$ with $RV$, then $L$ and $L^*$ are the same operator.

In light of this fact, it does not seem necessary for our notation to distinguish between $L$ and $L^*$.  Henceforth, we will denote them both by $L$.  However, we will preserve the distinction between the chain module $RV$ and the cochain module $R^V$ (and of course the cochain module $M^V$ for each $R$-module $M$).  The domain and codomain for the various operators denoted by $L$ will be made clear in context.

\begin{definition}
Let $(G,L)$ be an $R$-network.  We say that $u: V \to M$ is {\bf harmonic} if $L u(x) = 0$ for every $x \in V^\circ$.  We denote the $R$-module of harmonic functions by
\[
\c U(G,L,M) = \{u \in M^V\colon L u |_{V^\circ} \equiv 0\}.
\]
\end{definition}

Note that $\c U(G,L,-)$ is a covariant functor $R\text{\cat{-mod}} \to R\text{\cat{-mod}}$.  The significance of harmonic functions to the study of $\Upsilon$ comes from the following module-theoretic duality between $\Upsilon(G,L)$ and $\c U(G,L,M)$:

\begin{lemma} \label{lem:hom}
For every $R$-network $(G,L)$, there is a natural $R$-module isomorphism
\[
\c U(G,L,M) \cong \Hom_R(\Upsilon(G,L),M).
\]
\end{lemma}

\begin{proof}
A function $u\colon V \to M$ is equivalent to an $R$-module morphism $u\colon RV \to M$, and $(L u)(x) = (L^*u)(x)= u(L x)$.  Thus, $u$ is harmonic if and only if
\[
(L u)(x) = u(L x) = 0 \text{ for each } x \in V^\circ.
\]
In other words, $u$ is harmonic if and only if it vanishes on $L(RV^\circ)$.  Thus, harmonic functions are equivalent to $R$-module morphisms $RV / L(RV^\circ) \to M$, and $\Upsilon(G,L)$ was defined as $RV / L(RV^\circ)$.
\end{proof}

\subsection{Torsion and Degeneracy} \label{subsec:degeneracy}

A standard way to measure the torsion of an $R$-module $N$ is to use the functors $\Tor_j(N,-)$, which are the left-derived functors of the tensor-product functor $N \otimes -$.  An $R$-module $N$ is called {\bf flat} if $N \otimes -$ is exact, which is equivalent to $\Tor_j(N,-) = 0$ for $j > 0$ (see \cite[Exercise 2.25]{AtMac}).  If $R$ is a principal ideal domain (PID), then $N$ is flat if and only if it is torsion-free (see \cite[Exercise 10.4.26]{DummitandFoote}).

The functor $\Tor_1(\Upsilon(G,L),-)$ turns out to have an easy description in terms of harmonic functions.

\begin{definition}
Define
\[
\c U_0(G,L,M) = \{\text{finitely supported } u \in M^V\colon L u \equiv 0 \text{ and } u|_{\partial V} \equiv 0\}.
\]
\end{definition}

\begin{proposition} \label{prop:tor}
Suppose $R$ is commutative and $(G,L)$ is an $R$-network.  If $\c U_0(G,L,R) = 0$, then we have a natural $R$-module isomorphism
\[
\c U_0(G,L,M) \cong \Tor_1^R(\Upsilon(G,L), M),
\]
and $\Tor_j^R(\Upsilon(G,L), M) = 0$ for $j > 1$.  In the case where $\c U_0(G,L,R) \neq 0$, we still have a natural surjection
\[
\c U_0(G,L,M) \twoheadrightarrow \Tor_1^R(\Upsilon(G,L),M).
\]
\end{proposition}

\begin{proof}
Note that $RV^\circ$ can be interpreted as the module of finitely-supported $R$-valued functions that vanish on $\partial V$.  Thus, $\c U_0(G,L,R)$ is the kernel of the map $L\colon RV^\circ \to RV$.  Since we assumed $\c U_0(G,L,R) = 0$, we know that
\[
\dots \to 0 \to RV^\circ \xrightarrow{L} RV \to \Upsilon \to 0
\]
is a free resolution of $\Upsilon$.  Thus, $\Tor_j(\Upsilon,M)$ is the homology of the sequence
\[
\dots \to 0 \to RV^\circ \otimes M \xrightarrow{L \otimes \id} RV \otimes M \to 0.
\]
Thus, $\Tor_j(\Upsilon,M) = 0$ for $j > 1$, and $\Tor_1(\Upsilon,M)$ is the kernel of the map $L \otimes \id\colon RV^\circ \otimes M \to RV \otimes M$.  We can identify $RV^\circ \otimes M$ with the module of finitely supported functions $u\colon V \to M$ with $u|_{\partial V} = 0$, and then $L \otimes \id$ is simply the generalized Laplacian.  Hence, the kernel of $L \times \id: RV^\circ \otimes M \to RV \otimes M$ is precisely $\c U_0(G,L,M)$.  Thus, $\Tor_1^R(\Upsilon(G,L)) \cong \c U_0(G,L,M)$, and the naturality of the isomorphism with respect to $M$ is easy to verify from the construction.

In the general case, we have a free resolution
\[
\dots \to F_3 \to F_2 \to RV^\circ \xrightarrow{L} RV \to \Upsilon \to 0.
\]
Then $\Tor_1(\Upsilon(G,L),M)$ is obtained as a quotient of the kernel of $L \times \id: RV^\circ \otimes M \to RV \otimes M$, so there is a surjection $\c U_0(G,L,M) \twoheadrightarrow \Tor_1^R(\Upsilon(G,L),M)$.
\end{proof}

We will call a network {\bf non-degenerate} if $\c U_0(G,L,R) = 0$, and {\bf degenerate} if $\c U_0(G,L, R) \neq 0$.  Thus, Proposition \ref{prop:tor} shows that if $(G,L)$ is non-degenerate, then $\c U_0(G,L,M) \cong \Tor_1^R(\Upsilon(G,L),M)$.  We shall now show how non-degeneracy holds whenever the edge weights satisfy the same positivity conditions as resistor networks over $\R_+$.

\begin{definition}
An {\bf ordered ring} \cite[Chapter 6]{lamRing} is a ring $R$ together with a (transitive) total order $<$ given on $R$ such that, for all elements $a,b,c\in R$,
\begin{align*}
  &a<b\implies a+c<b+c,\\
  &0<a\text{ and }0<b\implies 0<ab.
\end{align*}
\end{definition}

\begin{proposition} \label{prop:maxPrin}
Suppose $R$ is an ordered commutative ring and $(G,L)$ is an $R$-network.  Assume $w(e) > 0$ for all $e \in E$ and $d(x) \geq 0$ for all $x \in V$.  Assume $G$ is connected and one of the following holds: (A) $\partial V \neq \varnothing$, (B) $V$ is infinite, or (C) there exists $x \in V$ with $d(x) > 0$.
\begin{enumerate}
	\item If $u$ is a finitely supported harmonic function and $u|_{\partial V} = 0$, then $u = 0$.
	\item If $u$ is a finitely supported harmonic function and $Lu|_{\partial V} = 0$, then $u$ is constant.  Moreover, if (B) or (C) holds, then $u = 0$.
	\item We have $\c U_0(G,L,R) = 0$, so the network is non-degenerate.
\end{enumerate}
\end{proposition}

\begin{proof}
This is a standard argument; one version can be found in \cite{CM}[\S 3.4 and Lemma 3.8].  For finitely supported functions $u, v: V \to R$, denote $\ip{u,v} = \sum_{x \in V} u(x) v(x)$.  Observe that
\[
\ip{u,Lu} = \sum_{x \in V} u(x)\left( d(x) u(x) +  \sum_{e \in \mathcal{E}(x)} w(e)[u(x) - u(e_-)] \right).
\]
If we let $E'$ be a set of oriented edges containing exactly one of the two orientations for each edge, then algebraic manipulation yields
\[
\ip{u,Lu} = \sum_{x \in V} d(x) u(x)^2 + \sum_{e \in E'} w(e)[u(e_+) - u(e_-)]^2.
\]
This is a sum of all nonnegative terms.  Thus, if $\ip{u,Lu} = 0$, we must have $u(e_+) - u(e_-) = 0$ for all $e \in E'$, which implies $u$ is constant since $G$ is connected.  Moreover, if (B) holds, then since $u$ is finitely supported, $u$ must be zero at some vertex and so $u \equiv 0$.  If (C) holds and $d(x) > 0$, then $u(x) = 0$ as well and hence $u \equiv 0$.

To prove (1), suppose $u$ is harmonic (that is, $Lu|_{V^\circ} = 0$) and $u|_{\partial V} = 0$.  Then
\[
\ip{u,Lu} = \sum_{x \in V^\circ} u(x) Lu(x) + \sum_{x \in \partial V} u(x) Lu(x) = 0.
\]
This implies $u$ is constant.  If (A) holds, then $u|_{\partial V} = 0$ implies that $u \equiv 0$, and if (B) or (C) holds, then the preceding argument already implies $u \equiv 0$.

To prove (2), suppose $u$ is harmonic and $Lu|_{\partial V} = 0$.  This amounts to saying $Lu = 0$ on all of $V$, which of course implies that $\ip{u,Lu} = 0$.  Thus, (2) follows from the preceding argument.  Moreover, (3) is an immediate consequence of either (1) or (2).
\end{proof}


In the language of PDE, (1) is a uniqueness principle for the Dirichlet problem and (2) is a uniqueness principle for the Neumann problem.  For a function $u$ to be in $\c U_0$, it must violate \emph{both} uniqueness principles simultaneously.  While this is impossible for $R$-valued functions under the hypotheses of Proposition \ref{prop:maxPrin}, it is entirely possible for functions which take values in a torsion module $M$.  In fact, Proposition \ref{prop:tor} says that torsion of $\Upsilon(G,L)$ corresponds to failure of the uniqueness principles for $M$-valued harmonic functions.  In electrical language, for a non-degenerate network, $\Tor_1(\Upsilon, M) \neq 0$ if and only if there are harmonic $M$-valued functions that are not detectable from boundary measurements of potential and current.

\begin{example}
Figure \ref{fig:torExample} shows a non-degenerate $\Z$-network such that
\[
\Tor_1(\Upsilon_\Z(G,L_{\std}), \Z / 2) = \c U_0(G,L_{\std}, \Z / 2) \neq 0.
\]
Non-degeneracy follows from Proposition \ref{prop:maxPrin}, and a nonzero element of $\c U_0(G,L_{\std})$ is shown in Figure \ref{fig:torExample}.  In fact, $\Upsilon_\Z(G,L_{\std}) \cong \Z^2 \oplus \Z/2$ (see Example \ref{ex:completebipartite}).
\end{example}

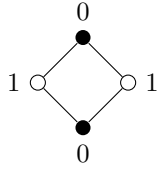
\begin{figure}
\centering
\begin{tikzpicture}[scale = 0.6]
	\node[bd] (A) at (0,0) [label=above:$0$] {};
	\node[int] (B) at (-1,-1) [label=left:$1$] {};
	\node[int] (C) at (1,-1) [label=right:$1$] {};
	\node[bd] (D) at (0,-2) [label=below:$0$] {};
	
	\draw (A) to (B);
	\draw (B) to (D);
	\draw (D) to (C);
	\draw (C) to (A);
\end{tikzpicture}
\caption{A $\Z$-network with $2$-torsion; $\bullet\in \partial V$, $\circ\in V^{\circ}$, and $w(e)=1$.  The numbers depict a nonzero element of $\c U_0(G,L_{\std}, \Z / 2)$.}
\label{fig:torExample}
\end{figure}

Since torsion of $\Upsilon(G,L)$ and degeneracy of $(G,L)$ are both measured by conditions of the form $\c U_0(G,L,M) \neq 0$, it is not surprising that they are related.  As a corollary of Proposition \ref{prop:tor}, we can show that \emph{torsion} of $\Upsilon$ for a non-degenerate network over $R$ is equivalent to \emph{degeneracy} of networks over quotient rings of $R$.  Given an $R$-network $(G,L)$ and an ideal $\mathfrak{a} \subset R$, define $(G,L /\mathfrak{a})$ as the $R / \mathfrak{a}$-network obtained by reducing the edge weights modulo $\mathfrak{a}$.

\begin{corollary} \label{cor:quotientnetwork}
Let $(G,L)$ be a non-degenerate $R$-network.  If $\mathfrak{a}$ is an ideal of $R$, then $\Tor_1(\Upsilon(G,L), R / \mathfrak{a}) = 0$ if and only if $(G, L / \mathfrak{a})$ is non-degenerate.  Hence, $\Upsilon(G,L)$ is flat if and only if $(G, L / \mathfrak{a})$ is non-degenerate for every proper ideal $\mathfrak{a}$.  Moreover, it suffices to check prime ideals or maximal ideals.
\end{corollary}

\begin{proof}
Note that for every function $u\colon V \to R / \mathfrak{a}$, we have $Lu = (L / \mathfrak{a})u$, and hence
\[
\Tor_1^R(\Upsilon_R(G,L), R / \mathfrak{a}) \cong \c U_0^R(G,L,R /\mathfrak{a}) = \c U_0^{R / \mathfrak{a}} (G,L / \mathfrak{a}, R / \mathfrak{a}).
\]
The first claim follows.  For the second claim, recall that an $R$-module $N$ is flat if and only if $\Tor_1(N, R / \mathfrak{a}) = 0$ for all proper ideals $\mathfrak{a} \subset R$, and it suffices to check prime ideals or maximal ideals \cite[Exercise 2.26]{AtMac}, \cite[Corollary 3.2.13]{Weibel}.
\end{proof}

\begin{example} \label{ex:characteristicpolynomial2}
As in Example \ref{ex:characteristicpolynomial}, let $G$ be a finite graph without boundary and consider the $\C[z]$-network $(G, zI - L_{\std})$.  Since $\det(zI - L_{\std}) \neq 0$, we know $zI - L_{\std}$ is an injective map $\C[z]^V \to \C[z]^V$ and hence the network is non-degenerate.  Recall that the maximal ideals of $\C[z]$ are $\{(z - \lambda): \lambda \in \C\}$.  For each $\lambda \in \C$, the quotient $\C[z] / (z - \lambda)$ is a field isomorphic to $\C$ via the obvious map $\C \to \C[z] / (z - \lambda)$.  For each $\lambda$,
\[
\Tor_1(\Upsilon_{\C[z]}(G, zI - L_{\std}), \C[z] / (z - \lambda)) \cong \c U_0(G, zI - L_{\std}, \C[z] / (z - \lambda)).
\]
If we reinterpret the right hand side in terms of the quotient network over $\C[z] / (z - \lambda)$ and apply the standard isomorphism $\C[z] / (z - \lambda) \cong \C$, we see that there is a vector space isomorphism
\[
\c U_0(G, zI - L_{\std}, \C[z] / (z - \lambda)) \cong \c U_0(G, \lambda I - L_{\std}, \C),
\]
where the right hand side is computed over $R = \C$.  This is precisely the $\lambda$-eigenspace of $L_{\std}$ if $\lambda$ is an eigenvalue and zero otherwise.
\end{example}

In the next example, for a finite $\partial$-graph $G$ and a field $F$, we will model a general $F^\times$-network by assigning indeterminates as edge weights.  This construction will be used in our algebraic characterization of layerability (Theorem \ref{thm:layerabilitycharacterization}).

\begin{definition} \label{def:genericfieldnetwork}
Let $G$ be a finite $\partial$-graph and let $F$ be a field.  We define the ring $R^*(G,F)$ and generalized Laplacian $L^*(G,F)$ as follows.  Let $R^*(G,F)$ be the polynomial algebra over $F$ generated by indeterminates $\{t_x: x \in V\}$ and $\{t_e^{\pm 1}: e \in E\}$, where $t_e = t_{\overline{e}}$.  The generalized Laplacian $L^*$ over $R^*$ is given by the functions $w^*(e) = t_e$ and $d^*(x) = t_x$.
\end{definition}

\begin{proposition} \label{prop:genericfieldnetwork}
The $(R^*)^\times$-network $(G,L^*)$ defined above is non-degenerate.  If $\Upsilon(G,L^*)$ is a flat $R^*$-module, then every $F^\times$-network on the $\partial$-graph $G$ is non-degenerate.  Moreover, the converse holds if $F$ is algebraically closed.
\end{proposition}

\begin{proof}
First we prove non-degeneracy.  Note that $\det L^*$ is clearly a nonzero polynomial, hence $L^*: R^*V \to R^*V$ is injective.  If $u \in \c U_0(G,L^*,R^*)$, then we have $L^* u \equiv 0$ (that is, $L^*u(x) = 0$ both for $x \in V^\circ$ and $x \in \partial V$), and therefore, $u \equiv 0$ by injectivity of $L^*$.

Next, we show that flatness of $\Upsilon(G,L^*)$ implies non-degeneracy of every $F^\times$-network $(G,L)$.  Suppose that $L$ is a generalized Laplacian over $F$ given by $w: E \to F^\times$ and $d: V \to F$.  Let $\mathfrak{a}$ be the ideal in $R^*$ generated by $t_e - w(e)$ for $e \in E$ and $t_x - d(x)$ for $x \in V$.  Then $\mathfrak{a}$ is a maximal ideal and $R^* / \mathfrak{a}$ is a field isomorphic to $F$, where $t_e$ is identified with $w(e) \in F^\times$ and $t_x$ is identified with $d(x)$.  Therefore, $(G, L^* / \mathfrak{a})$ corresponds to the $F^\times$-network $(G,L)$.  Thus, by Corollary \ref{cor:quotientnetwork}, we have a vector space isomorphism
\[
\Tor_1^{R^*}(\Upsilon(G,L^*), R^* / \mathfrak{a}) \cong \c U_0(G, L^* / \mathfrak{a}, R^* / \mathfrak{a}) \cong \c U_0(G,L,F).
\]
In particular, flatness of $\Upsilon_{R^*}(G,L^*)$ implies that every $F^\times$-network on the $\partial$-graph $G$ is non-degenerate.

Furthermore, the converse holds if $F$ is algebraically closed.  Indeed, in this case, every maximal ideal $\mathfrak{a}$ of $R^*$ has the form
\[
\mathfrak{a} = (t_e - w(e) \colon  e \in E; t_x - d(x) \colon x \in V)
\]
for some $w: E \to F$ and $d: V \to F$.  (This can be deduced by noting that $R^*$ is a localization of the polynomial algebra $F[t_e: e \in E; t_x: x \in V]$, and every proper ideal in the polynomial algebra $F[x_1,\dots,x_n]$ must have a common zero by Hilbert's Nullstellensatz \cite[Exercise 5.17]{AtMac}, \cite[Corollary 33 of \S 15.3]{DummitandFoote}.) Then since $t_e$ is a unit in $R^*$ by construction, we deduce that $w(e)$ is nonzero.  Thus, $w$ and $d$ define an $F^\times$-network on $G$, which corresponds to $(G, L / \mathfrak{a})$.  Hence, if every $F^\times$-network on $G$ is non-degenerate, then $\Tor_1^{R^*}(\Upsilon(G,L^*), R^* / \mathfrak{a}) = 0$ for every maximal ideal and hence $\Upsilon(G,L^*)$ is flat.
\end{proof}

\subsection{Exactness of $\c U(G,L,-)$}

Another way to measure the torsion of $\Upsilon$ is to test whether $\Upsilon$ is \emph{projective}.  Recall that for an $R$-module $N$, $\Hom(N,-)$ is always left exact, and $N$ is called {\bf projective} if it is also right exact.  The failure of $N$ to be projective is measured by the functors $\Ext^j(N,-)$, which are the right-derived functors of $\Hom(N,-)$, and $N$ is projective if and only if $\Ext^1(N,-) = 0$.  Free modules are always projective.  If $R$ is a PID and $N$ is a finitely generated $R$-module, then $N$ is torsion-free if and only if it is projective (as one can deduce from the classification of finitely generated modules over a PID).

The fundamental module $\Upsilon$ is projective if and only if $\Hom(\Upsilon(G,L), -) = \c U(G,L, -)$ is right exact.  Concretely, right exactness asks:\ given a surjective map $M \to N$ between  $R$-modules, is ${\c U(G,L,M)\to \c U(G,L,N)}$ a surjection?  In other words, does every $N$-valued harmonic function on $(G,L)$ lift to an $M$-valued harmonic function?

\begin{example}
$\c U(G,L, -)$ fails to be right exact for the $\Z$-network in Figure \ref{fig:notExact}.  Consider the surjection ${\Z / 4 \to \Z / 2}$.  The corresponding map $\c U(G,L, \Z / 4) \to \c U(G,L, \Z / 2)$ is not surjective.  If $u \in \c U(G,L, \Z / 4)$, then $2 u(B) = u(A) + u(D) = 2 u(C)$ mod $4$, and hence $u(B) = u(C)$ mod $2$.  However, not all $\Z/2$-valued harmonic functions satisfy $u(B) = u(C)$; for instance, the indicator function $1_B: V \to \Z / 2$ is harmonic.
\end{example}

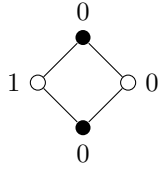
\begin{figure}

\centering
\begin{tikzpicture}[scale = 0.6]
	\node[bd] (A) at (0,0) [label=above:$0$] {};
	\node[int] (B) at (-1,-1) [label=left:$1$] {};
	\node[int] (C) at (1,-1) [label=right:$0$] {};
	\node[bd] (D) at (0,-2) [label=below:$0$] {};
	
	\draw (A) to (B);
	\draw (B) to (D);
	\draw (D) to (C);
	\draw (C) to (A);
\end{tikzpicture}

\caption{A $\Z$-network with $2$-torsion; $\bullet\in \partial V$, $\circ\in V^{\circ}$, and $w(e)=1$.  The numbers depict a $\Z / 2$-valued harmonic function that does not lift to a $\Z / 4$-valued harmonic function.} \label{fig:notExact}

\end{figure}

\subsection{Summary of Homological Properties} \label{subsec:algebrasummary}

Our results thus far provide a lexicon giving ``harmonic'' or ``electrical'' interpretations of the homological properties of $\Upsilon$ for non-degenerate networks:
\begin{enumerate}
	\item As remarked in the proof of Proposition \ref{prop:tor}, $\Upsilon(G,L)$ has a free resolution given by $0 \to RV^\circ \xrightarrow{L} RV \to \Upsilon(G,L) \to 0$.
	\item $\Hom(\Upsilon(G,L),M) = \mathcal{U}(G,L,M)$ is the module of $M$-valued harmonic functions.
	\item $\Tor_1(\Upsilon(G,L),M) = \c U_0(G,L,M)$ is the module of finitely supported harmonic functions with vanishing potential and current on the boundary.
	\item $\Ext^1(\Upsilon(G,L),-)$ is the right-derived functor of $\c U(G,L, -) = \Hom_R(\Upsilon(G,L), -)$. It measures the failure of $N$-valued harmonic functions to lift to $M$-valued harmonic functions when $M \to N$ is surjective.
	\item Using our free resolution of $\Upsilon(G,L)$, we can also compute $\Ext^1(\Upsilon(G,L),M)$ as the cokernel of $L\colon M^V \to M^{V^\circ}$.  In other words, it is the module of $M$-valued functions on $V^\circ$ modulo those that arise as the generalized Laplacian (or net current) of $M$-valued potentials on $V$.
\end{enumerate}

These observations lead to many different ways of computing $\Upsilon$ when the ring $R$ is a principal ideal domain (PID) such as $\Z$ or $\C[z]$ (Proposition \ref{prop:pidnon-degenerate}).  We recall the following terminology and facts about PIDs:  A ring is called a {\bf principal ideal domain (PID)} if every ideal is generated by a single element.  The classification of finite abelian groups (or $\Z$-modules) generalizes to PIDs: If $R$ is a PID, then any finitely generated $R$-module is isomorphic to one of the form
\[
R^{\oplus m} \oplus \bigoplus_{j=1}^n R / f_j,
\]
where $f_j | f_{j+1}$.  The numbers of $f_j$ are called the {\bf invariant factors} of $M$.  We call $R^{\oplus m}$ and $\bigoplus_{j=1}^n R / f_j$ respectively the {\bf free submodule} and {\bf torsion submodule} of $M$.  For details, see \cite[\S 12.1]{DummitandFoote}.

\begin{proposition} \label{prop:pidnon-degenerate}
Let $R$ be a PID, $F$ its field of fractions, $(G,L)$ a finite non-degenerate $R$-network.  Then the free submodule of $\Upsilon(G,L)$ has rank $|\partial V|$.  Moreover, the following are (non-canonically) isomorphic:
\begin{enumerate}
	\item The torsion submodule of $\Upsilon(G,L)$.
	\item The cokernel of $L\colon R^V \to R^{V^\circ}$.
	\item $\c U(G,L, F / R)$ modulo the image of $\c U(G,L, F)$.
	\item $\c U_0(G,L, F / R) = \Tor_1(\Upsilon(G,L), F / R)$.
\end{enumerate}
Thus, for instance,
\[
\Upsilon(G,L) \cong R^{|\partial V|} \oplus \c U_0(G,L, F / R).
\]
\end{proposition}

\begin{proof}
To show that the free rank is $|\partial V|$, note that since $(G,L)$ is non-degenerate, we have a short exact sequence
\[
0 \to RV^\circ \xrightarrow{L} RV \to \Upsilon(G,L) \to 0.
\]
Let $\Upsilon(G,L) \cong R^n \oplus N$, where $N$ is a torsion $R$-module.  Since $F$ is a flat $R$-module, we have a short exact sequence
\[
0 \to F \otimes RV^\circ \xrightarrow{L} \to F \otimes RV \to (F \otimes R^n) \oplus (F \otimes N) \to 0.
\]
However, $F \otimes N = 0$.  Thus, our sequence becomes $0 \to FV^\circ \to FV \to F^n \to 0$, and the rank-nullity theorem implies $n = |V| - |V^\circ| = |\partial V|$.

Next, we prove that (1) -- (4) are isomorphic.  Note (2) and (3) are two different ways of evaluating $\Ext^1(\Upsilon(G,L),R)$; (2) uses the projective resolution of $\Upsilon(G,L)$ and (3) uses the injective resolution of $R$ given by $0 \to R \to F \to F / R \to 0$.  To show that the torsion submodule of $\Upsilon(G,L)$ is isomorphic to $\Tor_1(\Upsilon(G,L), F/R)$ and $\Ext^1(\Upsilon(G,L),R)$, decompose $\Upsilon(G,L)$ as the direct sum of cyclic modules.

The last equation follows because $\Upsilon(G,L) \cong R^n \oplus N$, where $n = |\partial V|$ and $N \cong \c U_0(G,L,F/R)$ by the isomorphism between (1) and (4).
\end{proof}

\begin{figure}
\begin{center}
	\begin{tikzpicture}[scale=0.8]
		\node[int] (A) at (-0.5,1) {};
		\node[int] (B) at (0.5,1) {};
		\node[bd] (1) at (-1,0) {};
		\node[bd] (2) at (0,0) {};
		\node[bd] (3) at (1,0) {};
		
		\draw (B) to (1) to (A) to (2) to (B) to (3) to (A);
	\end{tikzpicture}
\end{center}

\caption{The complete boundary-interior bipartite $\partial$-graph $K_{3,2}$.} \label{fig:K32}

\end{figure}
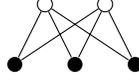

\begin{example}[Complete Bipartite Graphs] \label{ex:completebipartite}
Consider the complete bipartite graph $K_{m,n}$ whose partite sets consist of $m$ boundary vertices and $n$ interior vertices respectively (see Figure \ref{fig:K32}).  The standard Laplacian $L_{\std} \colon R^V \to R^{V^\circ}$ is
\[
    L_{\std} = \left(\begin{array}{cccc|cccc}
      -1 &\cdots&\cdots & -1 & m & & &\\
      \vdots &\ddots&\ddots & \vdots &  & \ddots& &\\
      \vdots &\ddots&\ddots & \vdots &  & &\ddots&\\
      \vspace{2mm}\undermat{\partial V}{-1 &\cdots&\cdots & -1} & \undermat{V^{\circ}}{\hspace{4mm} & \hspace{4mm} &\hspace{4mm} &m\\}
    \end{array}\right)\\
\]
\newline
and its cokernel is isomorphic to $(\Z/m)^{n-1}$.  Dually,
\[
\c U_0(G,L_{\std}, \Q / \Z) = \left\{u \in (\Q / \Z)^{V^\circ}: m u = 0, \sum_{x \in V^\circ} u(x) = 0 \right\} \cong (\Z / m)^{n-1}.
\]
Hence, by Proposition \ref{prop:pidnon-degenerate}, we have
\[
\Upsilon(G,L_{\std}) \cong \Z^m \oplus (\Z / m)^{n-1}.
\]
\end{example}

\subsection{Application to Critical Group} \label{subsec:criticalgroupalgebra}

Let us examine how our algebraic constructions work out in the case of the critical group.

\begin{proposition} \label{prop:criticalgroupnoboundary}
Let $G$ be a connected graph without boundary, considered as a $\partial$-graph with zero boundary vertices.
\begin{enumerate}
	\item $\Crit(G)$ is the torsion submodule of $\Upsilon_{\Z}(G,L_{\std})$.
	\item For every vertex $x$, the map $\Z x \to \Upsilon_{\Z}(G,L_{\std})$ is injective and we have an internal direct sum $\Upsilon_{\Z}(G,L_{\std}) = \Z x \oplus \Crit(G)$.
	\item We have $\c U(G,L_{\std},\Q / \Z) \cong \Crit(G) \times \Q / \Z$, where $\Q / \Z$ represents the constant functions $u: V \to \Q / \Z$.
	\item We have $\c U_0(G,L_{\std}, \Q / \Z) = \c U(G,L_{\std}, \Q / \Z)$.
\end{enumerate}
\end{proposition}

\begin{proof}
Let $\epsilon: \Z V \to \Z$ be the map given by $x \mapsto 1$ for every $x \in V$.  Note that $L_{\std}(\Z V) \subseteq \ker \epsilon$.  Moreover, it is well-known that $L_{\std}$ has rank $|V| - 1$ when $G$ is connected (see e.g.\ \cite[Lemma 3.8]{CM}), so that $\ker \epsilon / \im L_{\std}$ is a torsion $\Z$-module, and it this is known to be isomorphic to $\Crit(G)$ \cite[Theorem 4.2]{Biggs2} \cite[Definition 2.2]{DKM}.  For each vertex $x$, we have $\Z V = \Z x \oplus \ker \epsilon$, which implies that
\[
\Z V / \im L_{\std} = \Z x \oplus \Crit(G).
\]
This establishes (1) and (2).  Next, (3) follows by applying $\Hom(-,\Q/\Z)$ using Lemma \ref{lem:hom}, and (4) follows from the definition of $\c U$ and $\c U_0$ because there are no boundary vertices.
\end{proof}

Several constructions of the critical group involve designating a ``sink'' vertex $x$.  In a similar way, we can choose a boundary vertex $x$ when computing $\Crit(G)$.

\begin{proposition} \label{prop:criticalgrouponeboundary}
Let $G$ be a connected graph without boundary, and let $G'$ be obtained from $G$ by assigning one boundary vertex $x$.
\begin{enumerate}
	\item The network $(G',L_{\std})$ is non-degenerate.
	\item We have $\Upsilon_{\Z}(G',L_{\std}) = \Upsilon_{\Z}(G,L_{\std})$, hence (1), (2), (3) of Proposition \ref{prop:criticalgroupnoboundary} hold with $G$ replaced by $G'$.
	\item We have $\c U_0(G,L_{\std}, \Q / \Z) \cong \Tor_1(\Upsilon(G',L_{\std}), \Q / \Z) \cong \Crit(G)$.
\end{enumerate}
\end{proposition}

\begin{proof}
Non-degeneracy follows from Proposition \ref{prop:maxPrin}.  Let $V^\circ$ denote $V^\circ(G') = V(G) \setminus \{x\}$.  To prove that $\Upsilon(G,L_{\std}) = \Upsilon(G',L_{\std})$, it suffices to show that $L_{\std}(\Z V^\circ) = L_{\std}(\Z V)$.  Recall that the constant vector $c_0$ is in the kernel of $L_{\std}$.  If $w \in L_{\std}(\Z V)$, then $w = L_{\std} z$ for some $z \in \Z V$.  By subtracting a multiple of the $c_0$, we can assume that the coordinate of $z$ corresponding to the vertex $x$ is zero.  This means $z \in \Z V^\circ$, so $w \in L_{\std}(\Z V^\circ)$.

From $\Upsilon(G,L_{\std}) = \Upsilon(G',L_{\std})$, it immediately follows that $\Crit(G)$ is the torsion submodule of $\Upsilon(G',L_{\std})$.  From the application of $\Hom_{\Z}(-,M)$, we see that $\c U(G,L_{\std},M) = \c U(G',L_{\std},M)$ for every $\Z$-module $M$, and hence $\c U(G',L_{\std},\Q / \Z) \cong \c U(G,L_{\std},\Q / \Z) \cong \Crit(G) \times \Q / \Z$.  Finally, (3) follows from Proposition \ref{prop:pidnon-degenerate}.
\end{proof}

\section{A Family of `Chain Link Fence' Networks} \label{sec:CLF}

\subsection{Motivation and Set-Up} \label{subsec:CLFsetup}

To date, the theory of the critical group has mainly focused on graphs without boundary.  In this section, we will analyze an infinite family of $\partial$-graphs with nontrivial boundary.  These $\partial$-graphs resemble a chain-link fence which embeds either on the cylinder or on the M\"obius band (depending on parity).  This family is a variant of the ``purely cylindrical'' graphs described in \cite{LPcyl} which play a key role in the electrical inverse problem.  Though self-contained, our computation here will illustrate and motivate techniques that we will develop systematically later in the paper--including discrete harmonic continuation, symmetry and covering spaces, and subgraphs.

  \begin{figure}
    \begin{center}
      \begin{tikzpicture}[font=\fontsize{8}{8.5}\selectfont,inner sep=1pt,scale = 0.7]
        \draw[dashed, line width = .3mm, blue] (0,0) -- (0,-5);
        \draw[dashed, line width = .3mm, blue] (8,0) -- (8,-5);

        \foreach \j in {0,2,4,6} {
          \foreach \k in {0,1,2} {
            \draw (\j,-2 * \k) -- ({\j+1},{(-2 * \k) -1});
          }
          \foreach \k in {1,2} {
            \draw (\j,-2 * \k) -- ({\j+1},{(-2 * \k) + 1});
          }
        }
        
        \foreach \j in {1,3,5,7} {
          \foreach \k in {0,1,2} {
            \draw (\j, {2 * \k - 5}) -- ({\j+1},{2 * \k - 4});
          }
          \foreach \k in {1,2} {
            \draw (\j,{2 * \k - 5}) -- ({\j+1},{2 * \k - 6});
          }
        }
        
        \foreach \j in {0,2,4,6,8} {
          \draw
            let \n1={int(mod(\j,8))}
            in (\j,0)
              node[circle, fill]
              {\color{white}  ${\bf(\n1,0)}$};
          \foreach \k in {1,2} {
            \draw
              let \n1={int(mod(\j,8))}
              in (\j,{-2 * \k})
                node[circle, draw, fill=white]
                { ${\bf (\n1,\k)}$};
          }
        }
        
        \foreach \j in {1,3,5,7} {
          \node[circle,fill] at (\j,-5)
            {\color{white} ${\bf (\j,0)}$};
          \foreach \k in {1,2} {
            \node[circle,draw,fill=white] (v{\j}{\k}) at (\j,{-5 + 2 * \k})
              { ${\bf (\j,\k)}$};
          }
        }   
      \end{tikzpicture}

      \caption{\label{fig:CLFpicture}
        The $\partial$-graph $\CLF(8,2)$. Boundary vertices are black, interior vertices are white, and the vertices on the left and right sides are identified along the dashed lines.}

    \end{center}
  \end{figure}
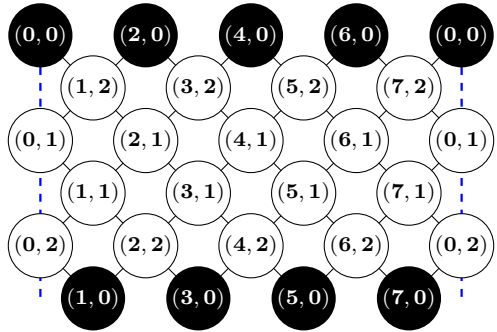

\begin{figure}

\newcommand\lam[3]{ 
	\begin{scope}[scale = #3]
		\draw[red,dashed] (0,0) circle (1);
		\draw[red,dashed] (0,0) circle ({#2 + 1});
	
		\foreach \i in {1,...,#1} {
			\foreach \j in {1,...,#2} {
				\draw ({(\i + 0.5 * \j - 0.5) * 360 / #1}:\j) -- ({(\i + 0.5 * \j) * 360 / #1}:{\j + 1});
				\draw ({(\i - 0.5 * \j + 0.5) * 360 / #1}:\j) -- ({(\i - 0.5 * \j) * 360 / #1}:{\j + 1});
			}
		}
		\foreach \i in {1,...,#1} {
			\fill ({\i * 360 / #1}: 1) circle (0.2);
			\foreach \j in {2,...,#2} {
				\path[draw=black, fill=white] ({(\i + 0.5 * \j - 0.5) * 360 / #1}:\j) circle (0.2);
			}
			\fill ({(\i + 0.5 * #2) * 360 / #1}:{#2 + 1}) circle (0.2);
		}
	\end{scope}
}

\begin{center}
	\begin{tikzpicture}
		\lam{6}{3}{0.7}
	\end{tikzpicture}
	
	\caption{The $\partial$-graph $\CLF(12,1)$ embedded in the annulus rather than the cylinder.}  \label{fig:CLFflower}
\end{center}
\end{figure}
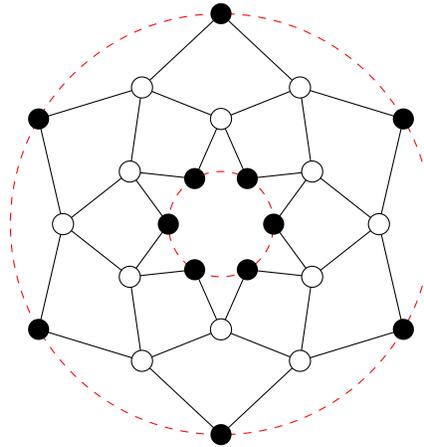

Consider a $\partial$-graph $\CLF(m,n)$ with $V = \Z / m \times \{0, \dots, n\}$ and $\partial V = \Z / m \times \{0\}$ and edges defined by
\begin{align*}
(j,k) &\sim (j+1,n-k+1) \text{ for } k \geq 1 \\
(j,k) &\sim (j+1,n-k) \text{ for } k \geq 0,
\end{align*}
as shown in Figure \ref{fig:CLFpicture}.  If $m$ is even then the network is one of Lam and Pylyavksyy's `purely cylindrical' graphs \cite{LPcyl}.  If $m$ is odd, then it resembles a chain-link fence twisted into a M\"obius band.

Consider the $\Z$-network $(\CLF(m,n),L_{\std})$.  Since the network is non-degenerate (Proposition \ref{prop:maxPrin}), we have by Proposition \ref{prop:pidnon-degenerate} that
\[
\Upsilon_{\Z}(\CLF(m,n), L_{\std}) \cong \Z^{|\partial V|} \oplus \c U_0(\CLF(m,n), L_{\std}, \Q / \Z).
\]
We will compute the torsion summand $\c U_0(\CLF(m,n), L_{\std}, \Q / \Z)$, showing that

\begin{theorem} \label{thm:CLF}
For the $\Z$-network $(\CLF(m,n),L_{\std})$, we have
\[
\c U_0(\CLF(m,n), L_{\std}, \Q / \Z) \cong
\begin{cases}
	(\Z/2)^n,& m\text{ odd} \\
	(\Z/2)^{2n},& m\equiv 2\mod 4 \\
	\displaystyle \left( \bigoplus_{j=1}^n \Z / \gcd(4^j,2m) \right)^{\oplus 2}, & m\equiv 0\mod 4.
\end{cases}
\]
\end{theorem}

In \S \ref{subsec:CLFcontinuation}, we use harmonic continuation to write our module in a simple form in terms of a $2n \times 2n$ matrix $\mathbf{T}$(Lemma \ref{lem:precomputation}).  Then in \S \ref{subsec:algebraiccomputation}, we work algebraically to find the invariant factor decomposition.  Finally, in \S \ref{subsec:CLFvariants}, we will bootstrap our computation to handle a slightly different family of $\partial$-graphs.

\subsection{Harmonic Continuation Computation} \label{subsec:CLFcontinuation}  Since we will deal with vectors in $\Z^n$ as well as $\Z^{2n}$, we establish the following notational conventions:
\begin{itemize}
	\item Vectors in $\Z^n$ or $(\Q / \Z)^n$ will be lowercase regular type.
	\item $n \times n$ matrices will be uppercase regular type.
	\item Vectors in $\Z^{2n}$ or $(\Q / \Z)^{2n}$ will be lowercase bold.
	\item $2n \times 2n$ matrices will be uppercase bold.
	\item ``$\cdot$'' denotes the dot product.
	\item $e_1$, \dots, $e_n$ and $\mathbf{e}_1$, \dots, $\mathbf{e}_{2n}$ denote the standard basis vectors.
	\item Vectors are assumed to be column vectors by default.
\end{itemize}
Moreover, we will abbreviate $\c U_0(\CLF(m,n), L_{\std}, \Q / \Z)$ to $U(m,n)$.

Our goal is to compute the $\Q / \Z$-valued harmonic functions with $u = L u = 0$ on $\partial V$.  We start by understanding the harmonic functions with $u = 0$ on the boundary using harmonic continuation around the circumference of the cylinder.  Assume $u(j,0) = 0$ and let
\[
a_j = \begin{pmatrix} u(j,1) \\ \vdots \\ u(j,n) \end{pmatrix} \in (\Q/\Z)^n.
\]
The idea is to solve for $a_{j+1}$ in terms of $a_j$ and $a_{j-1}$, such that partial function defined by $a_{j-1}$, $a_j$, and $a_{j+1}$ will be harmonic on the $j$th column of vertices.  Thus, we start with $a_1$ and $a_0$, then find $a_2$, $a_3$, \dots.  Recall the index $j$ for vertices in the graph is reduced modulo $m$.  The $a_j$'s to yield a well-defined harmonic function on $\CLF(m,n)$, we require that $a_m = a_0$ and $a_{m+1} = a_1$.

In terms of the $a_j$'s, harmonicity amounts to
\[
4a_j = E a_{j-1} + E a_{j+1},
\]
where $E$ is the $n \times n$ matrix with $1$'s on and directly above the skew-diagonal and zeros elsewhere--for instance,
\[
E = \begin{pmatrix} 0 & 0 & 0 & 1 & 1 \\ 0 & 0 & 1 & 1 & 0 \\ 0 & 1 & 1 & 0 & 0 \\ 1 & 1 & 0 & 0 & 0 \\ 1 & 0 & 0 & 0 & 0 \end{pmatrix}, \quad n = 5.
\]
Thus, the vectors $a_j$ satisfy the recurrence relation
\[
\begin{pmatrix} a_{j+1} \\ a_j \end{pmatrix} = \begin{pmatrix} 4E^{-1} & -I \\ I & 0 \end{pmatrix} \begin{pmatrix} a_j \\ a_{j-1} \end{pmatrix}.
\]

Let $\mathbf{T}$ be the $2n \times 2n$ ``propagation matrix'' of harmonic continuation, that is,
\[
\mathbf{T} = \begin{pmatrix} 4E^{-1} & -I \\ I & 0 \end{pmatrix}.
\]
Note that $\det \mathbf{T} = \pm 1$, so $\mathbf{T}$ is invertible over $\Z$.  Multiplying by $\mathbf{T}^{-1}$ corresponds to harmonic continuation in the opposite direction around the circumference of the cylinder.

Let us denote $\mathbf{a} = (a_1, a_0)^T \in (\Q / \Z)^{2n}$.  Through harmonic continuation, we can see that
\[
\{u \in \c U(\CLF(m,n), \Q / \Z): u|_{\partial V} = 0\} \cong \{ \mathbf{a} \in (\Q / \Z)^{2n}: \mathbf{T}^m \mathbf{a} = \mathbf{a}\}.
\]
Next, we must determine when a fixed point of $\mathbf{T}^m$ will yield a harmonic function $u$ with $L u|_{\partial V} = 0$, which amounts to writing all the net current conditions in terms of the first two columns of vertices.  The net current at a boundary vertex $(j,0)$ is
\[
L u(j,0) = -u(j-1,n) - u(j+1,n) = -\mathbf{e}_{2n} \cdot \mathbf{T}^{j-1} \mathbf{a} - \mathbf{e}_{2n} \cdot \mathbf{T}^{j+1} \mathbf{a}.
\]
We need to choose $\mathbf{a}$ so that this holds for $j = 1, \dots, m$, but since we also require $\mathbf{a}$ to be a fixed point of $\mathbf{T}^m$, we might as well require $\mathbf{e}_{2n} \cdot (\mathbf{T}^{j-1} + \mathbf{T}^{j+1}) \mathbf{a} = 0$ for all $j \in \Z$.  Therefore, we have
\begin{lemma} \label{lem:precomputation}
\begin{align*}
U(m,n) \cong \{\mathbf{a} \in (\Q / \Z)^{2n}: \mathbf{T}^m \mathbf{a} = \mathbf{a} \text{ and } \mathbf{e}_{2n} \cdot (\mathbf{T}^{j-1} + \mathbf{T}^{j+1}) \mathbf{a} = 0 \text{ for } j \in \Z \}.
\end{align*}
\end{lemma}

\subsection{An Explicit Basis for $\c U_0$.} \label{subsec:algebraiccomputation}

A key insight in the remaining computation is to consider the two conditions $\mathbf{T}^m \mathbf{a} = \mathbf{a}$ and $\mathbf{e}_{2n} \cdot (\mathbf{T}^{j-1} + \mathbf{T}^{j+1}) \mathbf{a} = 0$ separately.  We denote
\begin{align*}
	M_1 &= \{\mathbf{a} \in (\Q / \Z)^{2n}: \mathbf{e}_{2n} \cdot (\mathbf{T}^{j-1} + \mathbf{T}^{j+1}) \mathbf{a} = 0 \text{ for } j \in \Z\}, \\
	M_2 &= \{\mathbf{a} \in (\Q / \Z)^{2n}: \mathbf{T}^m \mathbf{a} = \mathbf{a} \}.
\end{align*}
Observe that $U(m,n) \cong M_1 \cap M_2$.  

\begin{remark} \label{rem:CLFcoveringspace}
Here is a geometric interpretation of $M_1$ and $M_2$.  Note that the universal cover of the cylinder or M\"obius band is an infinite strip, and $\CLF(m,n)$ is covered by a corresponding graph $\CLF(\infty,n)$ with vertex set $\Z \times \{0,\dots,n\}$.  The condition defining $M_1$ says that harmonic continuation with initial values $\mathbf{a}$ defines a harmonic function on $\CLF(\infty,n)$ with $u = 0$ and $L_{\std} u = 0$ on the boundary.  The condition defining $M_2$ says that $u(j,k)$ is periodic in $j$, and hence $u$ corresponds to a harmonic function on $\CLF(m,n)$.
\end{remark}

We will compute $M_1$ first, using two auxiliary lemmas.  In the following, $\Z[4E^{-1}]$ will denote the sub-ring of $M_{n \times n}(\Z)$ generated by $4E^{-1}$ and $\Z[\mathbf{T}, \mathbf{T}^{-1}]$ will denote the sub-ring of $M_{2n \times 2n}(\Z)$ generated by $\mathbf{T}$ and $\mathbf{T}^{-1}$.  We use $(I, 0) \cdot \Z[\mathbf{T}, \mathbf{T}^{-1}]$ to denote the $\Z$-submodule of $M_{n \times 2n}(\Z)$ consisting of matrices of the form $(I, 0) \cdot \mathbf{S}$, where $\mathbf{S} \in \Z[\mathbf{T}, \mathbf{T}^{-1}]$ and $(I, 0) \in M_{n \times 2n}(\Z)$ is written as a matrix with two $n \times n$ blocks.  The notation $\Z[4E^{-1}] \cdot (I,0)$ is to be interpreted similarly.

\begin{lemma} \label{lem:matrixalgebra1}
We have an equality of $\Z$-modules
\[
(I, 0) \cdot \Z[\mathbf{T}, \mathbf{T}^{-1}] = \Z[4E^{-1}] \cdot (I, 0) + \Z[4E^{-1}] \cdot (0, I).
\]
\end{lemma}

\begin{proof}
The inclusion $\subseteq$ is straightforward since $\mathbf{T}$ and $\mathbf{T}^{-1}$ are block $2 \times 2$ matrices with block entries in $\Z[4E^{-1}]$.  To prove the opposite inclusion, first observe that
\[
\mathbf{T}^{-1} = \begin{pmatrix} 0 & I \\ -I & 4E^{-1} \end{pmatrix},
\]
Then note that
\begin{align*}
(I, 0) = (I, 0) \mathbf{I} \in (I, 0) \cdot \Z[\mathbf{T}, \mathbf{T}^{-1}]
(0, I) = (I, 0) \mathbf{T}^{-1} \in (I, 0) \cdot \Z[\mathbf{T}, \mathbf{T}^{-1}].
\end{align*}
Moreover,
\[
\mathbf{T} + \mathbf{T}^{-1} = \begin{pmatrix} 4E^{-1} & 0 \\ 0 & 4E^{-1} \end{pmatrix} = 4 \mathbf{F},
\]
where $\mathbf{F} := \text{diag}(E^{-1}, E^{-1})$.  Thus, $4 \mathbf{F} \in \Z[\mathbf{T}, \mathbf{T}^{-1}]$.  This implies
\begin{align*}
(4^jE^{-j}, 0) &= (I, 0) 4^j \mathbf{F}^j \in (I, 0) \cdot \Z[\mathbf{T}, \mathbf{T}^{-1}] \\
(0, 4^jE^{-j}) &= (0, I) 4^j \mathbf{F}^j \in (I, 0) \cdot \Z[\mathbf{T}, \mathbf{T}^{-1}],
\end{align*}
and thus all of $\Z[4E^{-1}] \cdot (I, 0) + \Z[4E^{-1}] \cdot (0, I)$ is contained in $(I, 0) \cdot \Z[\mathbf{T}, \mathbf{T}^{-1}]$.
\end{proof}

\begin{lemma} \label{lem:matrixalgebra2}
The row vectors $\{e_n^t E^{-j}\}_{j=0}^{n-1}$ are a basis for $\Z^n$.
\end{lemma}

\begin{proof}
Because $E$ is invertible over $\Z$, it suffices to show that $\Z^n$ is spanned by $\{e_n^t E^{-j} E^{n-1}\}_{j=0}^{n-1} = \{e_n^t E^j\}_{j=0}^{n-1}$.  Let $W_k$ be the $\Z$-span of $e_n^t, e_n^t E, \dots, e_n^t E^{k-1}$.  We can show by induction on $k$ that $W_k$ includes the first $k$ vectors from the ordered basis
\[
e_n^1, e_1^t, e_{n-1}^t, e_2^t, e_{n-2}^t, e_3^t, \dots 
\]
The general procedure is clear from the first few steps:
\begin{itemize}
	\item We have $e_n^t \in W_1$ trivially.
	\item Because $e_n^t \in W_1$ and $e_n^t E = e_1^t$, we have $e_1^t \in W_2$.
	\item Next, because $e_1^t \in W_2$, we have $e_1^t E = e_n^t + e_{n-1}^t \in W_3$.  Moreover, $e_n^t \in W_1 \subseteq W_3$, so that $e_{n-1}^t \in W_3$.
	\item Next, because $e_{n-1}^t \in W_3$, we have $e_{n-1}^t E = e_1^t + e_2^t \in W_4$, which implies that $e_2^t \in W_4$.
\end{itemize}
At the last step of the induction, we obtain $W_n = \Z^n$ as desired.
\end{proof}

\begin{lemma} \label{lem:M1computation}
\[
M_1 \cong \left( \bigoplus_{j=1}^n \Z / 4^j \right)^{\oplus 2}
\]
\end{lemma}

\begin{proof}
Noting that $\mathbf{e}_{2n}^t \mathbf{T} = \mathbf{e}_n^t$, we have
\[
M_1 = \{\mathbf{a} \in (\Q / \Z)^{2n}: \mathbf{e}_n \cdot (\mathbf{T}^{j-1} + \mathbf{T}^{j+1}) \mathbf{a} = 0 \text{ for } j \in \Z\}.
\]
As in the proof of Lemma \ref{lem:matrixalgebra1}, we have $\mathbf{T} + \mathbf{T}^{-1} = 4 \mathbf{F}$, and hence,
\[
M_1 = \{\mathbf{a} \in (\Q / \Z)^{2n}: 4 \mathbf{e}_n \cdot \mathbf{T}^j \mathbf{F} \mathbf{a} = 0 \text{ for } j \in \Z\}.
\]
Since $\mathbf{F}$ is invertible, we can view it as a change of coordinates on $(\Q / \Z)^{2n}$ and replace $\mathbf{F} \mathbf{a}$ by $\mathbf{a}$, so that
\[
M_1 \cong \{\mathbf{a} \in (\Q / \Z)^{2n}: 4 \mathbf{e}_n^t \mathbf{T}^j \mathbf{a} = 0 \text{ for } j \in \Z\}.
\]
We can rewrite $\mathbf{e}_n^t \mathbf{T}^j$ as $e_n^t (I, 0) \mathbf{T}^j$, where $I$ and $0$ are $n \times n$ identity and zero matrices respectively as mentioned above.  Let $N \subseteq \Z^{2n}$ be the module of row vectors in $\Z^{2n}$ given by
\[
N = 4 e_n^t (I, 0) (\Z[\mathbf{T}, \mathbf{T}^{-1}]) = \{ 4 e_n^t (I, 0) \mathbf{S}: \mathbf{S} \in \Z[\mathbf{T}, \mathbf{T}^{-1}]\}.
\]
Then we have
\[
M_1 \cong \{\mathbf{a} \in (\Q / \Z)^{2n}: \mathbf{n}^t \mathbf{a} = 0 \text{ for all } \mathbf{n}^t \in N\}.
\]
The remainder of the proof will use Lemmas \ref{lem:matrixalgebra1} and \ref{lem:matrixalgebra2} to exhibit a convenient basis for $N$ from which the invariant factors decomposition of $M_1$ will be obvious.

From Lemma \ref{lem:matrixalgebra2}, we deduce that
\[
\{e_n^t E^{-j} (I, 0)\}_{j=0}^{n-1} \cup \{e_n^t E^{-j} (0, I)\}_{j=0}^{n-1} \text{ is a basis for } \Z^{2n},
\]
where vectors in $\Z^{2n}$ are viewed as row vectors.  Denote this new basis by $\{\mathbf{w}_1^t, \dots, \mathbf{w}_{2n}^t\}$.  Meanwhile, substituting the result of Lemma \ref{lem:matrixalgebra1} into the definition of $N$ shows that
\[
N \text{ is spanned by } \{4e_n^t (4E^{-1})^j(I, 0)\}_{j=0}^{n-1} \cup \{4e_n^t (4E^{-1})^j (0, I)\}_{j=0}^{n-1},
\]
These vectors are scalar multiples of the basis vectors for $\Z^{2n}$ given in the previous equation, hence independent, and thus
\[
\{4 \mathbf{w}_1^t, \dots, 4^n \mathbf{w}_n^t, 4 \mathbf{w}_{n+1}^t, \dots, 4 \mathbf{w}_{2n}^t\}  \text{ is a basis for } N.
\]
Let $\mathbf{S}: \Z^{2n} \to \Z^{2n}$ be the change of basis matrix such that $\mathbf{w}_j^t \mathbf{S} = \mathbf{e}_j^t$.  Then changing coordinates by $\mathbf{S}$ on $(\Q / \Z)^{2n}$ yields
\begin{align*}
M_1 &\cong \{\mathbf{a} \in (\Q / \Z)^{2n} \colon \mathbf{n}^t \mathbf{a} = 0 \text{ for all } \mathbf{n}^t \in N\} \\
&\cong \{\mathbf{a} \in (\Q / \Z)^{2n} \colon \mathbf{n}^t \mathbf{S} \mathbf{a} = 0 \text{ for all } \mathbf{n}^t \in N\} \\
&=\{\mathbf{a} \in (\Q / \Z)^{2n} \colon 4 \mathbf{e}_1^t \mathbf{a} = 0, \dots, 4^n \mathbf{e}_n^t \mathbf{a} = 0, 4 \mathbf{e}_{n+1}^t \mathbf{a}, \dots, 4^n \mathbf{e}_{2n} \mathbf{a} \} \\
&\cong \left( \bigoplus_{j=1}^n (\Z / 4^j) \right)^{\oplus 2} \qedhere
\end{align*}
\end{proof}

Having computed $M_1$, we now turn to $M_2$.  Although $M_2$ itself is difficult to compute, we now know that $M_1$ is a $2$-torsion module.  Thus, we only have to compute the $2$-torsion submodule of $M_2$, which greatly simplifies matters.  Let $\Z[1/2]$ denote sub-ring of $\Q$ generated by $\Z$ and $1/2$, viewed as a $\Z$-module, or equivalently the $\Z$-module of rational numbers whose denominators are powers of $2$.  Let
\[
M_2' = M_2 \cap (\Z[1/2] / \Z)^{2n} = \{\mathbf{a} \in (\Z[1/2] / \Z)^{2n}: (\mathbf{T}^m - \mathbf{I})\mathbf{a} = \mathbf{0} \}.
\]
Since $M_1 \subseteq (\Z[1/2] / \Z)^{2n}$, we know that
\[
\c U_0(\CLF(m,n), \Q / \Z) \cong M_1 \cap M_2 = M_1 \cap M_2'.
\]
We will determine the $2$-torsion properties of $\mathbf{T}^m - \mathbf{I}$ by finding an accurate enough $2$-adic expansion of it.

\begin{lemma} \label{lem:periodicity}
Suppose $m = r2^s$ with $r$ odd.  Then
\[
M_2' \cong \begin{cases} (\Z/2)^n, & m \text{ odd} \\ (\Z/2)^{2n}, & m \equiv 2 \text{ mod } 4 \\ (\Z/2^{s+1})^{2n}, & m \equiv 0 \text{ mod } 4. \end{cases}
\]
\end{lemma}

\begin{proof}
First, consider the case where $m$ is not divisible by $4$, or equivalently $s \leq 2$.  Note that
\[
\mathbf{T} = \begin{pmatrix} 0 & -I \\ I & 0 \end{pmatrix} \text{ mod } 4.
\]
Hence, when $m = 1$ mod $4$,
\[
\mathbf{T}^m - \mathbf{I} = \begin{pmatrix} -I & -I \\ I & -I \end{pmatrix} \text{ mod } 4.
\]
The kernel of this map on $(\Z[1/2] / \Z)^{2n}$ is therefore isomorphic to $(\Z / 2)^n$.  The case where $m = 3$ mod $4$ is similar.  When $m = 2$ mod $4$, then $\mathbf{T}^m - \mathbf{I} = -2 \mathbf{I}$ mod $4$, so we get $M_2' \cong (\Z/2)^{2n}$.

To handle the case where $m = 0$ mod $4$, we compute by hand that
\[
\mathbf{T}^4 = \begin{pmatrix} -I & -4E^{-1} \\ 4E^{-1} & -I \end{pmatrix}^2 = \begin{pmatrix} I & 8E^{-1} \\ -8E^{-1} & I \end{pmatrix} \text{ mod } 16.
\]
From here, one can verify by induction that
\[
\mathbf{T}^{2^s} = \mathbf{I} + 2^{s+1} \begin{pmatrix} 0 & E^{-1} \\ -E^{-1} & 0 \end{pmatrix} \text{ mod } 2^{s+2} \text{ for } s \geq 2.
\]
Hence, if $r$ is odd, then using binomial expansion, we obtain
\[
\mathbf{T}^{r2^s} = \mathbf{I} + r2^{s+1} \begin{pmatrix} 0 & E^{-1} \\ -E^{-1} & 0 \end{pmatrix} = \mathbf{I} + 2^{s+1} \begin{pmatrix} 0 & E^{-1} \\ -E^{-1} & 0 \end{pmatrix} \text{ mod } 2^{s+2} \text{ for } s \geq 2.
\]
Since $E^{-1}$ is invertible, this implies that the kernel of $\mathbf{T}^{r2^s} - \mathbf{I}$ over $\Z[1/2] / \Z$ is isomorphic to $(\Z / 2^{s+1})^{2n}$.
\end{proof}

\begin{proof}[Proof of Theorem \ref{thm:CLF}]
Recall that $U(m,n) \cong M_1 \cap M_2' \subseteq (\Q / \Z)^{2n}$.  Note that $(\Q / \Z)^{2n}$ has a unique submodule isomorphic to $(\Z / 2)^{2n}$, so we can regard $(\Z / 2)^{2n} \subseteq (\Q / \Z)^{2n}$.  In the case where $m$ is not divisible by $4$, we have
\[
M_2' \subseteq (\Z / 2)^{2n} \subseteq M_1,
\]
and therefore, $M_1 \cap M_2' = M_2'$, and this yields the asserted formula in Theorem \ref{thm:CLF}.  Now suppose that $m$ is divisible by $4$, and $m = r2^s$, where $r$ is odd.  Then $M_2'$ is the unique submodule of $(\Q / \Z)^{2n}$ isomorphic to $(\Z / 2^{s+1})^{2n}$, while $M_1 \cong \left( \bigoplus_{j=1}^n \Z / 4^j \right)^{\oplus 2}$.  Thus, the only possibility is that
\[
M_1 \cap M_2' \cong \left( \bigoplus_{j=1}^n \Z / \gcd(4^j,2^{s+1}) \right)^{\oplus 2} = \left( \bigoplus_{j=1}^n \Z / \gcd(4^j,2m) \right)^{\oplus 2}. \qedhere
\]
\end{proof}

Our proof technique in this section allows us to analyze other algebraic properties of $U(m,n)$.  For instance, we have the following lemma, which we will use in the next section:  Let $U_1 = U_1(m,n)$ be the submodule of $\c U_0(\CLF(m,n),L_{\std}, \Q / \Z)$ consisting of functions that vanish on the vertices $(j,0)$, and let $U_2$ be the submodule of functions vanishing on the vertices $(j,1)$.  Then we have

\begin{lemma} \label{lem:CLFdirectsum}
If $m$ is even, then
\[
U(m,n) = U_1(m,n) \oplus U_2(m,n),
\]
where
\[
U_1(m,n) \cong U_2(m,n) \cong
\begin{cases} (\Z / 2)^n, & m \equiv 2 \text{ mod } 4 \\ \bigoplus_{j=1}^n \Z / \gcd(2m,4^j), & m \equiv 0 \text{ mod } 4. \end{cases}
\]
\end{lemma}

\begin{proof}
Recall that we expressed a function $u \in \c U_0(\CLF(m,n),L_{\std}, \Q / \Z)$ in terms of the two vectors $a_0$ and $a_1$, representing its values on the first two columns of vertices.  The submodules $U_1$ and $U_2$ correspond to the conditions $a_0 = 0$ and $a_1 = 0$ respectively.  From the proof of Lemma \ref{lem:M1computation}, we have
\[
M_1 = \left\{\mathbf{a} = \begin{pmatrix} a_1 \\ a_0 \end{pmatrix} \in (\Q / \Z)^{2n}: \mathbf{n}^t \cdot \mathbf{F} \mathbf{a} = 0 \text{ for } \mathbf{n}^t \in N \right\}.
\]
Then we used the change of coordinates $\mathbf{S}$ to write the basis for $N$ in a simpler form.  The proof thus showed that
\[
\mathbf{F}^{-1} \mathbf{S}^{-1} M_1 = \{\mathbf{a} \in (\Q / \Z)^{2n} \colon 4 \mathbf{e}_1^t \mathbf{a} = 0, \dots, 4^n \mathbf{e}_n^t \mathbf{a} = 0, 4 \mathbf{e}_{n+1}^t \mathbf{a}, \dots, 4^n \mathbf{e}_{2n} \mathbf{a} \}.
\]
The latter module clearly decomposes as the direct sum of the submodule where $a_1 = 0$ and the submodule where $a_0 = 0$.  However, the changes of coordinates $\mathbf{F}$ and $\mathbf{S}$ both respect the decomposition of $(\Q / \Z)^{2n}$ into $(\Q / \Z)^n \times 0^n$ and $0^n \times (\Q / \Z)^n$.  Thus, $M_1$ has the same direct sum decomposition.  If $m$ is even, then by Lemma \ref{lem:periodicity} we know that for some $k$, $M_2'$ is the unique submodule of $(\Q / \Z)^{2n}$ isomorphic to $(\Z / 2^k)^{2n}$ which is invariant under change of coordinates.  Hence,
\[
M_1 \cap M_2' = [M_1 \cap ((\Z/2^k)^n \times 0^n)] \oplus [M_1 \cap (0^n \times (\Z / 2^k)^n)].
\]
In other words, the submodule of $\mathbf{a} \in (\Q / \Z)^{2n}$ corresponding to $\c U_0$ breaks up into a direct sum of vectors with $a_0 = 0$ and vectors with $a_1 = 0$.  The same argument as in the proof of Theorem \ref{thm:CLF} shows that each summand is isomorphic to $(\Z / 2)^n$ when $m \equiv 2$ mod $4$ and otherwise, it is $\bigoplus_{j=1}^n \Z  / \gcd(4^j,2m)$.
\end{proof}

\subsection{Chain-Link Fence Variants} \label{subsec:CLFvariants}

The family $\CLF(m,n)$ is a variant of the family of $\partial$-graphs described in \cite{LPcyl}.   We let $\CLF'(m,n)$ be the graph from \cite{LPcyl} described as follows:  The vertex set is
\[
V(\CLF'(m,n)) = \{(x,y) \in \Z / 2m \times \{0,\dots,n+1\} \colon x + y \text{ is even} \}.
\]
In the condition ``$x + y$ is even,'' we are implicitly reducing $x$ and $y$ mod $2$ using the canonical maps $\Z / 2m \to \Z / 2$ and $\Z \to \Z / 2$.  The boundary vertices are
\[
\partial V(\CLF'(m,n)) = V(\CLF'(m,n)) \cap \Z / 2m \times \{0,n+1\}.
\]
The edges are given by $(x, y) \sim (x + 1, y \pm 1)$ whenever $y$ and $y \pm 1$ are both in $\{0,\dots,n\}$.  See Figure \ref{fig:CLFprime}.

\newcommand\makeCLFprime[2] { 
	\pgfmathtruncatemacro{\nn}{#2 + 1}
	\pgfmathtruncatemacro{\mm}{2 * #1}
	
	\draw[blue,dashed,thick] (0,0) to (0,\nn);
	\draw[blue,dashed,thick] (\mm,0) to (\mm,\nn);
	
	\foreach \j in {0,...,\nn} {
		\foreach \i in {0,...,\mm} {
			\pgfmathtruncatemacro{\t}{(-1)^(\i + \j)}
			\ifthenelse{
				\t = 1
			}{
				\ifthenelse{
					\j = 0 \OR \j = \nn
				}{
					\node[bd] (\i x\j) at (\i,\j) {};
				}{
					\node[int] (\i x\j) at (\i,\j) {};
				}
			}{}
		}
	}
	\foreach \j in {0,...,\nn} {
		\foreach \i in {1,...,\mm} {
			\pgfmathtruncatemacro{\t}{(-1)^(\i + \j)}
			\pgfmathtruncatemacro{\A}{\i - 1}
			\pgfmathtruncatemacro{\B}{\j + 1}
			\pgfmathtruncatemacro{\C}{\j - 1}
			\ifthenelse{
				\t = 1
			}{
				\ifthenelse{ \j = \nn }{}{ \draw (\i x\j) to (\A x\B); }
				\ifthenelse{ \j = 0 }{}{ \draw (\i x\j) to (\A x\C); }
			}{}
		}
	}
}

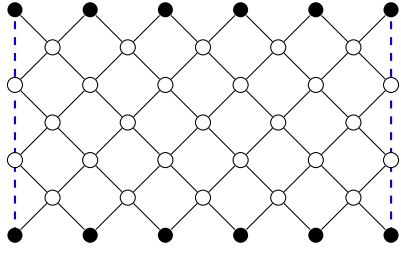
\begin{figure}
	\begin{center}
	
		\begin{tikzpicture}[scale=0.5]
			\makeCLFprime{5}{5}
		\end{tikzpicture}
	
		\caption{The $\partial$-graph $\CLF'(5,5)$.} \label{fig:CLFprime}
	
	\end{center}
\end{figure}

Observe that there is a $\partial$-graph isomorphism $\CLF(2m,n) \to \CLF'(m,2n)$ given by
\[
(j,k) \mapsto \begin{cases} (j,k), & i \text{ is even,} \\ (j,n-k), & i \text{ is odd.} \end{cases}
\]
Thus, we have already computed $\c U_0(\CLF'(m,n), L_{\std}, \Q / \Z)$ for even values of $n$ in Theorem \ref{thm:CLF}.  We will show that
\begin{theorem} \label{thm:CLF2}
Denote $U'(m,n) = \c U_0(\CLF'(m,n), L_{\std}, \Q/\Z)$.  Then
\[
U'(m,n) \cong \begin{cases} (\Z/2)^n, & m \text{ odd} \\ \bigoplus_{j=1}^{\lceil n/2 \rceil} \Z / \gcd(4^j,m) \oplus \bigoplus_{j=1}^{\lfloor n/2 \rfloor} \Z / \gcd(4^j,m), & m \text{ even.} \end{cases}
\]
\end{theorem}

The reader should verify that this agrees with Theorem \ref{thm:CLF} when $n$ is even.  We will deduce the odd case directly from Lemma \ref{lem:CLFdirectsum} using elementary reasoning with subgraphs.

There is also a canonical inclusion $f_{m,n}: \CLF'(m,n) \to \CLF'(m, n+1)$ given by mapping a vertex in $\CLF'(m,n)$ to the vertex in $\CLF'(m,n+1)$ with the same coordinates.  Thus, we can think of $\CLF'(m,n+1)$ as being obtained from $\CLF'(m,n)$ by adding another row of vertices at the top and changing the previous top row to interior vertices.  Next, if $u \in \c U_0(\CLF'(m,n), L_{\std}, \Q/\Z)$, then define $(f_{m,n})_* u$ on $\CLF'(m,n+1)$ by extending $u$ to be zero on the top row (or row $n + 2$) of vertices in $\CLF'(m,n+1)$.  Then

\begin{lemma} \label{lem:CLFinclusion}
The map $(f_{m,n})_*$ defines an injection $U'(m,n) \to  U'(m,n+1)$.  The image of consists of functions $v$ which vanish on row $n + 1$ in $\CLF'(m,n+1)$.  Moreover, if $v \in U'(m,n+1)$ vanishes on one vertex in row $n + 1$, then it vanishes on all vertices in row $n + 1$.
\end{lemma}

\begin{proof}
We must verify that $v := (f_{m,n})_* u$ is actually in $\c U_0(\CLF'(m,n+1), L_{\std}, \Q/\Z)$.  By construction $v = 0$ on the boundary rows $0$ and $n+2$ in $\CLF'(m,n+1)$.  We also have $L_{\std} v = L_{\std} u = 0$ on rows $0$ through $n$.  Because $v$ is zero on rows $n + 1$ and $n + 2$ we have $L_{\std} v = L_{\std} u = 0$ on row $n + 1$.  And finally, $v$ being zero on rows $n + 1$ and $n+2$ implies that the Laplacian is zero on row $n + 2$.

The injectivity of $(f_{m,n})_*$ is obvious, and clearly the image functions all vanish on row $n$.  Conversely, suppose $v \in \c U_0(\CLF'(m,n+1), L_{\std}, \Q/\Z)$ vanishes on the $n$th row, and let $u$ be the restriction to $\CLF'(m,n)$.  Since $v$ vanishes on rows $n + 1$ and $n + 2$, the edges between these rows make no contribution to $L_{\std} v$, and hence $L_{\std} u = L_{\std} v = 0$ on the $n$th row.  This shows $u \in \c U_0(\CLF'(m,n), L_{\std}, \Q/\Z)$ as desired.

For the final claim, suppose $v \in U'(m,n+1)$ vanishes on a vertex $(i,n+1)$ in row $n + 1$.  Then the conditions $v(i+1,n+2) = 0$ and $L_{\std} v(i+1,n+2) = 0$ force $v(i+2,n + 1) = 0$.  Similarly, $v(i+2,n+1) = 0$ implies $v(i+4,n+1) = 0$ and so on, so that $v$ vanishes on all of row $n + 1$.  (Recall there are no vertices at coordinates $(i+1,n+1)$, $(i+3,n+1)$, \dots)
\end{proof}

\begin{figure}
	\begin{center}
	
		\begin{tikzpicture}[scale=0.4]
			\begin{scope}
				\makeCLFprime{3}{2}
			\end{scope}
			
			\draw[->] (7,2) to node[auto] {$f_{3,2}$} (9,2);
			
			\begin{scope}[shift={(10,0)}]
				\makeCLFprime{3}{3}
			\end{scope}
			
			\draw[->] (17,2) to node[auto] {$f_{3,3}$} (19,2);
			
			\begin{scope}[shift={(20,0)}]
				\makeCLFprime{3}{4}
			\end{scope}		
		\end{tikzpicture}
	
		\caption{Inclusion maps $\CLF'(3,2) \to \CLF'(3,3) \to \CLF'(3,4)$.} \label{fig:CLFprimeinclusion}
	
	\end{center}
\end{figure}

As in Lemma \ref{lem:CLFdirectsum}, define
\begin{align*}
U_1'(m,n) &= \{u \in U'(m,n) \colon u(0,j) = 0 \text{ for all } j\} \\
U_2'(m,n) &= \{u \in U'(m,n) \colon u(1,j) = 0 \text{ for all } j\} .
\end{align*}

\begin{lemma} \label{lem:CLFdirectsum2}
We have $U'(m,n) = U_1'(m,n) \oplus U_2'(m,n)$.
\end{lemma}

\begin{proof}
The case for even $n$ follows from Lemma \ref{lem:CLFdirectsum}.  Suppose $n$ is odd.  To show that $U_1'(m,n) \cap U_2'(m,n) = 0$, note that $(f_{m,n})_*$ maps $U_j'(m,n)$ into $U_j'(m,n+1)$.  Since $U_1'(m,n+1) \cap U_2'(m,n+1) = 0$ by the even case and since $(f_{m,n})_*$ is injective, we deduce that $U_1'(m,n) \cap U_2'(m,n) = 0$.

To show that $U'(m,n) = U_1'(m,n) + U_2'(m,n)$, let $u \in U'(m,n)$ and let $v = (f_{m,n})_* u$.  From the even case, we know $v = v_1 + v_2$, where $v_1 \in U_1'(m,n+1)$ and $v_2 \in U_2'(m,n+1)$.  Now $v_1$ vanishes on $(0,n+1)$ by definition of $U_1'(m,n+1)$; then the last claim of Lemma \ref{lem:CLFinclusion} implies that $v_1$ vanishes on row $n+1$.  Since $v$ vanishes on row $n + 1$ by assumption, we know $v_2 = v - v_1$ also vanishes on row $n + 1$.  Therefore, $v_1$ and $v_2$ are in the image of $(f_{m,n})_*$, that is, $v_1 = (f_{m,n})_* u_1$ and $v_2 = (f_{m,n})_* u_2$ for some $u_1$, $u_2 \in U'(m,n)$.  Then clearly $u = u_1 + u_2$ and $u_1 \in U_1'(m,n)$ and $u_2 \in U_2'(m,n)$.
\end{proof}

\begin{proof}[Proof of Theorem \ref{thm:CLF2}]  The case where $n$ is even has already been handled in Theorem \ref{thm:CLF}.  Now suppose $n$ is odd.  Note that $(f_{m,n})_*$ gives an injection $U_1'(m,n) \to U_1'(m,n+1)$.  But in fact, this is an isomorphism because any function $v \in U_1'(m,n+1)$ vanishes on $(0,n+1)$, hence vanishes on all of row $n + 1$, hence comes from a function $u$ in $U_1'(m,n)$.  A similar argument shows that $U_2'(m,n) \cong U_2'(m,n-1)$.  Therefore,
\begin{align*}
U'(m,n) &= U_1'(m,n) \oplus U_2'(m,n)  \\
&= U_1'(m,n+1) \oplus U_2'(m,n-1)  \\
&= U_1(2m, \tfrac{n+1}{2}) \oplus U_2(2m, \tfrac{n-1}{2}),
\end{align*}
and the proof is completed by applying Lemma \ref{lem:CLFdirectsum}.
\end{proof}

\section{The Categories of $\partial$-Graphs and $R$-Networks} \label{sec:transformations}

\subsection{Motivation}

The example of the $\CLF$ networks already illustrated the usefulness of covering spaces (Remark \ref{rem:CLFcoveringspace}) and subgraphs (Lemmas \ref{lem:CLFinclusion} and \ref{lem:CLFdirectsum2}, proof of Theorem \ref{thm:CLF2}).  We will now describe general morphisms of $R$-networks, adapting the ideas of \cite{Urakawa,bakerNor1,Treumann} to $\partial$-graphs, as well as giving applications to spanning tree counts and eigenvectors.

Harmonic morphisms of graphs were defined by Urakawa \cite[Definition 2.2]{Urakawa}.  Baker and Norine showed that a harmonic morphism $\phi: G' \to G$ defines a map $\phi_*$ from the critical group of $G'$ to that of $G$ as well as a map $\phi^*$ from the critical group of $G$ to that of $G'$ \cite[\S 2.3]{bakerNor1}.  In other words, the critical group (a.k.a.\ sandpile group or Jacobian) can be viewed either as a covariant or as a contravariant functor from the category of graphs and harmonic morphisms to the category of abelian groups.  Special cases of Baker and Norine's construction were defined earlier (2002) in the undergraduate thesis of Treumann \cite{Treumann}.

We will construct categories of $\partial$-graphs and $R$-networks, and show that $\Upsilon(G,L)$ and $\c U_0(G,L,M)$ for fixed $M$ are covariant functors from $R$-networks to $R$-modules, and $\c U(G,L,M)$ is a contravariant functor.  It makes sense for $\c U_0(G,L,M)$ to be covariant and $\c U(G,L,M)$ to be contravariant with respect to $(G,L)$ because $\c U(G,L,M) \cong \Hom(\Upsilon(G,L),M)$ by Lemma \ref{lem:hom} and $\c U_0(G,L,M) \cong \Tor_1(\Upsilon(G,L), M)$ for non-degenerate networks by Proposition \ref{prop:tor}.

\subsection{The Category of $\partial$-Graphs} \label{subsec:categorydgraphs}

Before defining morphisms of $R$-networks and verifying the functorial properties, we must record and explain the purely combinatorial definition of a $\partial$-graph morphism.  For a vertex $x$ in a $\partial$-graph $G$, recall that we use the notation $\mathcal{E}(x) = \{e \in E(G): e_+ = x\}$ for the set of oriented edges exiting $x$.

\begin{definition} \label{def:dgraphmorphism}
A {\bf $\partial$-graph morphism} $f\colon G_1 \to G_2$ is a map ${f\colon V_1 \sqcup E_1 \to V_2 \sqcup E_2}$ such that
\begin{enumerate}
	\item $f$ maps vertices to vertices.
	\item $f$ maps interior vertices to interior vertices.
	\item If $f(e)$ is an oriented edge, then $f(e_+) = (f(e))_+$ and $f(e_-) = (f(e_))_-$ and $f(\overline{e}) = \overline{f(e)}$.
	\item If $f(e)$ is a vertex, then $f(\overline{e}) = f(e)$ and $f(e_{\pm})= f(e)$.
	\item For every $x \in V_1^\circ$, the restricted map $\mathcal{E}(x) \cap f^{-1}(E_2) \xrightarrow{f} \mathcal{E}(f(x))$ has constant fiber size.  In other words, it is $n$-to-$1$ for some integer $n \geq 0$ (depending on $x$).
\end{enumerate}
\end{definition}

The statement of condition (5) implicitly uses the fact that $f$ restricts to a map $\mathcal{E}(x) \cap f^{-1}(E_2) \to \mathcal{E}(f(x))$, which follows from (3).  The integer $n$ associated to a vertex $x \in V_1^\circ$ in condition (5) will be called the {\bf degree of $f$ at $x$} and denoted by $\deg(f,x)$.  Note that (5) implies
\[
\forall x \in V_1^\circ, \forall e \in \mathcal{E}(f(x)), \quad |\mathcal{E}(x) \cap f^{-1}(e)| = \deg(f,x).
\]
It will also be convenient to extend the definition of $\deg(f,x)$ to $x \in \partial V_1$ by setting
\[
\deg(f,x) = \max_{e \in \mathcal{E}(f(x))} |\mathcal{E}(x) \cap f^{-1}(e)|.
\]

Conditions (1), (3), and (4) say that $f$ is a graph homomorphism except that it allows an edge to be collapsed to a vertex as in Figure \ref{fig:projectionGHMexample}.  In other words, if we view $G_1$ and $G_2$ as cell complexes, then $f$ is a continuous cellular map.

Condition (5) says that $f$ restricts to an $n$-fold covering of the neighborhood $\mathcal{E}(x)$ of $x$ onto the neighborhood $\mathcal{E}(f(x))$ of $f(x)$, after ignoring collapsed edges (note that ignoring collapsed edges is exactly the effect of the taking the intersection of $\mathcal{E}(x)$ with $f^{-1}(E_2)$ in (5)).  For example, see Figure \ref{fig:modifiedprojectionGHMexample}.  The next lemma is the first step in establishing our functoriality properties.

\begin{figure}
  \centering
  \begin{tikzcd}%
    V_1^{\circ} \arrow{r} &  V_2^{\circ}
        \\%
    \partial V_1 \arrow{r} \arrow{ru}&  \partial V_2\\%
    E_1 \arrow{r}\arrow{ru}\arrow{ruu} &  E_2\\%
  \end{tikzcd}
  \vspace{-5mm}
  \caption{Where a harmonic morphism is allowed to map the sets $V^\circ$, $\partial V$, and $E$.}
  \label{fig:harmonicDfn}
\end{figure}
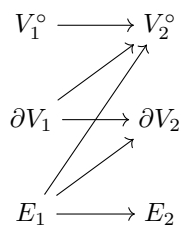

\begin{figure}
	\begin{center}
	\begin{tikzpicture}[scale=0.8]
		\begin{scope}
			\node[int] (0) at (0,0) [label=above:$0$] {};
			\node[bd] (1) at (0:1) [label=0:$1$] {};
			\node[bd] (2) at (60:1) [label=60:$2$] {};
			\node[bd] (3) at (120:1) [label=120:$3$] {};
			\node[bd] (4) at (180:1) [label=180:$1$] {};
			\node[bd] (5) at (240:1) [label=240:$2$] {};
			\node[bd] (6) at (300:1) [label=300:$3$] {};
			
			\draw (0) to (1); \draw (0) to (2); \draw (0) to (3);
			\draw (0) to (4); \draw (0) to (5); \draw (0) to (6);
		\end{scope}
		
		\draw[->] (2,0) -- (4,0);
		
		\begin{scope}[shift={(6,0)}]
			\node[int] (0) at (0,0) [label=above:$0$] {};
			\node[bd] (1) at (30:1) [label=30:$1$] {};
			\node[bd] (2) at (150:1) [label=150:$2$] {};
			\node[bd] (3) at (270:1) [label=270:$3$] {};
			
			\draw (0) to (1); \draw (0) to (2); \draw (0) to (3);
		\end{scope}
	\end{tikzpicture}
	
	\caption{A $\partial$-graph morphism.  The numbers show where each vertex is mapped.} \label{fig:branchedGHMexample}
	
	\end{center}
\end{figure}
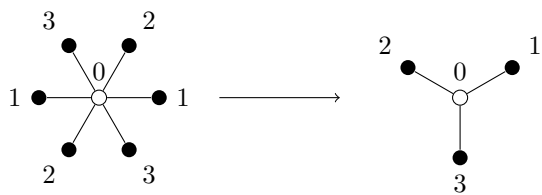

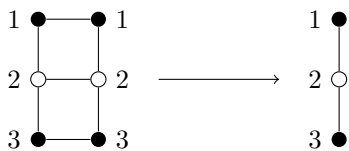
\begin{figure}
	\begin{center}
	\begin{tikzpicture}[scale=0.8]
		\begin{scope}
			\node[bd] (1A) at (0,1) [label=left:$1$] {};
			\node[int] (2A) at (0,0) [label=left:$2$] {};
			\node[bd] (3A) at (0,-1) [label=left:$3$] {};
			\node[bd] (1B) at (1,1) [label=right:$1$] {};
			\node[int] (2B) at (1,0) [label=right:$2$] {};
			\node[bd] (3B) at (1,-1) [label=right:$3$] {};
			
			\draw (1A) to (2A) to (3A);
			\draw (1B) to (2B) to (3B);
			\draw (1A) to (1B); \draw (2A) to (2B); \draw (3A) to (3B);	
		\end{scope}
		
		\draw[->] (2,0) -- (4,0);
		
		\begin{scope}[shift={(5,0)}]
			\node[bd] (1) at (0,1) [label=left:$1$] {};
			\node[int] (2) at (0,0) [label=left:$2$] {};
			\node[bd] (3) at (0,-1) [label=left:$3$] {};
			
			\draw (1) to (2) to (3);
		\end{scope}
	\end{tikzpicture}
	
	\caption{A $\partial$-graph morphism.  The horizontal edges are collapsed and mapped to the vertices $1$, $2$, $3$ on the right.  The vertical edges on the left graph are mapped to the vertical edges on the right.} \label{fig:projectionGHMexample}
	
	\end{center}
\end{figure}

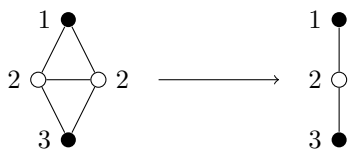
\begin{figure}
	\begin{center}
	\begin{tikzpicture}[scale=0.8]
		\begin{scope}
			\node[bd] (1) at (0.5,1) [label=left:$1$] {};
			\node[int] (2A) at (0,0) [label=left:$2$] {};
			\node[bd] (3) at (0.5,-1) [label=left:$3$] {};
			\node[int] (2B) at (1,0) [label=right:$2$] {};
			
			\draw (1) to (2A) to (3);
			\draw (1) to (2B) to (3);
			\draw (2A) to (2B);
		\end{scope}
		
		\draw[->] (2,0) -- (4,0);
		
		\begin{scope}[shift={(5,0)}]
			\node[bd] (1) at (0,1) [label=left:$1$] {};
			\node[int] (2) at (0,0) [label=left:$2$] {};
			\node[bd] (3) at (0,-1) [label=left:$3$] {};
			
			\draw (1) to (2) to (3);
		\end{scope}
	\end{tikzpicture}
	
	\caption{A $\partial$-graph morphism.  The horizontal edge is collapsed into the vertex $2$ on the right, whilte the slanted edges on the left map to the vertical edges on the right.} \label{fig:modifiedprojectionGHMexample}
	
	\end{center}
\end{figure}

\begin{lemma} \label{lem:categorydegree}
$\partial$-graphs form a category.  Moreover, if $f \colon G_1 \to G_2$ and $g: G_2 \to G_3$ are $\partial$-graph morphisms, then
\begin{align*}
\deg(g \circ f, x) &= \deg(f,x) \deg(g,f(x)) \text{ for all } x \in V_1^\circ \\
\deg(g \circ f, x) &\leq \deg(f,x) \deg(g,f(x)) \text{ for all } x \in \partial V_1.
\end{align*}
\end{lemma}

\begin{proof}
To verify the category axioms, it suffices to show that if $f \colon G_1 \to G_2$ and $g: G_2 \to G_3$ are $\partial$-graph morphisms, then so is $g \circ f$.  Clearly, (1) and (2) are preserved by composition.  To check $g \circ f$ satisfies (3), note that if $g \circ f(e)$ is an oriented edge, then $f(e)$ must be an oriented edge by (1), and hence we can apply (3) to $f$ at the edge $e$ and (3) to $g$ at the edge $f(e)$.

To check (4), suppose $g \circ f(e)$ is a vertex.  If $f(e)$ is a vertex, then apply (4) to $f$.  If $f(e)$ is an edge, then apply (3) to $f$ and (4) to $g$.

To check (5), suppose $x \in V_1^\circ$ and $e \in \mathcal{E}(g \circ f(x))$.  Any element of $f^{-1}(g^{-1}(e))$ must be mapped into $\mathcal{E}(f(x))$ by $f$, and thus
\[
\mathcal{E}(x) \cap f^{-1}(g^{-1}(e)) = \sqcup_{e' \in \mathcal{E}(f(x)) \cap g^{-1}(e)} \mathcal{E}(x) \cap f^{-1}(e').
\]
Using the fact that $f(x) \in V_2^\circ$, we see that this is a disjoint union of $\deg(g,f(x))$ sets of size $\deg(f,x)$.  This implies that
\[
|\mathcal{E}(x) \cap f^{-1}(g^{-1}(e))| = \deg(f,x) \deg(g,f(x)) \text{ for all } e \in \mathcal{E}(g \circ f(x)),
\]
and hence $g \circ f$ satisfies (5) and hence is a $\partial$-graph morphism.

Moreover, the last computation showed that $\deg(g \circ f, x) = \deg(f,x) \deg(g,f(x))$ and a similar argument shows that $\deg(g \circ f, x) \leq \deg(f,x) \deg(g,f(x))$ for $x \in \partial V_1$.
\end{proof}

As in \cite{bakerNor1, Urakawa, Mercat}, we can think of $\partial$-graph morphisms as a discrete analogue of holomorphic maps between Riemann surfaces with boundary.  The first, perhaps trivial, analogy is that both $\partial$-graph morphisms and holomorphic functions are closed under composition.  Moreover, in the next section, we will show that if $f \colon G_1 \to G_2$ is a $\partial$-graph morphism and $u$ is harmonic on $G_2$, then $u \circ f$ is harmonic on $G_1$.

Just as with Riemann surfaces, the simplest type of $\partial$-graph morphism is a covering map, which completely preserves local structure.  In the discrete setting, we define covering maps as follows.  Note that this agrees with topological definition if we view $\partial$-graphs as a cell complexes and forget the distinction between interior and boundary vertices.

\begin{definition} \label{def:coveringmap}
A {\bf covering map} is a $\partial$-graph morphism $f \colon \tilde{G} \to G$ such that $f$ defines a surjection $V(\tilde{G}) \sqcup E(\tilde{G}) \to V(G) \sqcup E(G)$, $f$ maps interior vertices to interior vertices, $f$ maps boundary vertices to boundary vertices, $f$ maps edges to edges, and the restricted map $\mathcal{E}(x) \to \mathcal{E}(f(x))$ is a bijection for every $x \in \tilde{V}$.
\end{definition}

We have already seen a covering map in Remark \ref{rem:CLFcoveringspace}.  Moreover, the standard construction of the bipartite double cover for a graph easily adapts to $\partial$-graphs.  We shall say more about covering spaces in \S \ref{sec:openproblems}.

Another important type of morphism is the inclusion of sub-$\partial$-graphs.  We have already used sub-$\partial$-graphs in \S \ref{subsec:CLFvariants}.  Later, in \S \ref{sec:layering}, we will consider restricting harmonic functions to sub-$\partial$-graphs, and extending them from sub-$\partial$-graphs using a discrete analogue of harmonic continuation.  We now record the precise definition of a sub-$\partial$-graph for future use:

\begin{definition} \label{def:sub-d-graph}
Assume that $G_1$ and $G_2$ are $\partial$-graphs, such that $(V_1,E_1)$ is a subgraph of $(V_2,E_2)$.  Then we say $G_1$ is a {\bf sub-$\partial$-graph} of $G_2$ if the inclusion map $G_1 \to G_2$ is a $\partial$-graph morphism.  One can verify from Definition \ref{def:dgraphmorphism} that a subgrpah $G_1$ will be a sub-$\partial$-graph if and only if $x \in V_1^\circ$ implies that $x \in V_2^\circ$ and $\mathcal{E}_{G_1}(x) = \mathcal{E}_{G_2}(x)$.
\end{definition}

Like a holomorphic function, a $\partial$-graph morphism $f \colon G_1 \to G_2$ may exhibit {\em ramification} when a star $\mathcal{E}(x)$ in $G_1$ is an $n$-fold cover of a star $\{e: e_+ = f(x)\}$ in $G_2$ for $n > 1$. For example, see Figure \ref{fig:branchedGHMexample}.  This is a discrete model of the behavior of the map $z \mapsto z^n$ in a neighborhood of the origin in $\C$.  The formula $\deg(g \circ f, x) = \deg(f,x) \deg(g,f(x))$ also mimics the way that local degrees of holomorphic maps are multiplicative under composition.  The behavior of a $\partial$-graph morphism is unconstrained by condition (5) at the boundary, just as an analytic function on a Riemann surface need not be $n$-to-$1$ in the neighborhood of a boundary point.

Recall that for compact connected Riemann surfaces without boundary, every non-constant holomorphic map is surjective as a consequence of the open mapping theorem.  Now we will prove an analogous statement in the discrete case, which applies even to infinite $\partial$-graphs.  We will view graphs without boundary as the subclass of $\partial$-graphs with no boundary vertices.  Continuing the terminology of \cite{Urakawa,bakerNor1}, we will refer to $\partial$-graph morphisms of boundary graphs without boundary as {\bf harmonic morphisms}.  An example of such a morphism is shown in \cite[Figure 1]{bakerNor1}.  An alternative proof of the following Proposition can also be found in \cite[Lemmas 2.4 and 2.7]{bakerNor1}.

\begin{proposition} \label{prop:harmonicmorphismsurjective}
Let $f \colon G_1 \to G_2$ be a harmonic morphism of nonempty connected graphs without boundary.  Either $f$ maps $V_1 \sqcup E_1$ to a single vertex $x$ of $G_2$, or $f$ is a surjection $V_1 \sqcup E_1 \to V_2 \sqcup E_2$.
\end{proposition}

\begin{proof}
Let $A = \{x \in f(V_1): \deg(f,y) > 0 \text{ for some } y \in f^{-1}(x)\}$.  We claim that if $x \in A$, then all neighbors of $x$ are also in $A$.  Suppose $x \in A$ and $x'$ is joined to $x$ by an edge $e$.  Since $x \in A$, there exists $y \in f^{-1}(x)$ with $\deg(f,y) > 0$.  This implies that $y$ has some edge $\tilde{e}$ which maps to $e$, so $y$ has some neighbor $y'$ which maps to $x'$.  Since the edge $\tilde{e}$ incident to $y'$ maps to an edge in $G_2$, we must have $\deg(f,y') > 0$.  Therefore, $x' \in A$.

Since $G_2$ is connected, either $A = V_2$ or $A = \varnothing$.  In the first case, $f$ must be surjective onto $V_2$, and then by definition of $A$ and $\deg(f,y)$, we deduce that $f$ is surjective onto $E_2$.  In the second case, we have $\deg(f,y) = 0$ for all $y \in V_1$, which implies that all edges in $G_2$ are collapsed to vertices.  Then since $G_2$ is connected, $f$ must be constant.
\end{proof}

Though $\partial$-graph morphisms are much like holomorphic maps, the ability of $\partial$-graph morphisms to collapse an edge into a vertex seems to have no direct analogue in complex analysis.  In a neighborhood of a vertex $x$ where $\deg(f,x) = 1$ and some edges in $\mathcal{E}(x)$ are collapsed, a $\partial$-graph $f$ behaves more like an orthogonal projection (recall that that if $f \colon \R^m \to \R^n$ is an orthogonal projection and $u: \R^n \to \R$ is harmonic, then $u \circ f$ is harmonic).  A more precise analogy is between the projection maps associated to a product of Riemannian manifolds and the projection maps associated to a \emph{box product} of $\partial$-graphs (also known as the \emph{Cartesian product}), which is defined as follows:

Let $G_1$ and $G_2$ be $\partial$-graphs.  Then we define the {\bf box product} $G = G_1 \square G_2$ by
\[
V = V_1 \times V_2, \quad E = E_1 \times V_2 \cup V_1 \times E_2, \quad V^\circ = V_1^\circ \times V_2^\circ.
\]
Then if $(e,x) \in E_1 \times V_2$, we define $\overline{(e,x)} = (\overline{e},x)$, $(e,x)_+ = (e_+,x)$, and $(e,x)_- = (e_-,x)$, and make a similar definition for $(x,e) \in V_1 \times V_2$.  In particular, two vertices $(x,y)$ and $(x',y')$ are adjacent if $x = x'$ and $y \sim y'$ or if $x \sim x'$ and $y = y'$.  Then the obvious projection map $f_1: G \to G_1$ is a $\partial$-graph morphism.  Note that $f_1$ has degree $1$ at each vertex and collapses all edges in $V_1 \times E_2$ to vertices.

In general, the local behavior of a $\partial$-graph morphism at an interior vertex combines ramification and collapsing, that is, it combines the behavior of branched covering maps and projections.  What is unique about the discrete setting is that $f$ may behave like a holomorphic map $\R^2 \to \R^2$ near one vertex and behave like a projection map $\R^3 \to \R^2$ at another vertex; this cannot happen for manifolds because the dimension does not vary from point to point.

\subsection{The Category of $R$-Networks} \label{subsec:categoryRnetworks}

Recall that an $R$-network is given by a pair $(G,L)$, where the off-diagonal terms of $L$ are given by a weight function $w: E \to R$ and the diagonal terms are given by $d: V \to R$.

\begin{definition} \label{def:Rnetworkmorphism}
An $R$-network morphism $f \colon (G_1,L_1) \to (G_2,L_2)$ is given by a $\partial$-graph morphism $f \colon G_1 \to G_2$ satisfying
\[
w_1(e) = w_2(f(e)) \text{ for every } e \in f^{-1}(E_2)
\]
and
\[
d_1(x) = \deg(f,x) d_2(f(x)) \text{ for every } x \in V_1^\circ.
\]
\end{definition}

It is straightforward to verify that $R$-networks form a category, using the fact that $\partial$-graphs form a category and $\deg(g \circ f,x) = \deg(f,x) \deg(g,f(x))$ for an interior vertex $x$.  We denote this category by $R\text{\cat{-net}}$.

\begin{remark}
Note that the second condition of Definition \ref{def:Rnetworkmorphism} is trivially satisfied in the case where $d_j = 0$.  In particular, if $f \colon G_1 \to G_2$ is a $\partial$-graph morphism, then $f$ automatically defines a $\Z$-network morphism $(G_1,L_{\std,1}) \to (G_2, L_{\std,1})$, where $L_{\std}$ is the standard Laplacian with edge weights $1$.  Thus, $G \mapsto (G,L_{\std})$ is a functor from $\partial$-graphs to $\Z$-networks.
\end{remark}

The functoriality properties we will prove rely on the following observation.

\begin{lemma} \label{lem:laplaciandegree}
Let $f \colon (G_1,L_1) \to (G_2,L_2)$ be an $R$-network morphism.  Note that $f$ extends linearly to a map $RV_1 \to RV_2$.  If $x \in V_1^\circ$, then
\[
f(L_1 x) = \deg(f,x) L_2(f(x)).
\]
\end{lemma}

\begin{proof}
Note that
\[
f(L_1x) = d_1(x) f(x) + \sum_{e \in \mathcal{E}(x)} w_1(e)(f(x) - f(e_-)).
\]
By Definition \ref{def:Rnetworkmorphism}, $d_1(x) = \deg(f,x) d_2(x)$.  Moreover, by Definition \ref{def:dgraphmorphism}, each $e \in \mathcal{E}(x)$ will either map to $f(x)$ or to some $e' \in \mathcal{E}(f(x))$.  Thus, the sum over the edges becomes
\[
\sum_{e \in \mathcal{E}(x) \cap f^{-1}(f(x))} w_1(e)(f(x) - f(e_-)) + \sum_{e' \in \mathcal{E}(f(x))} \sum_{e \in \mathcal{E}(x) \cap f^{-1}(e')} w_1(e)(f(x) - f(e_-)).
\]
The first term vanishes since if $e$ is collapsed into $x$, then $f(e_-) = f(x)$.  In the second term, note $w_1(e) = w_2(f(e))$ by Definition \ref{def:Rnetworkmorphism} and the number of terms corresponding to each $e'$ is $\deg(f,x)$.  Hence,
\begin{align*}
\sum_{e \in \mathcal{E}(x)} w_1(e)(f(x) - f(e_-)) &= \sum_{e' \in \mathcal{E}(f(x))} \sum_{e \in \mathcal{E}(x) \cap f^{-1}(e')} w_2(e')(f(x) - f(e_-)) \\
&= \sum_{e' \in \mathcal{E}(f(x))} \deg(f,x) w_2(e')(f(x) - e_-').
\end{align*}
Therefore,
\begin{align*}
f(L_1x) &= \deg(f,x) d_2(f(x)) f(x) + \deg(f,x) \sum_{e' \in \mathcal{E}(f(x))}  w_2(e')(f(x) - e_-') \\
&= \deg(f,x) L_2(f(x)). \qedhere
\end{align*}
\end{proof}

Now we can establish the promised functoriality properties.  We remark that special cases of Lemma \ref{lem:upsilonfunctor} were proved in \cite{Treumann} and the case of Lemma \ref{lem:ufunctor} where $R = \Z$ and $L = L_{\std}$ was proved in \cite[Proposition 2.8]{bakerNor1}.

\begin{lemma} \label{lem:upsilonfunctor}
The map $(G,L) \mapsto \Upsilon(G,L)$ is a functor $R\text{\cat{-net}} \to R\text{\cat{-mod}}$, where the definition on morphisms is as follows:  If $f: (G_1,L_1) \to (G_2,L_2)$ is an $R$-network morphism, then $\Upsilon f$ is given by
\[
\Upsilon f \colon \Upsilon(G_1,L_1) \to \Upsilon(G_2,L_2) \colon  x + L(RV_1^\circ) \mapsto f(x) + L(RV_2^\circ).
\]
\end{lemma}

\begin{proof}
Let us verify that the map $\Upsilon f$ is well-defined.  Recall
\[
\Upsilon(G,L) = RV(G) / L(RV^\circ(G)).
\]
The map $f \colon (G_1,L_1) \to (G_2,L_2)$ defines a map $f \colon RV_1 \to RV_2$.  By Lemma \ref{lem:laplaciandegree}, $f$ maps $L_1(RV_1^\circ)$ into $L_2(RV_2^\circ)$.  This implies $f$ yields a well-defined map on the quotient.  Checking that $\Upsilon(g \circ f) = \Upsilon g \circ \Upsilon f$ is straightforward.
\end{proof}

\begin{lemma} \label{lem:ufunctor}
The map $(G,L), M \mapsto \c U(G,L,M)$ is a functor $R\text{\cat{-net}}^{\text{op}} \times R\text{\cat{-mod}} \to R\text{\cat{-mod}}$, where the definition on morphisms is given as follows:  If $f \colon (G_1,L_1) \to (G_2,L_2)$ and $\phi: M \to M'$ is an $R$-module morphism, then we have
\[
\c U(f,\phi) \colon \c U(G_2,L_2,M) \to \c U(G_1,L_1,M') \colon u \mapsto \phi \circ u \circ f.
\]
The isomorphism $\c U(G,L,M) \cong \Hom_R(\Upsilon(G,L),M)$ given by Lemma \ref{lem:hom} is natural in both variables.
\end{lemma}

\begin{proof}
To check that the map $\c U(f,\phi)$ actually maps into $\c U(G_1,L_1,M')$, let $u \in \c U(G_2,L_2,M)$.  If $x \in V_1^\circ$, then $f(x) \in V_2^\circ$ and by Lemma \ref{lem:laplaciandegree}, we have
\[
L_1(\phi \circ u \circ f)(x) = \phi \circ u \circ f(L_1x) = \deg(f,x) \phi \circ u(L_2 f(x)) = 0.
\]
Functoriality is straightforward to check, and the naturality of the isomorphism $\c U(G,L,M) \cong \Hom_R(\Upsilon(G,L),M)$ is checked easliy from the proof of Lemma \ref{lem:hom}.
\end{proof}

\begin{lemma} \label{lem:u0functor}
The map $(G,L), M \mapsto \c U_0(G,L,M)$ is a functor $R\text{\cat{-net}} \times R\text{\cat{-mod}} \to R\text{\cat{-mod}}$, where the definition on morphisms is given as follows:  If $f \colon (G_1,L_1) \to (G_2,L_2)$ and $\phi: M \to M'$, then we have
\[
[\c U_0(f,\phi) u](y) = \sum_{x \in f^{-1}(y) \cap V_1} \deg(f,x) \phi \circ u(x).
\]
The surjection $\c U_0(G,L,M) \to \Tor_1^R(\Upsilon(G,L), M)$ given by Proposition \ref{prop:tor} is natural in both variables.
\end{lemma}

\begin{proof}
As in the proof of Proposition \ref{prop:tor}, for an $R$-network $(G,L)$, we can view $RV^\circ \otimes M$ as the module of finitely-supported functions $u: V^\circ \to M$ by the identification
\[
u \leftrightarrow \sum_{x \in V^\circ} x \otimes u(x).
\]
Then
\[
\c U_0(G,L,M) = \ker(L \otimes \id: RV^\circ \otimes M \to RV \otimes M).
\]
By Lemma \ref{lem:laplaciandegree}, if $f \colon (G_1,L_1) \to (G_2,L_2)$, then the left diagram below commutes and hence the right diagram below commutes:
\[
\begin{tikzcd}
RV_1^\circ \arrow{r}{L_1} \arrow[swap]{d}{\deg(f,\cdot) f} & RV_1 \arrow{d}{f} & RV_1^\circ \otimes M \arrow{r}{L_1} \arrow[swap]{d}{\deg(f,\cdot) f \otimes \phi} & RV_1 \otimes M \arrow{d}{f \otimes \phi} \\
RV_2^\circ \arrow{r}{L_2} & RV_2 & RV_2^\circ \otimes M' \arrow{r}{L_2} & RV_2 \otimes M'
\end{tikzcd}
\]
In the diagram at right, the kernels of the horizontal maps are $\c U_0(G_1,L_1,M)$ and $\c U_0(G_2,L_2,M')$ respectively, so we obtain a map $\c U_0(G_1,L_1,M) \to \c U_0(G_2,L_2,M')$ satisfying
\[
\sum_{x \in V_1^\circ} x \otimes u(x) \mapsto \sum_{x \in V_1^\circ} \deg(f,x) f(x) \otimes \phi(u(x)).
\]
The asserted formula follows from grouping the terms in the right-hand sum by the value of $f(x)$.  Again, functoriality of the construction is straightforward to check, and so is the naturality of the transformation in Proposition \ref{prop:tor}.
\end{proof}

\subsection{Applications of Functoriality}

We now give some immediate applications of functoriality to the critical group and the characteristic polynomial of $L$.  The following application to the critical group and spanning trees is a generalization of \cite[Theorem 5.7]{Berman} and \cite[Proposition 19 and Corollary 20]{Treumann}.  We do not know of a combinatorial proof of Corollary \ref{cor:spanningtreecover} below and suggest it as a question for future research.

\begin{proposition}
Viewed as the torsion submodule of $\Upsilon_{\Z}(G,L_{\std})$, the critical group is a covariant functor from finite connected graphs to $\Z$-modules, where the morphisms between graphs are the harmonic morphisms.  If $f \colon G_1 \to G_2$ is a non-constant harmonic morphism between connected graphs, then $\Crit(f): \Crit(G_1) \to \Crit(G_2)$ is surjective, and if $f$ is constant, then it is zero.
\end{proposition}

\begin{proof}
Note that $f$ induces a map $\Upsilon_{\Z} f \colon \Upsilon_{\Z}(G_1,L_{\std,1}) \to \Upsilon_{\Z}(G_2,L_{\std,2})$.  Also, recall from Proposition \ref{prop:criticalgroupnoboundary} that $\Crit(G_j)$ is the torsion part of $\Upsilon_{\Z}(G_j, L_{\std,j})$, and for each vertex $x \in V_1$, we have a internal direct sums
\begin{align*}
\Upsilon_{\Z}(G_1, L_{\std,1}) &= \Crit(G_1) \oplus \Z x \\
\Upsilon_{\Z}(G_2, L_{\std,2}) &= \Crit(G_2) \oplus \Z f(x).
\end{align*}
The map $\Upsilon f$ maps the first summand into the first summand and the second summand into the second summand.  In particular, $f$ restricts to a map $\Crit(f): \Crit(G_1) \to \Crit(G_2)$.

If $f$ is non-constand, then it is surjective by Proposition \ref{prop:harmonicmorphismsurjective}.  In this case, $f$ induces a surjection $RV_1 \to RV_2$, and hence by passing to the quotient, $\Upsilon f \colon \Upsilon_{\Z}(G_1, L_{\std,1}) \to \Upsilon_{\Z}(G_2, L_{\std,2})$ is surjective.  It follows that in the direct sum decomposition above, $\Upsilon f$ must be surjective on each summand.  On the other hand, if $f$ is constant, then $f$ maps every vertex to $f(x)$ and hence the image of $\Upsilon f$ is completely contained in $\Z f(x)$.  Thus, $\Upsilon f$ must map $\Crit(G_1)$ to zero.
\end{proof}

\begin{corollary} \label{cor:spanningtreecover}
Suppose $f \colon G_1 \to G_2$ is a non-constant harmonic morphism of graphs.  Then the number of spanning trees on $G_2$ divides the number of spanning trees on $G_1$.
\end{corollary}

\begin{proof}
It is well-known that the number of spanning trees $\tau(G_j)$ equals the size of the critical group $\Crit(G_j)$; see for instance \cite[\S 3]{HLMPPW}. Since $\Crit(f)$ is surjective, the order of $\Crit(G_2)$ divides the order of $\Crit(G_1)$.
\end{proof}

\begin{proposition} \label{prop:characteristicpolynomial}
Suppose $f \colon G_1 \to G_2$ is a non-constant harmonic morphism of connected graphs which satisfies $\deg(f,x) = n$ for all $x \in V_1$.  Then $f$ induces a $\C[z]$-network morphism
\[
f_*: (G_1, nzI - L_{\std,1}) \to (G_2, zI - L_{\std,2}).
\]
Moreover the induced map $\Upsilon_{\C[z]} f_*$ is surjective.  In particular, if $G_1$ and $G_2$ are finite, then the characteristic polynomials of the standard Laplacians for the two graphs are related by
\[
\det(zI - L_{\std,2}) | \det(nzI - L_{\std,1}).
\]
\end{proposition}

\begin{proof}
Observe that $f$ defines a $\C[z]$-network morphism $(G_1, nzI - L_{\std,1}) \to (G_2, zI - L_{\std,2})$; the second condition of Definition \ref{def:Rnetworkmorphism} holds because for $L_1 = nzI - L_{\std,1}$ and $L_2 = zI - L_{\std,2}$, we have
\[
d_1(x) = nz = \deg(f,x) z = \deg(f,x) \cdot d_2(f(x)).
\]
By Proposition \ref{prop:harmonicmorphismsurjective}, $f$ is surjective on vertices on edges.  By  Lemma \ref{lem:upsilonfunctor}, we have a surjective $\C[z]$-module morphism
\[
\Upsilon_{\C[z]}(G_1, nzI - L_{\std,1}) \to \Upsilon_{\C[z]}(G_2, zI - L_{\std,2}),
\]
which by application of the $\Hom$ functor induces an injective map
\[
\c U(G_2, zI - L_{\std,2}, \C[z] / (z - \lambda)) \to \c U(G_1, nzI - L_{\std,1}, \C[z] / (z - \lambda)).
\]
Similar to Example \ref{ex:characteristicpolynomial2}, this means that $\lambda$-eigenvectors of $G_2$ pull back to $n \lambda$-eigenvectors of $G_1$.  Thus, if $\lambda$ is an eigenvalue of $L_{\std,2}$, then $n \lambda$ is also an eigenvalue of $L_{\std,1}$ with the same or greater multiplicity.  This implies $\det(zI - L_{\std,2}) | \det(nzI - L_{\std,1})$.
\end{proof}

\begin{remark}
In addition, Lemma \ref{lem:u0functor} yields a map
\[
\c U_0(G_1, nzI - L_{\std,1}, \C[z] / (z - \lambda)) \to \c U_0(G_2, zI - L_{\std,2}, \C[z] / (z - \lambda)),
\]
which implies that $n\lambda$-eigenvectors of $L_{\std,1}$ push forward to $\lambda$-eigenvectors of $L_{\std,2}$.  In particular, if $n \lambda$ is an eigenvalue of $G_1$ but $\lambda$ is not an eigenvalue for $G_2$, then any $n \lambda$-eigenvector of $G_1$ will push forward to zero; in other words, the values on each fiber will add up to zero.  These results on the characteristic polynomial and eigenvectors generalize in a straightforward way to weighted graphs as well.
\end{remark}

\section{Layer-Stripping and Harmonic Continuation} \label{sec:layering}

\subsection{Layerable Extensions}

Based on the technique of layer-stripping from the electrical inverse problem (see \S \ref{subsec:motivationlayering} and \cite{CIM,CM,dVGV,layering}), we shall now describe three operations (simple layerable extensions), which add a vertex or edge onto the boundary of a $\partial$-graph.  Individually, these modifications are simple enough that their effect on $\Upsilon$, $\c U$, and $\c U_0$ is easy to understand, but when applied in sequence, they provide nontrivial information about our algebraic invariants.  We will show, for instance, that if $(G',L')$ is obtained from $(G,L)$ by a sequence of simple layerable extensions, then any harmonic function on $(G,L)$ extends to a harmonic function on $(G',L')$.  For simplicity, we focus on the case of finite $\partial$-graphs first and then consider infinite $\partial$-graphs in \S \ref{subsec:infinitefiltration}.

\begin{definition} \label{def:adjointIBV}
An {\bf isolated boundary vertex} $x$ of a $\partial$-graph is a boundary vertex with no neighbors incident to it.  We say $G'$ is obtained from $G$ by {\bf adjoining the isolated boundary vertex} $x$ if $x$ is an isolated boundary vertex in $G'$ and
\[
V(G') = V(G) \sqcup \{x\}, \quad E(G') = E(G), \quad V^\circ(G') = V^\circ(G).
\]
Equivalently, we say that $G$ is obtained from $G'$ by {\bf deleting the isolated boundary vertex} $x$.  See Figure \ref{fig:adjoinisolatedboundaryvertex}.
\end{definition}

\begin{figure}
\begin{center}

\begin{tikzpicture}[scale=0.7]
	\begin{scope}
		\node at (-0.8,0.5) {$G$};
		\node[bd] (1) at (0,0) {};
		\node[int] (2) at (1,0) {};
		\draw (1) to (2);
	\end{scope}

	\begin{scope}[shift={(6,0)}]
		\node at (-2.2,0.5) {$G'$};
		\node[bd] (0) at (-1,0) [label = left:$x$] {};
		\node[bd] (1) at (0,0) {};
		\node[int] (2) at (1,0) {};
		\draw (1) to (2);
	\end{scope}
\end{tikzpicture}

\caption{Adjoining or deleting an isolated boundary vertex $x$.}  \label{fig:adjoinisolatedboundaryvertex}

\end{center}
\end{figure}
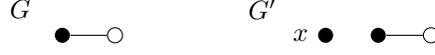

\begin{definition} \label{def:adjoinspike}
A {\bf boundary spike} $e$ of a $\partial$-graph is an edge $e$ such that one endpoint $e_-$ is an interior vertex and the other $e_+$ is a boundary vertex with no other edges incident to it.  We say $G'$ is obtained from $G$ by {\bf adjoining the boundary spike} $e$ if $e$ is a boundary spike in $G'$ and we have
\[
V(G') = V(G) \sqcup \{e_+\}, \quad E(G') = E(G) \sqcup \{e,\overline{e}\}, \quad V^\circ(G') = V^\circ(G) \sqcup \{e_-\}.
\]
Equivalently, we say that $G$ is obtained from $G'$ by {\bf contracting the boundary spike} $e$.  See Figure \ref{fig:adjoinboundaryspike}.
\end{definition}

\begin{figure}
\begin{center}

\begin{tikzpicture}[scale=0.7]
	\begin{scope}
		\node at (-0.8,1.5) {$G$};
	
		\node[bd] (1) at (0,0) {};
		\node[bd] (2) at (0,1) {};
		\node[int] (3) at (1,0) {};
		\node[int] (4) at (1,1) {};
		
		\draw (1) to (2) to (4) to (3) to (1);
	\end{scope}
	
	\begin{scope}[shift = {(6,0)}]
		\node at (-1,1.5) {$G'$};
		
		\node[bd] (0) at (-1,0) [label = below:$e_+$] {};
		\node[int] (1) at (0,0) [label = below:$e_-$] {};
		\node[bd] (2) at (0,1) {};
		\node[int] (3) at (1,0) {};
		\node[int] (4) at (1,1) {};
		
		\draw (0) to node[auto] {$e$} (1) to (2) to (4) to (3) to (1);
	\end{scope}
\end{tikzpicture}

\caption{Adjoining or contracting a boundary spike $e$.} \label{fig:adjoinboundaryspike}

\end{center}
\end{figure}
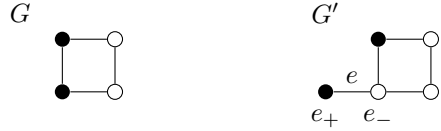

\begin{definition} \label{def:adjoinboundaryedge}
A {\bf boundary edge} $e$ of a $\partial$-graph is an edge $e$ such that both endpoints are boundary vertices.  We say that $G'$ is obtained from $G$ by {\bf adjoining a boundary edge} $e$ if $e$ is a boundary edge in $G'$ and
\[
V(G') = V(G), \quad E(G') = E(G) \sqcup \{e,\overline{e}\}, \quad V^\circ(G') = V^\circ(G).
\]
Equivalently, we say that $G$ is obtained from $G'$ by {\bf deleting the boundary edge} $e$.
\end{definition}

\begin{figure}
\begin{center}

\begin{tikzpicture}[scale=0.7]
	\begin{scope}
		\node at (-0.8,1.5) {$G$};
	
		\node[bd] (1) at (0,0) {};
		\node[bd] (2) at (0,1) {};
		\node[int] (3) at (1,0) {};
		\node[int] (4) at (1,1) {};
		
		\draw (2) to (4) to (3) to (1);
	\end{scope}
	
	\begin{scope}[shift = {(5,0)}]
		\node at (-1.2,1.5) {$G'$};
		
		\node[bd] (1) at (0,0) [label = below:$e_-$] {};
		\node[bd] (2) at (0,1) [label = above:$e_+$] {};
		\node[int] (3) at (1,0) {};
		\node[int] (4) at (1,1) {};
		
		\draw (1) to node[auto] {$e$} (2) to (4) to (3) to (1);
	\end{scope}
\end{tikzpicture}

\caption{Adjoining or deleting a boundary edge $e$.} \label{fig:adjoinboundaryedge}

\end{center}
\end{figure}
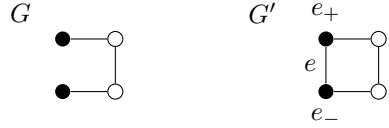

\begin{definition} \label{def:simplelayerableextension}
If $G'$ is obtained from $G$ by adjoining an isolated boundary vertex, boundary spike, or boundary edge, then we say that $G'$ is a {\bf simple layerable extension} of $G$.  We will equivalently say that $G$ is obtained from $G'$ by a {\bf layer-stripping operation}.
\end{definition}

\begin{observation} \label{obs:layeringsubdgraph}
If $G'$ is a simple layerable extension of $G$, then $G$ is a sub-$\partial$-graph of $G'$, as one can check by straightforward casework.
\end{observation}

\begin{definition}
We say $(G',L')$ is obtained from $(G,L)$ by adjoining an isolated boundary vertex, boundary spike, or boundary edge, if $G'$ is obtained from $G$ by the corresponding operation and in addition $w'|_{E(G)} = w$ and $d'|_{V(G)} = d$.  (Note that given the geometric setup, $w'|_{E(G)} = w$ and $d'|_{V(G)} = d$ is a necessary and sufficient condition to make the inclusion $(G,L) \to (G',L')$ an $R$-network morphism.)
\end{definition}

\begin{lemma} \label{lem:layeringUpsilon1}
Let $(G',L')$ be an $R^\times$-network.  If $(G',L')$ is obtained from $(G,L)$ by adjoining a boundary spike or boundary edge, then the induced maps $\Upsilon(G,L) \to \Upsilon(G',L')$, $\c U(G,L,M) \to \c U(G',L',M)$ are isomorphisms.

If $(G',L')$ is obtained from $(G,L)$ by adjoining an isolated boundary vertex $x$, then these maps furnish isomorphisms
\[
\Upsilon(G',L') \cong \Upsilon(G,L) \oplus Rx
\]
and
\[
\c U(G',L',M) \cong \c U(G,L,M) \times M^{\{x\}}.
\]
\end{lemma}

\begin{proof}
Suppose $G'$ is obtained from $G$ by adjoining an isolated boundary vertex $x$.  Then we have
\begin{align*}
RV(G') &= RV(G) \oplus Rx \\
L'(RV^\circ(G')) &= L(RV^\circ(G)) \subseteq RV(G).
\end{align*}
Taking the quotient of the top row by the bottom row yields $\Upsilon(G',L') \cong \Upsilon(G,L) \oplus Rx$.  Then applying $\Hom(-,M)$ yields the desired equation for $\c U(-,M)$.

Suppose $G'$ is obtained from $G$ by adjoining a boundary edge.  Then $V(G') = V(G)$ and $V^\circ(G') = V^\circ(G)$.  Since the endpoints of $e$ are boundary vertices, the two Laplacians $L'$ and $L$ agree on $RV^\circ(G)$.  Hence, $\Upsilon(G',L') = \Upsilon(G,L)$, and application of $\Hom(-,M)$ yields the corresponding statement for $\c U(-,M)$.

Finally, suppose $G'$ is obtained from $G$ by adjoining a boundary spike $e$ with $x = e_+ \in \partial V(G')$ and $y = e_- \in V^\circ(G')$.  Then consider the commutative diagram:
\[
\begin{tikzcd}
0 \arrow{r} & RV^\circ(G) \arrow{r} \arrow{d}{L} & RV^\circ(G') \arrow{r} \arrow{d}{L'} & Ry \arrow{r} \arrow{d}{\psi} & 0 \\
0 \arrow{r} & RV(G) \arrow{r} & RV(G') \arrow{r} & Rx \arrow{r} & 0,
\end{tikzcd}
\]
where the horizontal arrows are given by the direct sum decompositions induced by $V^\circ(G') = V^\circ(G) \sqcup \{y\}$ and $V(G') = V(G) \sqcup \{x\}$, and the vertical arrow $\psi: Ry \to Rx$ is given by $y \mapsto -w'(e)x$.  Commutativity of the left square follows from the fact that $w'|_G = w$ and $d'|_G = d$.  To check commutativity of the right square, start with an arbitrary element of $RV^\circ(G')$ written in the form $z + ry$, where $z \in RV^\circ(G)$ and $r \in R$.  Going right and then down produces $-rw'(e)x$.  In the other direction, going down from $RV^\circ(G')$ to $RV(G')$ yields
\[
L'(z + ry) = L'z + r \left(d'(y)y + \sum_{e':e_+' = y} w'(e')(y - e_-)\right) \in RV(G) - r w'(e)x,
\]
and then following the diagram right to $Rx$ yields $-rw'(e)x$.

The rows are clearly exact.   Hence, the Snake Lemma yields an exact sequence
\[
\dots \to \ker \psi \to \Upsilon(G,L) \to \Upsilon(G',L') \to \coker \psi \to 0.
\]
Since $\psi$ is an isomorphism, this shows that $\Upsilon(G,L) \to \Upsilon(G',L')$ is an isomorphism, and application of $\Hom(-,M)$ yields the corresponding statement for $\c U(-,M)$ by Lemma \ref{lem:hom}.
\end{proof}

\begin{lemma} \label{lem:layeringUpsilon2}
If an $R^\times$-network $(G',L')$ is obtained from $(G,L)$ by a simple layerable extension, then the induced map $\c U_0(G,L,M) \to \c U_0(G',L',M)$ is an isomorphism.
\end{lemma}

\begin{proof}
Since $(G,L)$ is a subnetwork of $(G',L')$, the map $\c U_0(G,L,M) \to \c U_0(G',L',M)$ defined by Lemma \ref{lem:u0functor} extends a function $u$ on $G$ to a function $u'$ on $G'$ by setting $u'|_{V(G') \setminus V(G)} = 0$.  This map is clearly injective.  We prove surjectivity by cases.

Suppose $G'$ is obtained from $G$ by adjoining an isolated boundary vertex.  If $u' \in \c U_0(G',L',M)$, then clearly $u'(x) = 0$.  Moreover, $L'u'|_{V(G)} = L(u'|_{V(G)})$.  Thus, $u'$ restricts to a function in $u \in \c U_0(G,L,M)$, and $u$ is mapped to $u'$ by the extension map.

Suppose $G'$ is obtained from $G$ by adjoining a boundary edge $e$.  Recall $V(G') = V(G)$, so functions on $G'$ and functions on $G$ are equivalent.  If $u \in \c U_0(G',L',M)$, then $u(e_+) = u(e_-) = 0$ and hence $Lu = L'u$.  Thus, $u \in \c U_0(G',L',M)$ and $u$ is mapped to itself by the extension map.

Suppose $G'$ is obtained from $G$ by adjoining a boundary spike $e$ with boundary endpoint $e_+ = x$ and interior endpoint $e_- = y$.  Suppose that $u' \in \c U_0(G',L',M)$ and let $u = u'|_{V(G)}$.  Note that
\[
0 = L'u'(x) = d(x) u'(x) + w(e)(u'(x) - u'(y)) = 0 + w(e)(0 - u(y)).
\]
Since $w(e)$ is a unit in $R$, this implies $u(y) = 0$.  Moreover, we have
\[
0 = L'u'(y) = Lu(y) + w(e)(u'(y) - u'(x)) = Lu(y).
\]
Thus, $u(y) = Lu(y) = 0$, which shows $u \in \c U_0(G,L,M)$.  Thus, $u'$ is in the image of the extension map.
\end{proof}

\begin{remark} \label{rem:nonunitlayering}
In Lemmas \ref{lem:layeringUpsilon1} and \ref{lem:layeringUpsilon2}, the only place where we used the fact that $w(e) \in R^\times$ was for the case of a boundary spike.  When $w(e)$ is not a unit in $R$, the Snake Lemma still yields an exact sequence relating $\Upsilon(G,L)$ and $\Upsilon(G',L')$.  Though we will focus on the case of unit edge weights in this paper, applying layerable filtrations in the general case seems like a promising avenue for future research.
\end{remark}

Now we will describe layerable extensions formed from a sequence of simple layerable extensions.

\begin{definition}
We say that $G'$ is a {\bf finite layerable extension} of $G$ if there exists a finite of sequence of sub-$\partial$-graphs of $G'$
\[
G = G_0 \subseteq G_1 \subseteq \dots \subseteq G_n = G'
\]
such that $G_j$ is a simple layerable extension of $G_{j-1}$.  We call $\{G_j\}$ a {\bf layerable filtration from $G$ to $G'$}.  We say that $G_j$ is {\bf layerable} if it is a layerable extension of $\varnothing$.
\end{definition}

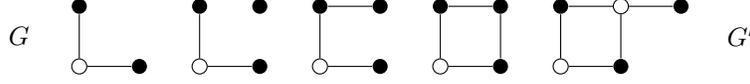
\begin{figure}
\begin{center}

\begin{tikzpicture}[scale=0.8]

	\node at (-1,0.5) {$G$};
	\node at (11,0.5) {$G'$};

	\begin{scope}
		\node[int] (00) at (0,0) {};
		\node[bd] (01) at (0,1) {};
		\node[bd] (10) at (1,0) {};
		
		\draw (01) to (00) to (10); 
	\end{scope}

	\begin{scope}[shift={(2,0)}]
		\node[int] (00) at (0,0) {};
		\node[bd] (01) at (0,1) {};
		\node[bd] (10) at (1,0) {};
		\node[bd] (11) at (1,1) {};
		
		\draw (01) to (00) to (10);
	\end{scope}

	\begin{scope}[shift={(4,0)}]
		\node[int] (00) at (0,0) {};
		\node[bd] (01) at (0,1) {};
		\node[bd] (10) at (1,0) {};
		\node[bd] (11) at (1,1) {};
		
		\draw (01) to (00) to (10);
		\draw (01) to (11);
	\end{scope}

	\begin{scope}[shift={(6,0)}]
		\node[int] (00) at (0,0) {};
		\node[bd] (01) at (0,1) {};
		\node[bd] (10) at (1,0) {};
		\node[bd] (11) at (1,1) {};
		
		\draw (01) to (00) to (10);
		\draw (01) to (11);
		\draw (10) to (11);
	\end{scope}

	\begin{scope}[shift={(8,0)}]
		\node[int] (00) at (0,0) {};
		\node[bd] (01) at (0,1) {};
		\node[bd] (10) at (1,0) {};
		\node[int] (11) at (1,1) {};
		\node[bd] (21) at (2,1) {};
		
		\draw (01) to (00) to (10);
		\draw (01) to (11) to (21);
		\draw (10) to (11);
	\end{scope}

\end{tikzpicture}

\caption{A layerable filtration from a $\partial$-graph $G$ to a layerable extension $G'$ of $G$.}

\end{center}
\end{figure}

\begin{proposition} \label{prop:filtration}
Suppose that a finite $R^\times$-network $(G',L')$ is a layerable extension of $(G,L)$ through the filtration $\{(G_j,L_j)\}_{j=0}^n$.  Let $S \subseteq V(G')$ be the set of vertices which are adjoined as isolated boundary vertices at some step of the filtration.  Then the inclusion map $(G,L) \to (G',L')$ induces isomorphisms
\begin{align*}
\Upsilon(G',L') &\cong \Upsilon(G,L) \oplus RS \\
\c U(G',L',M) &\cong \c U(G,L,M) \times M^S \\
\c U_0(G',L',M) &\cong \c U_0(G,L,M).
\end{align*}
\end{proposition}

\begin{proof}
In the first claim, the map $\Upsilon(G,L) \to \Upsilon(G',L')$ is given by functoriality of $\Upsilon$ and the map $RS \to \Upsilon(G',L')$ is given by the composition $RS \to RV(G') \to \Upsilon(G',L')$.  Let $S_j = S \cap V(G_j)$.  The three cases of Lemma \ref{lem:layeringUpsilon1} show that the map
\[
\Upsilon(G_j,L_j) \oplus R(S_{j+1} \setminus S_j) \to \Upsilon(G_{j+1}, L_{j+1})
\]
is an isomorphism (in the case of adjoining a boundary spike or boundary edge, $S_{j+1} \setminus S_j = \varnothing$).  By induction, this implies that
\[
\Upsilon(G,L) \oplus RS_n \to \Upsilon(G_n,L_n) \text{ is an isomorphism,}
\]
which completes the proof because $(G_n,L_n) = (G',L')$.

The second claim follows by application of the functor $\Hom_R(-,M)$ to the first claim, or alternatively by inductive application of the second isomorphism in Lemma \ref{lem:layeringUpsilon1}.  The third claim follows by inductive application of Lemma \ref{lem:layeringUpsilon2}.
\end{proof}

The claim $\c U(G',L',M) \cong \c U(G,L,M) \times M^S$ in Proposition \ref{prop:filtration} has the following interpretation in terms of harmonic extensions:  Any harmonic function $u$ on $(G,L)$ extends to a harmonic function $u'$ on $(G',L')$.  For any $\phi \in M^S$, there is a unique harmonic extension $u'$ of $u$ such that $u'|_S = \phi$.  We can see from the inductive application of Lemma \ref{lem:layeringUpsilon1} that the extension $u'$ can be constructed using step-by-step extensions from $(G_j,L_j)$ to $(G_{j+1},L_{j+1})$.  In other words, a layerable filtration provides a geometric model for discrete harmonic continuation, loosely analogous to the sequences of domains used for harmonic continuation in complex analysis.

The special case of Proposition \ref{prop:filtration} where $G = \varnothing$ deserves special comment:

\begin{proposition} \label{prop:layerablebehavior}
Suppose that $(G',L')$ is a layerable $R^\times$-network, and let $\{(G_j,L_j)\}$ be a layerable filtration from $\varnothing$ to $(G',L')$.  Let $S$ be the set of vertices which are adjoined as isolated boundary vertices at some step of the filration.
\begin{enumerate}
	\item $\Upsilon(G',L')$ is a free $R$-module and a basis is given by $S$.
	\item For every $R$-module $M$, for every $\phi \in M^S$, there is a unique harmonic function $u$ such that $u|_S = \phi$.
	\item The network $(G',L')$ is non-degenerate and $\c U_0(G,L,M) = 0$ for every $R$-module $M$.
\end{enumerate}
\end{proposition}

\begin{remark}
If the network in Proposition \ref{prop:layerablebehavior} is finite, then $|S| = |\partial V|$ (one verifies this by induction after observing that adjoining a boundary spike or boundary edge does not change $|\partial V|$).  Thus, $\Upsilon(G,L)$ is a free $R$-module of rank $|\partial V|$.
\end{remark}

\begin{remark} \label{rem:inwardcontinuation}
If we unwind the proof leading up to Proposition \ref{prop:layerablebehavior} (3), it does in fact boil down to the intuitive argument given in \S \ref{subsec:motivationlayering} that we are starting with zero data on the boundary and deducing that $u = 0$ further and further into the network using harmonic continuation.  Indeed, if $(G_0,L_0) \subseteq (G_1,L_1) \subseteq \dots (G_n,L_n) = (G,L)$ is a layerable filtration, then Lemma \ref{lem:layeringUpsilon2} shows
\[
(u|_{\partial V_{j+1}}, L_{j+1} u|_{\partial V_{j+1}}) = 0 \implies (u|_{\partial V_j}, L_j u|_{\partial V_j}) = 0.
\]
Inductive application of this statement showed that zero potential and current conditions propagate inward from $\partial V_n$ to $\partial V_{n-1}$ and so forth to $\partial V_0$.  Since all vertices are contained in some $\partial V_j$, this shows $u \equiv 0$.  By linearity, the same argument shows that a harmonic function is uniquely determined by its boundary data and the values can be deduced using harmonic continuation from the boundary inward.
\end{remark}

\subsection{Algebraic Characterization of Layerability}

We are now ready to give our algebraic characterization of layerability.  We begin by considering networks over a field and then apply our algebraic machinery to state several other equivalent conditions.

\begin{lemma} \label{lem:layerabilityfieldcharacterization}
Let $G$ be a finite $\partial$-graph and let $F$ be a field with at least three elements.  Then $G$ is layerable if and only if every $F^\times$-network on the $\partial$-graph $G$ is non-degenerate.
\end{lemma}

\begin{proof}
If $G$ is layerable and $(G,L)$ is an $F^\times$-network on $G$, then $(G,L)$ is non-degenerate by Proposition \ref{prop:layerablebehavior} (3).

To prove the converse, suppose $G$ is not layerable.  Let $G'$ be a minimal non-layerable sub-$\partial$-graph of $G$.  Since $G'$ is minimal, it cannot have a boundary spike, boundary edge, or isolated boundary vertex, since performing a layer-stripping operation on $G'$ would preserve layerability.  Hence, every boundary vertex in $G'$ must have multiple edges incident to it, and all its neighbors are interior vertices.  In particular, $\mathcal{E}(x)$ and $\mathcal{E}(y)$ are disjoint for distinct vertices $x, y \in \partial V(G')$.  Since $F$ has at least three elements, we can write zero as the sum of $n$ nonzero elements for every $n \in \N$.  Thus, we can choose $w': E(G') \to F^\times$ such that $\sum_{e \in \mathcal{E}(x)} w'(e) = 0$ for every $x \in \partial V(G')$.

Let $u: V(G') \to F$ be given by $u = 0$ on $\partial V(G')$ and $u = 1$ on $V^\circ(G')$.  Define $d: V(G') \to F$ such that
\[
d'(x) = - \sum_{\substack{e \in \mathcal{E}(x) \\ e_- \in \partial V(G')}} w'(e).
\]
Then note that
\[
L'u(x) = d'(x) + \sum_{\substack{e \in \mathcal{E}(x) \\ e_- \in \partial V(G')}} w'(e) = 0 \text{ for } x \in V^\circ(G'),
\]
while
\[
L'u(x) = -\sum_{e \in \mathcal{E}(x)} w'(e) = 0 \text{ for } x \in \partial V(G').
\]
Therefore, $u \in \c U_0(G',L',F)$ and clearly $u$ is not identically zero.

Let $w$ and $d$ be any extension of $w'$ and $d'$ to $G$.  Then $(G',L') \to (G,L)$ is an injective $F^\times$-network morphism and hence extension by zero defines an injective map $\c U_0(G',L') \to \c U_0(G,L)$ (as a special case of Lemma \ref{lem:u0functor}).  Thus, $\c U_0(G,L)$ is nonzero, so $(G,L)$ is degenerate as desired.
\end{proof}

Next, we state some other equivalent algebraic conditions.  In (4) below, $R^*(G,F)$ and $L^*(G,F)$ are the ring and network defined in Definition \ref{def:genericfieldnetwork}.  The ring $R^*$ is the polynomial algebra $F(t_e^{\pm 1}, e \in E; t_x, x \in V)$, where $t_{\overline{e}} = t_e$, and $L^*$ is given by setting $w(e) = t_e$ and $d(x) = t_x$.

\begin{theorem} \label{thm:layerabilitycharacterization}
Let $G$ be a finite $\partial$-graph.  The following are equivalent:
\begin{enumerate}
	\item $G$ is layerable.
	\item For every ring $R$, every $R^\times$-network on $G$ is non-degenerate.
	\item For every ring $R$, for every non-degenerate $R^\times$-network $(G,L)$ on the $\partial$-graph $G$, $\Upsilon(G,L)$ is a free $R$-module.
	\item There exists a field $F$ with at least three elements such that $\Upsilon(G, L^*(G,F))$ is a flat $R^*(G,F)$-module.
	\item There exists a field $F$ with at least three elements such that every $F^\times$-network on $G$ is non-degenerate.
\end{enumerate}
\end{theorem}

\begin{proof}
We prove the implications in two cycles.  First, (1) $\implies$ (2) by Proposition \ref{prop:layerablebehavior}; (2) $\implies$ (5) trivially; and (5) $\implies$ (1) by Lemma \ref{lem:layerabilityfieldcharacterization}.

Next, observe (1) $\implies$ (3) by Proposition \ref{prop:layerablebehavior}; (3) $\implies$ (4) because $(G,L^*)$ is non-degenerate (Proposition \ref{prop:genericfieldnetwork}) and because free modules are automatically flat; (4) $\implies$ (5) was proved in Proposition \ref{prop:genericfieldnetwork}; and we already know (5) $\implies$ (1).
\end{proof}

\begin{remark}
Based on the statement alone, one can come up with several other equivalent conditions:  For instance, ``$\Upsilon(G,L)$ is flat for every non-degenerate $R^\times$-network on $G$ for every ring $R$,'' or ``every $F^\times$-network on $G$ is non-degenerate for every field $F$.''
\end{remark}

\subsection{Transformations of Harmonic Boundary Data} \label{subsec:layeringcoordinates}

We now describe how adjoining a boundary spike or boundary edge affects the boundary potential and current data of harmonic functions.  The simple coordinate system given here will be used later in Theorem \ref{thm:explicitalgorithm}, which gives an algorithm for simplifying computation of $\c U_0(G,L,M)$.  The discussion will also relate layer-stripping to symplectic matrices as in \cite{LP} \cite[\S 12]{WJ}.

For a harmonic function $u \in \c U(G,L,M)$, the term {\bf boundary data} will refer to $(u|_{\partial V}, Lu|_{\partial V})$.  If $(G',L')$ is a simple layerable extension of $(G,L)$ and if $u \in \c U(G,L,M)$ extends to $u' \in \c U(G',L',M')$, then it is easy to explicitly compute the boundary data of $u'$ from that of $u$.

{\bf Adjoining a Boundary Spike:}  Suppose $G'$ is obtained from $G$ by adjoining a boundary spike $e$ with $e_+ \in \partial V(G')$ and $e_- \in V^\circ(G')$, and assume $w(e) = w \in R^\times$ and $d(e_+) = d \in R$.  For explicitness, we will index the boundary vertices of each $\partial$-graph by integers:  Let $[m]$ denote $\{1,\dots,m\}$, and let $\ell: [m] \to \partial V(G)$ and $\ell': [m] \to \partial V(G')$ be bijections labelling the vertices.  We assume the two labellings are {\bf consistent}, meaning that $e_- \in \partial V(G)$ and $e_+ \in \partial V(G')$ have the same index, and every $x \in \partial V(G) \cap \partial V(G')$ has the same index with respect to the two different labellings.

Let $j$ be the index of $e_- \in \partial V(G)$ and $e_+ \in \partial V(G')$, so that $e_+ = \ell'(j)$ and $e_- = \ell(j)$.  Let $E_{p,q}$ denote the matrix with a $1$ in the $(p,q)$ entry and zeros elsewhere.  Then we claim that
\[
\begin{pmatrix} u' \circ \ell' \\ L'u' \circ \ell' \end{pmatrix} = \begin{pmatrix} I & w^{-1}E_{j,j} \\ d E_{j,j} & I + dw^{-1} E_{j,j} \end{pmatrix} \begin{pmatrix} u \circ \ell \\ Lu \circ \ell \end{pmatrix},
\]
where the matrix blocks are $m \times m$, and $u \circ \ell$ is viewed as a vector in $M^m$.  Note that the matrix has a simpler form when $d = 0$.  To verify the matrix formula, note that
\[
0 = L'u'(e_-) = Lu(e_-) + w \cdot (u(e_-) -  u(e_+))
\]
so that
\[
u'(e_+) = u(e_-) + w^{-1} Lu(e_-)
\]
and
\[
L'u'(e_+) = d \cdot u'(e_+) + w(u'(e_+) - u'(e_-)) = d u(e_-) + (d w^{-1} + 1) Lu(e_-)
\]

{\bf Adjoining a Boundary Edge:}  Suppose $G'$ is obtained from $G$ by adjoining a boundary edge $e$ with $w = w(e)$.  In this case, $\partial V(G) = \partial V(G')$, so two indexings $\ell: [m] \to \partial V(G)$ and $\ell': [m] \to \partial V(G')$ are called {\bf consistent} if $\ell = \ell'$.  Assume that the indices of $e_-$ and $e_+$ are $i$ and $j$.  Then a similar computation as before shows that
\[
\begin{pmatrix} u' \circ \ell' \\ L'u' \circ \ell' \end{pmatrix} = \begin{pmatrix} I & 0 \\ w(E_{i,i} + E_{j,j} - E_{i,j} - E_{j,i}) & I \end{pmatrix} \begin{pmatrix} u \circ \ell \\ Lu \circ \ell \end{pmatrix}.
\]

The matrices described above will be called the {\bf boundary data transformations} for adjoining a boundary spike or boundary edge.

{\bf Standard Form for Layerable Filtrations:}  In general a layerable filtration is allowed to mix adjoining boundary spikes, adjoining boundary edges, and adjoining isolated boundary vertices in any order.  However, it is sometimes convenient for the sake of computation to assume that all the isolated boundary vertices are adjoined before the other operations.  For finite networks, we can always arrange this; simply take all the isolated boundary vertices that are adjoined at any step of the filtration, adjoin them at the beginning, and include them in all subsequent $\partial$-graphs in the filtration.

A {\bf standard form layerable filtration} for $(G,L)$ is a sequence of $R$-networks
\[
(G_0,L_0) \subseteq (G_1,L_1) \subseteq \dots \subseteq (G_n,L_n) = (G,L),
\]
where $(G_0,L_0)$ consists entirely of isolated boundary vertices, and $(G_{j+1},L_{j+1})$ is obtained from $(G_j,L_j)$ by adjoining a boundary spike or boundary edge.  A {\bf consistent labelling} for such a filtration consists of (bijective) labelling functions $\ell_j: [m] \to \partial V_j$ such that for each $j$, $\ell_j$ and $\ell_{j+1}$ are consistent in the sense described above for the two cases of adjoining a boundary spike and adjoining a boundary edge.  See Figure \ref{fig:standardformfiltration}.

\begin{figure}
\begin{center}
\begin{tikzpicture}[scale=0.7]

	\begin{scope}[shift={(0,0)}]
		\node at (0.5,3) {$G_0$};
	
		\node[bd] (1) at (0,0) [label = right:$1$] {};
		\node[bd] (2) at (0,1) [label = right:$2$] {};
		\node[bd] (3) at (0.5,2) [label = right:$3$] {};
	\end{scope}

	\begin{scope}[shift={(2,0)}]
		\node at (0.5,3) {$G_1$};
	
		\node[bd] (1) at (0,0) [label = right:$1$] {};
		\node[bd] (2) at (0,1) [label = right:$2$] {};
		\node[bd] (3) at (0.5,2) [label = right:$3$] {};
		
		\draw (1) to (2);
	\end{scope}

	\begin{scope}[shift={(4,0)}]
		\node at (0.5,3) {$G_2$};
	
		\node[bd] (1) at (0,0) [label = right:$1$] {};
		\node[bd] (2) at (0,1) [label = right:$2$] {};
		\node[bd] (3) at (0.5,2) [label = right:$3$] {};
		
		\draw (1) to (2);
		\draw (2) to (3);
	\end{scope}

	\begin{scope}[shift={(6,0)}]
		\node at (0.5,3) {$G_3$};
	
		\node[bd] (1) at (0,0) [label = right:$1$] {};
		\node[int] (2) at (0,1) {};
		\node[bd] (3) at (0.5,2) [label = right:$3$] {};
		\node[bd] (5) at (1,1) [label = right:$2$] {};
		
		\draw (1) to (2);
		\draw (2) to (3);
		\draw (2) to (5);
	\end{scope}

	\begin{scope}[shift={(8,0)}]
		\node at (0.5,3) {$G_4$};
	
		\node[bd] (1) at (0,0) [label = right:$1$] {};
		\node[int] (2) at (0,1) {};
		\node[bd] (3) at (0.5,2) [label = right:$3$] {};
		\node[bd] (5) at (1,1) [label = right:$2$] {};
		
		\draw (1) to (2);
		\draw (2) to (3);
		\draw (2) to (5);
		\draw (1) to (5);
	\end{scope}

	\begin{scope}[shift={(10,0)}]
		\node at (0.5,3) {$G_5$};
	
		\node[int] (1) at (0,0) {};
		\node[int] (2) at (0,1) {};
		\node[bd] (3) at (0.5,2) [label = right:$3$] {};
		\node[bd] (4) at (1,0) [label = right:$1$] {};
		\node[bd] (5) at (1,1) [label = right:$2$] {};
		
		\draw (1) to (2);
		\draw (2) to (3);
		\draw (2) to (5);
		\draw (1) to (5);
		\draw (1) to (4);
	\end{scope}

	\begin{scope}[shift={(12,0)}]
		\node at (0.5,3) {$G_6$};
	
		\node[int] (1) at (0,0) {};
		\node[int] (2) at (0,1) {};
		\node[bd] (3) at (0.5,2) [label = right:$3$] {};
		\node[bd] (4) at (1,0) [label = right:$1$] {};
		\node[bd] (5) at (1,1) [label = right:$2$] {};
		
		\draw (1) to (2);
		\draw (2) to (3);
		\draw (2) to (5);
		\draw (1) to (5);
		\draw (1) to (4);
		\draw (3) to (5);
	\end{scope}

	\begin{scope}[shift={(14,0)}]
		\node at (0.5,3) {$G_7$};
	
		\node[int] (1) at (0,0) {};
		\node[int] (2) at (0,1) {};
		\node[bd] (3) at (0.5,2) [label = right:$3$] {};
		\node[bd] (4) at (1,0) [label = right:$1$] {};
		\node[bd] (5) at (1,1) [label = right:$2$] {};
		
		\draw (1) to (2);
		\draw (2) to (3);
		\draw (2) to (5);
		\draw (1) to (5);
		\draw (1) to (4);
		\draw (3) to (5);
		\draw (4) to (5);
	\end{scope}

	\begin{scope}[shift={(16,0)}]
		\node at (0.5,3) {$G_8$};
	
		\node[int] (1) at (0,0) {};
		\node[int] (2) at (0,1) {};
		\node[bd] (3) at (0.5,2) [label = right:$3$] {};
		\node[bd] (4) at (1,0) [label = right:$1$] {};
		\node[bd] (5) at (1,1) [label = right:$2$] {};
		
		\draw (1) to (2);
		\draw (2) to (3);
		\draw (2) to (5);
		\draw (1) to (5);
		\draw (1) to (4);
		\draw (3) to (5);
		\draw (4) to (5);
		\draw (3) to [bend right = 25] (4);
	\end{scope}
	
\end{tikzpicture}

\end{center}

\caption{A layerable filtration in standard form with a consistent labelling.} \label{fig:standardformfiltration}

\end{figure}
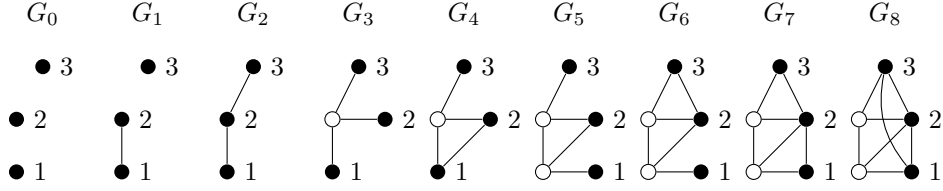

\begin{lemma} \label{lem:layeringcoordinates}
Let $(G,L)$ be a layerable $R^\times$-network and let $M$ be an $R$-module.  Let $(G_0,L_0), \dots, (G_n,L_n)$ be a standard form layerable filtration for $(G,L)$, let $\ell_0, \dots, \ell_n$ be a consistent labelling for it.  Let
\[
T_0 = \begin{pmatrix} I & 0 \\ D & I \end{pmatrix},
\]
where $D = \text{diag}(d \circ \ell_0(1), \dots, d \circ \ell_0(m))$.  For $j \geq 0$, let $T_j$ be the boundary data transformation associated to the operation $(G_{j-1},L_{j-1}) \mapsto (G_j,L_j)$.  For every $\phi \in M^m$, there is a unique harmonic function such that $u|_{\partial V_0} \circ \ell_0 = \phi$; and if $u_j = u|_{V_j}$, we have
\[
\begin{pmatrix} u_j \circ \ell_j \\ L_j u_j \circ \ell_j \end{pmatrix} = T_j \dots T_1 T_0 \begin{pmatrix} \phi \\ 0 \end{pmatrix} \text{ for } j = 0, \dots, n.
\]
\end{lemma}

\begin{proof}
The existence and uniqueness of $u$ follows from Proposition \ref{prop:filtration}.  To prove the matrix equation, note that $G_0$ consists of isolated boundary vertices.  Thus, all functions $u_0$ are harmonic and have $L_0 u(x) = d(x) u(x)$, which establishes the case $j = 0$.  Then the equation follows for all $j$ by inductive application of the foregoing computation.
\end{proof}

\begin{remark}
In the last lemma, computing the boundary data of $u_j$ at every step of the filtration is sufficient to find the values of $u$ on all of $G$.  This is because every vertex of $G$ must be a boundary vertex at some step of the filtration.
\end{remark}

\begin{remark}
To simplify the statement of Lemma \ref{lem:layeringcoordinates}, we have assumed that each layerable extension adds only one boundary spike or boundary edge.  In general, it may be computationally convenient to adjoin multiple spikes or multiple boundary edges at once.  For instance, if $m = 2$ and one adjoins two boundary spikes with parameters $w_1$, $w_2$ and $d_1$, $d_2$ then the matrix is
\[
\begin{pmatrix}
1 & 0 & w_1^{-1} & 0 \\
0 & 1 & 0 & w_2^{-1}  \\
d_1 & 0 & 1 + d_1 w_1^{-1} & 0 \\
0 & d_2 & 0 & 1 + d_2 w_2^{-1}
\end{pmatrix}.
\]
\end{remark}

\begin{remark} \label{rem:alternativebookkeeping}
Recording boundary potential and current data is not the only feasible bookkeeping method for the harmonic continuation process.  In some situations, it could more convenient to record the values of $u$ on the boundary of $G_j$ and all vertices adjacent to the boundary, rather than the recording $u$ and $Lu$ on the boundary.
\end{remark}

\begin{remark}
Direct computation will verify that the boundary data transformations discussed here are all symplectic matrices, that is, they satisfy
\[
T^t J T = J,
\]
where
\[
J = \begin{pmatrix} 0 & -I \\ I & 0 \end{pmatrix}.
\]
The relationship between layer-stripping and the symplectic group over $\R$ was studied in \cite{LP} from the viewpoint of Lie theory, while our discussion makes the connection in a more elementary and explicit way.  The same matrices were written down in \cite[p.\ 48]{WJ}, but without the interpretation as transformations of boundary behavior.
\end{remark}

\subsection{Using Harmonic Continuation to Compute $\c U_0$}

Harmonic continuation can be applied on non-layerable networks as well.  In \S \ref{subsec:CLFcontinuation}, we used harmonic continuation to compute $\c U_0(\CLF(m,n), L_{\std}, \Q / \Z)$.  Now we generalize this approach to compute $\c U_0(G,L,M)$ for a finite $R^\times$-network $(G,L)$, using layerable filtrations to keep track of the harmonic continuation process.

If $G$ is a $\partial$-graph and $S \subseteq V^\circ(G)$, then we define $G_{S \to \partial}$ as the $\partial$-graph obtained by changing the vertices in $S$ from interior to boundary, so that
\[
V(G_{S \to \partial}) = V(G), \quad E(G_{S \to \partial}) = E(G), \quad \partial V(G_{S \to \partial}) = \partial V(G) \sqcup S.
\]

\begin{theorem} \label{thm:explicitalgorithm}
Let $(G,L)$ be a finite $R^\times$-network, let $S \subseteq V^\circ(G)$, and suppose $G_{S \to \partial}$ is layerable.  Then there exists a matrix $A \in M_{|S| \times (|\partial V| + |S|)}(R)$ such that for every $R$-module $M$
\[
\c U_0(G,L,M) \cong \ker \left(A: M^{|S|} \to M^{|S| + |\partial V|}\right).
\]
The matrix $A$ can be computed explicitly from a given layerable filtration for $G_{S \to \partial}$, as described by equations \eqref{eq:complement} and \eqref{eq:matrixA} below.
\end{theorem}

\begin{proof}
First, let us motivate the proof in light of earlier results.  Let $G' = G_{S \to \partial}$ for short and recall $\partial V(G') = \partial V(G) \sqcup S$.  Note that since $G'$ and $G$ only differ in the assignment of boundary vertices, we have
\[
\c U_0(G,L,M) = \{u \in \c U(G',L,M) \colon u|_{\partial V(G)} = 0, Lu|_{\partial V(G')} = 0 \}.
\]
In other words, a function $u \in \c U_0(G,L,M)$ is equivalent to a harmonic function on $(G',L)$ satisfying the additional boundary conditions that $u|_{\partial V(G)} = 0$ and $Lu|_{\partial V(G')} = 0$.

Now $G'$ is layerable.  Thus, as explained in Remark \ref{rem:inwardcontinuation}, a harmonic function $u$ on $G'$ is uniquely determined by its boundary data.  Moreover, the values of the function $u$ can be found from the boundary data by harmonically continuing from $\partial V(G')$ inward along a layer-stripping filtration of $G'$.  Thus, to find $\c U_0(G,L,M)$, we need to start with boundary data on $G'$ satisfying our extra boundary conditions, harmonically continue inward, and then ensure that the result we get is actually harmonic on $G'$.

\begin{figure}
\begin{center}

\begin{tikzpicture}[scale = 0.9]
	
	\draw[fill=blue!20] (0,4) rectangle (6,3);
	\draw[fill=green!50!black!20] (0,1) rectangle (6,0);
	
	\draw (3,3) -- (3,4);
	\draw (0,0) rectangle (6,4);
	
	\node[text = blue!90!black!90] at (-1,3.5) {$\partial V(G_n)$};
	\node[text = green!20!black!90] at (-1,0.5) {$\partial V(G_0)$};
	
	\node at (1.5,4.5) {$S$};
	\node at (4.5,4.5) {$\partial V(G)$};
	
	\node at (1.5,3.5) {$u = \phi$};
	\node at (4.5,3.5) {$u = 0$};
	
	\draw[->] (1.5,2.5) -- (1.5,1.5);
	\draw[->] (3,2.5) -- (3,1.5);
	\draw[->] (4.5,2.5) -- (4.5,1.5);
	
\end{tikzpicture}

\caption{Schematic overview of the proof of Theorem \ref{thm:explicitalgorithm}.  Arrows show the direction of harmonic continuation.}

\end{center}
\end{figure}
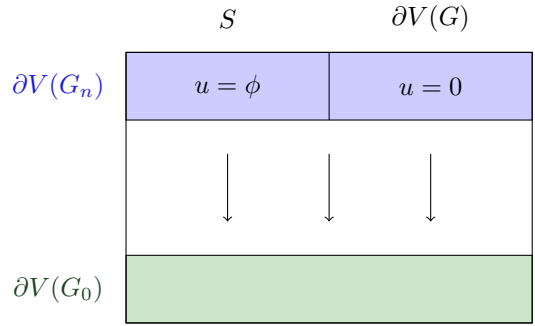

Suppose that we have a standard-form layerable filtration of $G'$ given by
\[
\varnothing \subseteq (G_0,L_0) \subseteq \dots \subseteq (G_n,L_n) = (G',L).
\]
Our harmonic continuation will start with a function defined on $\partial V(G_n)$ and then extend it inward to $\partial V(G_{n-1})$.  At step $j$ (counting backwards from $n$), we will have a partially defined harmonic function whose domain is the roughly the \emph{complement} of $G_j$.

To make this idea precise, we define a complementary sub-$\partial$-graph $H_j$ by
\begin{align}
V(H_n) &= V(G') \setminus V^\circ(G_n) \label{eq:complement} \\
E(H_n) &= E(G') \setminus E(G_n) \nonumber \\
V^\circ(H_n) &= V(G') \setminus V(G_n) \nonumber \\
\partial V(H_n) &= \partial V(G_n). \nonumber
\end{align}
See Figure \ref{fig:complementaryfiltration}.  The following facts follow from direct casework:
\begin{itemize}
	\item If $G_{j+1}$ is obtained from $G_j$ by adjoining a boundary spike, then $H_{j+1}$ is obtained from $H_j$ by contracting a boundary spike.
	\item If $G_{j+1}$ is obtained from $G_j$ by adjoining a boundary edge, then $H_{j+1}$ is obtained from $H_j$ by deleting a boundary edge.
	\item The graph $H_n$ consists of isolated boundary vertices.
	\item The graph $H_0$ has the same vertex and edge sets as $G$ and $\overline{G}$, but a different choice of boundary vertices.  Specifically, $\partial V(H_0)$ is equal to $\partial V(G_0)$ rather than $\partial V(G_n) = \partial V(G)$.
\end{itemize}
This shows that $H_n \subseteq H_{n-1} \subseteq \dots \subseteq H_0$ is a standard-form layerable filtration of $H_0$.

\begin{figure}
\begin{center}
\begin{tikzpicture}[scale=0.7]

	\begin{scope}[shift={(0,0)}]
		\node at (0.5,3) {$G_0$};
	
		\node[bd] (1) at (0,0) {};
		\node[bd] (2) at (0,1) {};
		\node[bd] (3) at (0.5,2) {};
	\end{scope}
	
	\begin{scope}[shift={(0,-4)}]
		\node at (0.5,3) {$H_0$};

		\node[bd] (1) at (0,0) {};
		\node[bd] (2) at (0,1) {};
		\node[bd] (3) at (0.5,2) {};
		\node[int] (4) at (1,0) {};
		\node[int] (5) at (1,1) {};

		\draw (1) to (2);
		\draw (2) to (3);
		\draw (2) to (5);
		\draw (1) to (5);
		\draw (3) to (5);
		\draw (4) to (5);
		\draw (1) to (4);
		\draw (3) to [bend right = 25] (4);
	\end{scope}

	\begin{scope}[shift={(2,0)}]
		\node at (0.5,3) {$G_1$};
	
		\node[bd] (1) at (0,0) {};
		\node[bd] (2) at (0,1) {};
		\node[bd] (3) at (0.5,2) {};
		
		\draw (1) to (2);
	\end{scope}
	
	\begin{scope}[shift={(2,-4)}]
		\node at (0.5,3) {$H_1$};

		\node[bd] (1) at (0,0) {};
		\node[bd] (2) at (0,1) {};
		\node[bd] (3) at (0.5,2) {};
		\node[int] (4) at (1,0) {};
		\node[int] (5) at (1,1) {};

		\draw (2) to (3);
		\draw (2) to (5);
		\draw (1) to (5);
		\draw (3) to (5);
		\draw (4) to (5);
		\draw (1) to (4);
		\draw (3) to [bend right = 25] (4);
	\end{scope}

	\begin{scope}[shift={(4,0)}]
		\node at (0.5,3) {$G_2$};
	
		\node[bd] (1) at (0,0) {};
		\node[bd] (2) at (0,1) {};
		\node[bd] (3) at (0.5,2) {};
		
		\draw (1) to (2);
		\draw (2) to (3);
	\end{scope}
	
	\begin{scope}[shift={(4,-4)}]
		\node at (0.5,3) {$H_2$};

		\node[bd] (1) at (0,0) {};
		\node[bd] (2) at (0,1) {};
		\node[bd] (3) at (0.5,2) {};
		\node[int] (4) at (1,0) {};
		\node[int] (5) at (1,1) {};

		\draw (2) to (5);
		\draw (1) to (5);
		\draw (3) to (5);
		\draw (4) to (5);
		\draw (1) to (4);
		\draw (3) to [bend right = 25] (4);
	\end{scope}

	\begin{scope}[shift={(6,0)}]
		\node at (0.5,3) {$G_3$};
	
		\node[bd] (1) at (0,0) {};
		\node[int] (2) at (0,1) {};
		\node[bd] (3) at (0.5,2) {};
		\node[bd] (5) at (1,1) {};
		
		\draw (1) to (2);
		\draw (2) to (3);
		\draw (2) to (5);
	\end{scope}
	
	\begin{scope}[shift={(6,-4)}]
		\node at (0.5,3) {$H_3$};

		\node[bd] (1) at (0,0) {};
		\node[bd] (3) at (0.5,2) {};
		\node[int] (4) at (1,0) {};
		\node[bd] (5) at (1,1) {};

		\draw (1) to (5);
		\draw (3) to (5);
		\draw (4) to (5);
		\draw (1) to (4);
		\draw (3) to [bend right = 25] (4);
	\end{scope}

	\begin{scope}[shift={(8,0)}]
		\node at (0.5,3) {$G_4$};
	
		\node[bd] (1) at (0,0) {};
		\node[int] (2) at (0,1) {};
		\node[bd] (3) at (0.5,2) {};
		\node[bd] (5) at (1,1) {};
		
		\draw (1) to (2);
		\draw (2) to (3);
		\draw (2) to (5);
		\draw (1) to (5);
	\end{scope}
	
	\begin{scope}[shift={(8,-4)}]
		\node at (0.5,3) {$H_4$};

		\node[bd] (1) at (0,0) {};
		\node[bd] (3) at (0.5,2) {};
		\node[int] (4) at (1,0) {};
		\node[bd] (5) at (1,1) {};
	
		\draw (3) to (5);
		\draw (4) to (5);
		\draw (1) to (4);
		\draw (3) to [bend right = 25] (4);
	\end{scope}

	\begin{scope}[shift={(10,0)}]
		\node at (0.5,3) {$G_5$};
	
		\node[int] (1) at (0,0) {};
		\node[int] (2) at (0,1) {};
		\node[bd] (3) at (0.5,2) {};
		\node[bd] (4) at (1,0) {};
		\node[bd] (5) at (1,1) {};
		
		\draw (1) to (2);
		\draw (2) to (3);
		\draw (2) to (5);
		\draw (1) to (5);
		\draw (1) to (4);
	\end{scope}
	
	\begin{scope}[shift={(10,-4)}]
		\node at (0.5,3) {$H_5$};

		\node[bd] (3) at (0.5,2) {};
		\node[bd] (4) at (1,0) {};
		\node[bd] (5) at (1,1) {};
	
		\draw (3) to (5);
		\draw (4) to (5);
		\draw (3) to [bend right = 25] (4);
	\end{scope}

	\begin{scope}[shift={(12,0)}]
		\node at (0.5,3) {$G_6$};
	
		\node[int] (1) at (0,0) {};
		\node[int] (2) at (0,1) {};
		\node[bd] (3) at (0.5,2) {};
		\node[bd] (4) at (1,0) {};
		\node[bd] (5) at (1,1) {};
		
		\draw (1) to (2);
		\draw (2) to (3);
		\draw (2) to (5);
		\draw (1) to (5);
		\draw (1) to (4);
		\draw (3) to (5);
	\end{scope}
	
	\begin{scope}[shift={(12,-4)}]
		\node at (0.5,3) {$H_6$};

		\node[bd] (3) at (0.5,2) {};
		\node[bd] (4) at (1,0) {};
		\node[bd] (5) at (1,1) {};
		
		\draw (4) to (5);
		\draw (3) to [bend right = 25] (4);
	\end{scope}

	\begin{scope}[shift={(14,0)}]
		\node at (0.5,3) {$G_7$};
	
		\node[int] (1) at (0,0) {};
		\node[int] (2) at (0,1) {};
		\node[bd] (3) at (0.5,2) {};
		\node[bd] (4) at (1,0) {};
		\node[bd] (5) at (1,1) {};
		
		\draw (1) to (2);
		\draw (2) to (3);
		\draw (2) to (5);
		\draw (1) to (5);
		\draw (1) to (4);
		\draw (3) to (5);
		\draw (4) to (5);
	\end{scope}
	
	\begin{scope}[shift={(14,-4)}]
		\node at (0.5,3) {$H_7$};

		\node[bd] (3) at (0.5,2) {};
		\node[bd] (4) at (1,0) {};
		\node[bd] (5) at (1,1) {};
		
		\draw (3) to [bend right = 25] (4);
	\end{scope}
	
	\begin{scope}[shift={(16,0)}]
		\node at (0.5,3) {$G_8$};
	
		\node[int] (1) at (0,0) {};
		\node[int] (2) at (0,1) {};
		\node[bd] (3) at (0.5,2) {};
		\node[bd] (4) at (1,0) {};
		\node[bd] (5) at (1,1) {};
		
		\draw (1) to (2);
		\draw (2) to (3);
		\draw (2) to (5);
		\draw (1) to (5);
		\draw (1) to (4);
		\draw (3) to (5);
		\draw (4) to (5);
		\draw (3) to [bend right = 25] (4);
	\end{scope}
	
	\begin{scope}[shift={(16,-4)}]
		\node at (0.5,3) {$H_8$};

		\node[bd] (3) at (0.5,2) {};
		\node[bd] (4) at (1,0) {};
		\node[bd] (5) at (1,1) {};
	\end{scope}
	
\end{tikzpicture}

\end{center}

\caption{Complementary layerable filtrations as defined by \eqref{eq:complement}.} \label{fig:complementaryfiltration}

\end{figure}
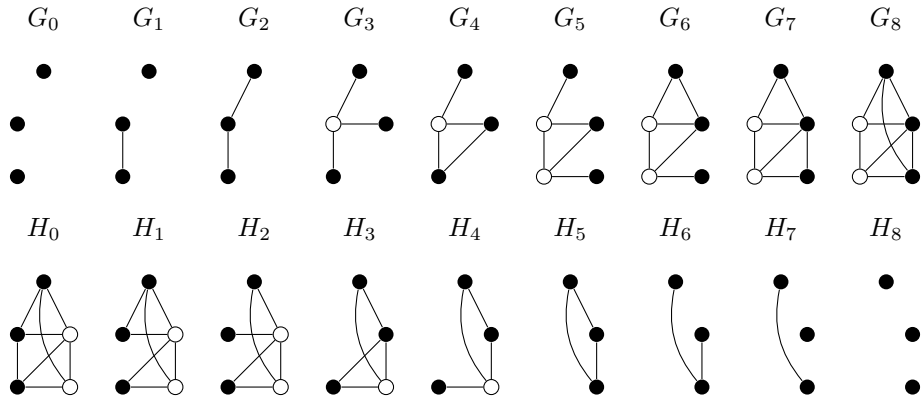

Harmonic continuation inward from the boundary of $G$ toward $\partial V(G_0)$ will correspond to building harmonic extensions through the filtration $H_n$, $H_{n-1}$, \dots   But now that $H_j$ has been defined, we no longer need to work directly with the $\partial$-graphs $G_j$.  We can express $\c U_0(G,L,M)$ in terms of harmonic functions on $H_0$ rather than harmonic functions on $G_n$:  Note that $H_0$ and $G$ only differ in the assignment of boundary vertices; expressing the conditions for $u \in \c U_0(G,L,M)$ in terms of $H_0$ yields
\[
\c U_0(G,L,M) = \{u \in \c U(H_0,L,M) \colon u|_{\partial V(G)} = 0, Lu|_{\partial V(H_0)} = 0\}.
\]
As in Lemma \ref{lem:layeringcoordinates}, we can use the filtration $H_n \subseteq \dots \subseteq H_0$ to parametrize the harmonic functions on $H_0$ in terms of their values on $\partial V(H_n) = \partial V(G) \sqcup S$.  The condition $u|_{\partial V(G)} = 0$ simply says that a subset of our initial parameters will be zero, and thus our harmonic function will be parametrized by the values on $S$ with the values on the rest of $\partial V(H_n)$ set to zero.

Thus, we proceed as follows:  Given a vector $\phi \in M^S$, we extend $\phi$ by zero to a vector in $M^{\partial V(H_n)} = \c U(H_n,L|_{H_n},M)$, then apply the sequence of boundary data transformations associated to the filtration $H_n$, \dots, $H_0$ as in Lemma \ref{lem:layeringcoordinates} to compute a harmonic extension $u$ to $H_0$.  We let $A$ be the transformation $\phi \mapsto Lu|_{\partial V(H_n)}$ that sends $\phi \in M^S$ to the boundary values of $Lu$ for the harmonic extension to $H_0$.  Then the functions $u \in \c U_0(G,L,M)$ correspond to the values of $\phi$ such that $A \phi = 0$, so that
\begin{equation}
\c U_0(G,L,M) \cong \ker(A: M^S \to M^{\partial V(H_n)}).
\end{equation}

The matrix $A$ is given explicitly as follows:  Let $m = |\partial V(H_0)| = |S| + |\partial V(G)|$, choose a consistent labelling for the filtration $\{H_j\}$.  Assume that in the indexing of $\partial V(H_n)$, the vertices in $S$ are indexed first by $1, \dots, s$ and then the vertices of $\partial V(G)$ are indexed by $s + 1$, \dots, $m$.  Let $T_n$, $T_{n-1}$, \dots, $T_0$ be the sequence of boundary data transformations corresponding to the filtration $H_n$, $H_{n-1}$, \dots, $H_0$ (as in Lemma \ref{lem:layeringcoordinates} except with $G_j$ replaced by $H_{n-j}$).  Here $T_n$ is the transformation for the initial network $H_n$, and $T_j$ is the transformation from $H_{j+1}$ to $H_j$.  Then set
\begin{equation}
A = (0_{m \times m}, I_{m \times m}) T_0 T_1 \dots T_n \begin{pmatrix} I_{s \times s} \\ 0_{(2m-s) \times s} \end{pmatrix}. \qedhere \label{eq:matrixA}
\end{equation}
\end{proof}

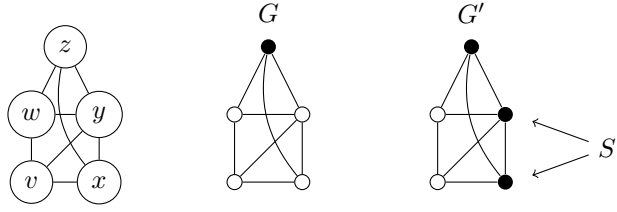
\begin{figure}
\begin{center}
\begin{tikzpicture}[scale=0.9]
	\begin{scope}[shift={(0,0)}]
		\node[circle,draw] (1) at (0,0) {$v$};
		\node[circle,draw] (2) at (0,1) {$w$};
		\node[circle,draw] (3) at (0.5,2) {$z$};
		\node[circle,draw] (4) at (1,0) {$x$};
		\node[circle,draw] (5) at (1,1) {$y$};
		
		\draw (1) to (2);
		\draw (2) to (3);
		\draw (2) to (5);
		\draw (1) to (5);
		\draw (1) to (4);
		\draw (3) to (5);
		\draw (4) to (5);
		\draw (3) to [bend right = 25] (4);
	\end{scope}

	\begin{scope}[shift={(3,0)}]
		\node at (0.5,2.5) {$G$};
	
		\node[int] (1) at (0,0) {};
		\node[int] (2) at (0,1) {};
		\node[bd] (3) at (0.5,2) {};
		\node[int] (4) at (1,0) {};
		\node[int] (5) at (1,1) {};
		
		\draw (1) to (2);
		\draw (2) to (3);
		\draw (2) to (5);
		\draw (1) to (5);
		\draw (1) to (4);
		\draw (3) to (5);
		\draw (4) to (5);
		\draw (3) to [bend right = 25] (4);
	\end{scope}

	\begin{scope}[shift={(6,0)}]
		\node at (0.5,2.5) {$G'$};
	
		\node[int] (1) at (0,0) {};
		\node[int] (2) at (0,1) {};
		\node[bd] (3) at (0.5,2) {};
		\node[bd] (4) at (1,0) {};
		\node[bd] (5) at (1,1) {};
		
		\node (S) at (2.5,0.5) {$S$};
		\draw[->] (S) to (1.4,0.9);
		\draw[->] (S) to (1.4,0.1);
		
		\draw (1) to (2);
		\draw (2) to (3);
		\draw (2) to (5);
		\draw (1) to (5);
		\draw (1) to (4);
		\draw (3) to (5);
		\draw (4) to (5);
		\draw (3) to [bend right = 25] (4);
	\end{scope}

\end{tikzpicture}

\caption{Graphs for Example \ref{ex:explicitalgorithmexample}.}  \label{fig:explicitalgorithmexample}

\end{center}
\end{figure}

\begin{example} \label{ex:explicitalgorithmexample}
Let us apply Theorem \ref{thm:explicitalgorithm} to the critical group of the graph shown in Figure \ref{fig:explicitalgorithmexample}, left.  By Proposition \ref{prop:criticalgrouponeboundary}, it suffices to compute $\c U_0(G,L_{\std}, \Q / \Z)$, where $G$ is the middle $\partial$-graph in the figure with one boundary vertex.  Let $S = \{x,y\}$; then $G' = G_{S \to \partial}$ is the graph in Figure \ref{fig:explicitalgorithmexample}, right.  We let $G_j$ and $H_j$ be the $\partial$-graphs shown earlier in Figure \ref{fig:complementaryfiltration}.  For each $G_j$ and $H_j$, we define $\ell_j$ by labelling the boundary vertices $1$, $2$, $3$ from bottom to top (thus, $1$ corresponds to $x$ or $v$, $2$ corresponds to $y$ or $w$, and $3$ corresponds to $z$).  To compute $A$ by equation \eqref{eq:matrixA}, we write down the transformations for $H_8$, $H_7$, \dots.  Note that because $d = 0$, the initial transformation $T_8$ is the identity.  The next few are
\begin{align*}
T_7 = \begin{pmatrix} 1 & 0 & 0 & 0 & 0 & 0 \\ 0 & 1 & 0 & 0 & 0 & 0 \\ 0 & 0 & 1 & 0 & 0 & 0 \\ 1 & 0 & -1 & 1 & 0 & 0 \\ 0 & 0 & 0 & 0 & 1 & 0 \\ -1 & 0 & 1 & 0 & 0 & 1 \end{pmatrix},
\qquad &
T_6 = \begin{pmatrix} 1 & 0 & 0 & 0 & 0 & 0 \\ 0 & 1 & 0 & 0 & 0 & 0 \\ 0 & 0 & 1 & 0 & 0 & 0 \\ 1 & -1 & 0 & 1 & 0 & 0 \\ -1 & 1 & 0 & 0 & 1 & 0 \\ 0 & 0 & 0 & 0 & 0 & 1 \end{pmatrix}, \\
T_5 = \begin{pmatrix} 1 & 0 & 0 & 0 & 0 & 0 \\ 0 & 1 & 0 & 0 & 0 & 0 \\ 0 & 0 & 1 & 0 & 0 & 0 \\ 0 & 0 & 0 & 1 & 0 & 0 \\ 0 & 1 & -1 & 0 & 1 & 0 \\ 0 & -1 & 1 & 0 & 0 & 1 \end{pmatrix},
\quad &
T_4 = \begin{pmatrix} 1 & 0 & 0 & 1 & 0 & 0 \\ 0 & 1 & 0 & 0 & 0 & 0 \\ 0 & 0 & 1 & 0 & 0 & 0 \\ 0 & 0 & 0 & 1 & 0 & 0 \\ 0 & 0 & 0 & 0 & 1 & 0 \\ 0 & 0 & 0 & 0 & 0 & 1 \end{pmatrix}, 
\end{align*}
and the rest of the transformations are computed similarly.  After some straightforward computation,
\[
A = \begin{pmatrix} 0_{3 \times 3} & I_{3 \times 3} \end{pmatrix} T_0 \dots T_7 T_8 \begin{pmatrix} I_{2 \times 2} \\ 0_{4 \times 2} \end{pmatrix} = \begin{pmatrix} 12 & -9 \\ -15 & 15 \\ 3 & -6 \end{pmatrix}.
\]
By performing integer row and column operations, we can convert $A$ into the Smith normal form
\[
A' = \begin{pmatrix} 3 & 0 \\ 0 & 15 \\ 0 & 0 \end{pmatrix}.
\]
This implies
\[
\c U_0(G, L_{\std}, \Q / \Z) \cong \Z / 3 \times \Z / 15 \cong \Z / 3 \times \Z / 3 \times \Z / 7.
\]

Harmonic continuation also allows us to compute the harmonic functions explicitly from the parameters on $S$.  We will demonstrate this by computing $\c U_0(G, L_{\std}, \Z / 3)$ and $\c U_0(G, L_{\std}, \Z / 5)$.  For a harmonic function $u$, we have from the last equation of Lemma \ref{lem:layeringcoordinates} that
\[
\begin{pmatrix} u(v) \\ u(w) \\ u(z) \\ 0 \\ 0 \\ 0 \end{pmatrix} = T_0 \dots T_7 T_8 \begin{pmatrix} u(x) \\ u(y) \\ u(z) \\ 0 \\ 0 \\ 0 \end{pmatrix}.
\]
From explicit computation of the first two columns of $T_0 \dots T_8$, we find that when $u(z) = 0$, we have
\[
u(v) = 3u(x) - u(y), \quad u(w) = -4 u(x) + 5 u(y).
\]
To compute $\c U_0(G,L_{\std}, \Z / 3)$, we observe that since $A = 0$ mod $3$, every choice of two parameters in $\Z / 3$ on $S$ yields a harmonic function, and we obtain two generators shown in Figure \ref{fig:z3z5}.  Next, to compute $\c U_0(G,L_{\std}, \Z / 5)$, we read off from the matrix $A$ that $u$ must satisfy $3 u(x) - 6 u(y) = 0$, and so a generator is given by taking $u(x) = 2$ and $u(y) = 1$.  The resulting harmonic function is shown in Figure \ref{fig:z3z5}.
\end{example}

\begin{figure}

\begin{center}
\begin{tikzpicture}[scale=0.7]

	\begin{scope}[shift={(0,0)}]
		\node at (-0.5,2.5) {$\Z / 3$};
	
		\node[int] (1) at (0,0) [label = left:$0$] {};
		\node[int] (2) at (0,1) [label = left:$-1$] {};
		\node[bd] (3) at (0.5,2) [label = above:$0$] {};
		\node[int] (4) at (1,0) [label = right:$1$] {};
		\node[int] (5) at (1,1) [label = right:$0$] {};
		
		\draw (1) to (2);
		\draw (2) to (3);
		\draw (2) to (5);
		\draw (1) to (5);
		\draw (1) to (4);
		\draw (3) to (5);
		\draw (4) to (5);
		\draw (3) to [bend right = 25] (4);
	\end{scope}

	\begin{scope}[shift={(4,0)}]
		\node at (-0.5,2.5) {$\Z / 3$};
		
		\node[int] (1) at (0,0) [label=left:$-1$] {};
		\node[int] (2) at (0,1) [label = left:$-1$] {};
		\node[bd] (3) at (0.5,2) [label = above:$0$] {};
		\node[int] (4) at (1,0) [label = right:$0$] {};
		\node[int] (5) at (1,1) [label = right:$1$] {};
		
		\draw (1) to (2);
		\draw (2) to (3);
		\draw (2) to (5);
		\draw (1) to (5);
		\draw (1) to (4);
		\draw (3) to (5);
		\draw (4) to (5);
		\draw (3) to [bend right = 25] (4);
	\end{scope}

	\begin{scope}[shift={(8,0)}]
		\node at (-0.5,2.5) {$\Z / 5$};
		
		\node[int] (1) at (0,0) [label=left:$0$] {};
		\node[int] (2) at (0,1) [label = left:$2$] {};
		\node[bd] (3) at (0.5,2) [label = above:$0$] {};
		\node[int] (4) at (1,0) [label = right:$2$] {};
		\node[int] (5) at (1,1) [label = right:$1$] {};
		
		\draw (1) to (2);
		\draw (2) to (3);
		\draw (2) to (5);
		\draw (1) to (5);
		\draw (1) to (4);
		\draw (3) to (5);
		\draw (4) to (5);
		\draw (3) to [bend right = 25] (4);
	\end{scope}

\end{tikzpicture}

\caption{Generators for $\c U_0(G, L_{\std}, \Z / 3)$ and $\c U_0(G,L_{\std}, \Z / 5)$ on the network $G$ from Example \ref{ex:explicitalgorithmexample}.} \label{fig:z3z5}

\end{center}

\end{figure}

\begin{example}
In \S \ref{sec:CLF}, we used harmonic continuation to compute $\c U_0(\CLF(m,n), L_{\std}, \Q / \Z)$ without using Theorem \ref{thm:explicitalgorithm}.  But in fact, we \emph{could} have applied Theorem \ref{thm:explicitalgorithm}, and it is instructive to see how our method in \S \ref{sec:CLF} can be derived from the ideas in this section.

Let $G = \CLF(m,n)$.  Using the indexing of the vertices from \S \ref{subsec:CLFsetup}, define $S = \{0,1\} \times \{1,\dots,n\}$ and $S' = \{m-1,0\} \times \{1,\dots,n\}$.  Note $G' = G_{S \to \partial}$ is layerable, and a filtration is shown in Figure \ref{fig:CLFfiltration}.
\begin{align*}
\partial V(G_n) &= \partial V(G) \cup S \\
\partial V(G_0) &= \partial V(G) \cup S'.
\end{align*}
The layer-stripping filtration strips away the graph column by column; for $j = 0, \dots, m-1$, it removes the edges from $\{j\} \times \{0,\dots,n\}$ and $\{j+1\} \times \{0,\dots,n\}$.  In each column, it removes the edges from bottom to top or from top to bottom depending on parity.

The filtration suggests a process of harmonic continuation where the initial parameters are the values of $u$ on $S = \{0,1\} \times \{0,\dots,n\}$ and the harmonic continuation moves column by column from left to right in the picture.  In the notation of \ref{subsec:CLFcontinuation}, this means solving for $u$ in terms of $a_0$ and $a_1$ by finding $a_j$ inductively.  In \ref{subsec:CLFcontinuation}, we did not use the same method of bookkeeping as in \S \ref{subsec:layeringcoordinates} and Theorem \ref{thm:explicitalgorithm}, but rather kept track of potential values on two consecutive columns (see Remark \ref{rem:alternativebookkeeping}).

Using the filtration pictured here, $a_{j+1}$ can be found from $a_0,\dots,a_j$ in $2n - 1$ steps corresponding to the $2n - 1$ edges connecting the $j$th and $(j+1)$th columns.  Our approach in \S \ref{subsec:CLFcontinuation} combined all these operations into one step by writing harmonicity in terms of the matrix $E$.  In fact, executing the $2n - 1$ steps for each column amounts to inverting the matrix $E$ (recall $4E^{-1}$ appears in the upper left block of the transformation $\mathbf{T}$ from \S \ref{subsec:CLFcontinuation}).  The reason we did not have to do this was that we avoided dealing directly with $E^{-1}$ in the proof of Lemma \ref{lem:matrixalgebra2}.

After solving for $a_2$, \dots, $a_{m-1}$ through harmonic continuation, the method of Theorem \ref{thm:explicitalgorithm} requires us to find the values of the intitial parameters $a_0$ and $a_1$ that will guarantee $Lu|_{\partial V(G_0)} = 0$.  Recall $\partial V(G_0) = \partial V(G) \cup S'$.  Because the matrix $\mathbf{T}$ was constructed to check harmonicity on one column of vertices, the condition $L_{\std}u|_{S'} = 0$ is equivalent to
\[
\mathbf{T} \begin{pmatrix} a_{m-1} \\ a_{m-2} \end{pmatrix} = \begin{pmatrix} a_0 \\ a_{m-1} \end{pmatrix}, \qquad 
\mathbf{T} \begin{pmatrix} a_0 \\ a_{m-1} \end{pmatrix} = \begin{pmatrix} a_1 \\ a_0 \end{pmatrix}.
\]
Since the bottom block row of $\mathbf{T}$ is $(I, 0)$, it suffices to check $\mathbf{T}^2 (a_{m-1}, a_{m-2})^t = (a_1,a_0)^t$.  In short, $Lu|_{S'} = 0$ amounts to the fixed-point condition $\mathbf{T}^m (a_1,a_0)^t = (a_1,a_0)^t$ in \S \ref{subsec:CLFcontinuation}.  Meanwhile, the condition $L_{\std}u|_{\partial V(G)} = 0$ was checked through our computation of $M_1$ in \S \ref{subsec:algebraiccomputation}.
\end{example}

\newcommand\makeCLFnode[4] { 
	\node[#4] (#2x#3) at (#2, {(1 + (-1)^#2) * (#1 + 0.5) - 2 * (-1)^#2 * #3}) {};
}

\newcommand\makeCLFgraphstrippedA[4] { 
	\draw[blue,dashed,thick] (0,0) -- (0,2 * #2 + 1);
	\draw[blue,dashed,thick] (#1,0) -- (#1, 2 * #2 + 1);

	\foreach \i [evaluate={\ii = int(\i - 1)}, evaluate={\iii = int(\i + 1)}] in {0,...,#1} {
		\foreach \j [evaluate={\k = int(#2 - \j)}] in {0,...,#2} {
			\ifthenelse{
				\j = 0 \OR
				\i = 0 \OR \( \i = 1 \AND #3 = 0 \) \OR
				\i = #1 \OR
				\( \iii = #1 \AND \i = #3 \) \OR
				\(\i = #3 \AND \( \j < #4 \OR \j = #4\) \) \OR
				\( \ii = #3 \AND \( \k > #4 \OR \k = #4 \)\)
			}{
				\makeCLFnode{#2}{\i}{\j}{bd}
			}{
				\ifthenelse {
					\i > #3
				}{
					\makeCLFnode{#2}{\i}{\j}{int}
				}{}
			}
		}
	}
	\foreach \i [evaluate={\x = int(\i - 1)}] in {1,...,#1} {
		\foreach \j [evaluate={\y = int(#2 - \j)}] in {0,...,#2} {
			\ifthenelse{ \x > #3 \OR \( \x = #3 \AND \(\y < #4 \OR \y = #4\)\) }{ \draw (\i x\j) to (\x x\y); }{}
		}
		\foreach \j [evaluate={\y = int(#2 - \j + 1)}] in {1,...,#2} {
			\ifthenelse{ \x > #3 \OR \( \x = #3 \AND \(\y < #4 \OR \y = #4\)\) }{ \draw (\i x\j) to (\x x\y); }{}
		}
	}
}

\newcommand\makeCLFgraphstrippedB[4] { 
	\draw[blue,dashed,thick] (0,0) -- (0,2 * #2 + 1);
	\draw[blue,dashed,thick] (#1,0) -- (#1, 2 * #2 + 1);

	\foreach \i [evaluate={\ii = int(\i - 1)}, evaluate={\iii = int(\i + 1)}] in {0,...,#1} {
		\foreach \j [evaluate={\k = int(#2 - \j)}] in {0,...,#2} {
			\ifthenelse{
				\j = 0 \OR
				\i = 0 \OR \( \i = 1 \AND #3 = 0 \) \OR
				\i = #1 \OR
				\( \iii = #1 \AND \i = #3 \) \OR
				\(\i = #3 \AND \( \j < #4 \OR \j = #4\) \) \OR
				\( \ii = #3 \AND \( \k > #4 \OR \k = #4 \)\)
			}{
				\makeCLFnode{#2}{\i}{\j}{bd}
			}{
				\ifthenelse {
					\i > #3
				}{
					\makeCLFnode{#2}{\i}{\j}{int}
				}{}
			}
		}
	}
	\foreach \i [evaluate={\x = int(\i - 1)}] in {1,...,#1} {
		\foreach \j [evaluate={\y = int(#2 - \j)}] in {0,...,#2} {
			\ifthenelse{ \x > #3 \OR \( \x = #3 \AND \y < #4\) }{ \draw (\i x\j) to (\x x\y); }{}
		}
		\foreach \j [evaluate={\y = int(#2 - \j + 1)}] in {1,...,#2} {
			\ifthenelse{ \x > #3 \OR \( \x = #3 \AND \(\y < #4 \OR \y = #4\)\) }{ \draw (\i x\j) to (\x x\y); }{}
		}
	}
}

\begin{figure}
\begin{center}

\begin{tikzpicture}[scale=0.4]

	\foreach \p in {0,...,3} {
		\foreach \q in {1,...,2} {
			\begin{scope}[shift={(-12 * \q + 6, -\p * 7)}]
				\makeCLFgraphstrippedB{4}{2}{\p}{\q}
			\end{scope}
		}
		
		\foreach \q in {0,...,2} {
			\begin{scope}[shift={(-12 * \q, -\p * 7)}]
				\makeCLFgraphstrippedA{4}{2}{\p}{\q}
			\end{scope}
		}
	}

\end{tikzpicture}

\caption{A layerable filtration of $\CLF(4,2)$ with $S = \{0,1\} \times \{1,2\}$ changed to boundary vertices.  The vertices on the left and right sides of each picture are identified.} \label{fig:CLFfiltration}

\end{center}
\end{figure}
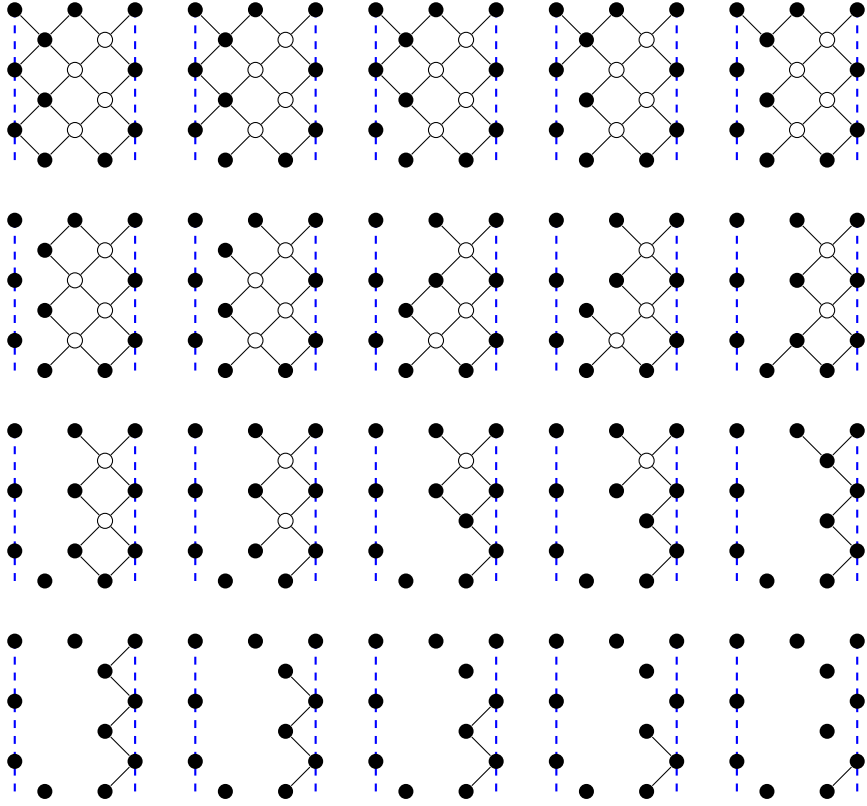

\subsection{Application to the Critical Group and Eigenvalues}

In the case of graphs without boundary, Theorem \ref{thm:explicitalgorithm} can be applied to compute the critical group or the eigenvalues of the Laplacian.  Even when it is not practical to use Theorem \ref{thm:explicitalgorithm} for the complete computation, it still provides a priori upper bounds on the number of invariant factors of the critical group and the mutliplicity of Laplacian eigenvalues (Corollaries \ref{cor:invariantfactorsbound} and \ref{cor:multiplicitybound} below).  Though we do not claim these bounds are always sharp, they have the advantage of being computed geometrically without writing down any matrices.  Moreover, we will give several infinite families for which they \emph{are} sharp.

\begin{corollary} \label{cor:invariantfactorsbound}
Suppose that $G$ is a graph without boundary and that $G_{S \to \partial}$ is layerable.  Then $\Crit(G)$ has at most $|S| - 1$ invariant factors.
\end{corollary}

\begin{proof}
Choose some $x \in S$, and let $G' = G_{x \to \partial}$.  Then by Proposition \ref{prop:criticalgrouponeboundary}, we have
\[
\Crit(G) \cong \c U_0(G', L_{\std}, \Q / \Z).
\]
Let $S' = S \setminus \{x\}$.  Then $(G')_{S' \to \partial}$ is layerable, so by Theorem \ref{thm:explicitalgorithm}, there is a matrix $A$ with $|S'|$ columns such that
\[
\c U_0(G', L_{\std}, \Q / \Z) \cong \ker_{\Q / \Z}(A).
\]
Recall that a matrix $A'$ is in Smith normal form if the only nonzero entries are on the diagonal, and the diagonal entries $a_1$, $a_2$, \dots, $a_n$ satisfy $a_j | a_{j+1}$.  It is a standard fact that if $A$ is a matrix with entries in $\Z$, then there exist square matrices $U$ and $V$ invertible over $\Z$ such that $A' = UAV$ is in Smith normal form (see \cite[\S 12.1 Exercises 16-19]{DummitandFoote}).  By interpreting $U$ and $V$ as changes of coordinates, we can see that $\ker_{\Q / \Z}(A) \cong \ker_{\Q / \Z}(A')$.  Thus, the invariant factors for $\Crit(G)$ are found from the diagonal entries of $A'$.  Since $A'$ has $|S'| = |S| - 1$ columns, there are at most $|S| - 1$ invariant factors.
\end{proof}

\begin{example}
Consider the complete graph $K_n$.  Note that we can make the graph $K_n$ into a layerable $\partial$-graph by changing $n - 1$ of the $n$ vertices to boundary vertices.  On the other hand, it is shown in \cite[Lemma, p.\ 278]{Lor1} that $\Crit(K_n) \cong (\Z / n)^{n-2}$, which has $n - 2 = (n - 1) - 1$ invariant factors.
\end{example}

\begin{example}
Let $Q_n$ be the $1$-skeleton of the $k$-dimensional cube, described explicitly by $V = \{0,1\}^n$ with $x \sim y$ if and only if $x$ and $y$ have exactly $n - 1$ coordinates equal to each other.

Let $S = \{0\} \times \{0,1\}^{n-1}$ and $T = \{1\} \times \{0,1\}^{n-1}$.  Then $(Q_n)_{S \to \partial}$ is layerable, as we will verify by reducing it to the empty graph through a sequence of layer-stripping operations.  All the edges between vertices in $S$ are boundary edges, so we can delete them.  After that, all edges from $S$ to $T$ are boundary spikes, so we can contract them, and then the vertices in $T$ become boundary vertices.  Finally, we can delete all the edges between vertices in $T$ and we are left with $T$ as a set of isolated boundary vertices.

Since $(Q_n)_{S \to \partial}$ is layerable, Corollary \ref{cor:invariantfactorsbound} shows that $\Crit(Q_n)$ has at most $|S| - 1 = 2^{n-1} - 1$ invariant factors.  This bound is sharp | it was shown in \cite[Theorem 1.1]{Bai} that there are exactly $2^{n-1} - 1$ invariant factors.
\end{example}

\begin{corollary} \label{cor:multiplicitybound}
Suppose $G$ is a graph without boundary and $G_{S \to \partial}$ is layerable.  Then for any $\R^\times$-network $(G,L)$, every eigenvalue of the generalized Laplacian $L$ has multiplicity at most $|S|$.  In particular, this holds when $L$ is the standard adjacency matrix or standard Laplacian.
\end{corollary}

\begin{proof}
Consider an $\R^\times$-network $(G,L)$.  Recall that $L$ is symmetric, hence diagonalizable over $\R$.  Let $\lambda \in \R$.  Since $G_{S \to \partial}$ is layerable, by Theorem \ref{thm:explicitalgorithm}, there exists a matrix $A$ with $|S|$ columns such that
\[
\c U_0(G, \lambda - L, \R) \cong \ker_{\R}(A).
\]
This implies that the $\lambda$-eigenspace of $L$ has dimension at most $|S|$.
\end{proof}

\begin{example}
We mentioned that for the complete graph $K_n$, one must assign $n - 1$ boundary vertices.  The adjacency matrix of $K_n$ has all entries equal to $1$, and thus has the eigenvalue $0$ with multiplicity $n - 1$.
\end{example}

\begin{example}
Let $C_n$ by the $n$-cycle graph.  Observe that we can make $C_n$ into a layerable $\partial$-graph by changing two adjacent vertices to boundary vertices.  This implies that every eigenvalue of the adjacency matrix has multiplicity $\leq 2$.  The adjacency matrix is $\Sigma + \Sigma^{-1}$, where $\Sigma$ is the permutation matrix representing the $n$-cycle.  Since the eigenvalues of $\Sigma$ are $\{ e^{2\pi i k/n}\}_{k=0}^{n-1}$, the eigenvalues of $\Sigma + \Sigma^{-1}$ are $\{2 \cos(2\pi k /n)\}_{k=0}^{n-1}$.   In particular, every eigenvalue $\lambda \neq \pm 2$ has multiplicity $2$.
\end{example}

\subsection{Infinite Layerable Filtrations} \label{subsec:infinitefiltration}

In this section, we generalize Proposition \ref{prop:filtration} to infinite $\partial$-graphs, allowing an infinite sequence of simple layerable extensions.  We use the following auxiliary definition:

\begin{definition} \label{def:harmonicsubgraphIU}
Suppose that $\{G_\alpha\}$ is a collection of sub-$\partial$-graphs of a given $\partial$-graph $G$.  Define $\bigcup_\alpha G_\alpha$ by
\[
V\left(\bigcup_\alpha G_\alpha \right) = \bigcup_\alpha V(G_\alpha), \quad V^\circ \left(\bigcup_\alpha G_\alpha \right) = \bigcup_\alpha V^\circ(G_\alpha), \quad E \left(\bigcup_\alpha G_\alpha \right) = \bigcup_\alpha E(G_\alpha).
\]
The definition for $\bigcap_\alpha G_\alpha$ is the same with ``$\cup$'' replaced by ``$\cap$.''  One checks straightforwardly from Definition \ref{def:sub-d-graph} that $\bigcup_\alpha G_\alpha$ and $\bigcap_\alpha G_\alpha$ are sub-$\partial$-graphs.
\end{definition}

\begin{definition}
We say that $G'$ is a layerable extension of $G$ if there is a sequence of sub-$\partial$-graphs
\[
G = G_0 \subseteq G_1 \subseteq \dots
\]
such that $G' = \bigcup_\alpha G_j$.
\end{definition}

In order to extend Proposition \ref{prop:filtration} to the infinite case, we need to be able to take limits of our algebraic functors along increasing sequences of sub-$\partial$-graphs.  To do this, we use the notions of \emph{direct (inductive) limits} and \emph{inverse (projective) limits} of $R$-modules.  For background, see \cite[Exercise 7.6.8, 10.3.25-26]{DummitandFoote}, \cite[Exercises 2.14-19]{AtMac}.

\begin{lemma}
Let $(G_0,L_0) \subseteq (G_1,L_1) \subseteq \dots$ be a sequence of subnetworks of $(G',L')$ and assume that $\bigcup_j G_j = G'$.  Then
\begin{enumerate}
	\item $\Upsilon(G',L')$ is (isomorphic to) the direct limit of the sequence
	\[
	\Upsilon(G_0,L_0) \to \Upsilon(G_1,L_1) \to \dots
	\]
	\item $\c U(G',L',M)$ is (isomorphic to) the inverse limit of the sequence
	\[
	\dots \to \c U(G_1,L_1,M) \to \c U(G_0,L_0,M).
	\]
	\item $\c U_0(G',L',M)$ is (isomorphic to) the direct limit of the sequence
	\[
	\c U_0(G_0,L_0,M) \to \c U_0(G_1,L_1,M) \to \dots
	\]
\end{enumerate}
\end{lemma}

\begin{proof}
To prove (1), let $Y$ be the direct limit of the sequence $\Upsilon(G_0,L_0)$.  Note that we have a short exact sequence of directed systems
\[
\begin{tikzcd}
0 \arrow{r} & L_0(RV^\circ(G_0)) \arrow{r} \arrow{d}{\alpha_{0,1}} & RV(G_0) \arrow{r} \arrow{d}{\beta_{0,1}} & \Upsilon(G_0,L_0) \arrow{r} \arrow{d}{\gamma_{0,1}} & 0 \\
0 \arrow{r} & L_1(RV^\circ(G_1)) \arrow{r} \arrow{d}{\alpha_{1,2}} & RV(G_1) \arrow{r} \arrow{d}{\beta_{1,2}} & \Upsilon(G_1,L_1) \arrow{r} \arrow{d}{\gamma_{1,2}} & 0 \\
 & \vdots & \vdots & \vdots &
\end{tikzcd}
\]
where the vertical maps are the obvious ones obtained from the inclusion $V(G_n) \to V(G_{n+1})$ and $V^\circ(G_n) \to V^\circ(G_{n+1})$.  Moreover, we have maps
\[
\begin{tikzcd}
0 \arrow{r} & L_n(RV^\circ(G_n)) \arrow{r} \arrow{d}{\alpha_n} & RV(G_n) \arrow{r} \arrow{d}{\beta_n} & \Upsilon(G_n,L_n) \arrow{r} \arrow{d}{\gamma_n} & 0 \\
0 \arrow{r} & L'(RV^\circ(G')) \arrow{r} & RV(G') \arrow{r} & \Upsilon(G',L') \arrow{r}  & 0,
\end{tikzcd}
\]
where the vertical arrows are obtained from the inclusion $G_n \to G$.  These maps satisfy $\alpha_n \circ \alpha_{m,n} = \alpha_m$ for $m < n$, and the same holds for $\beta$ and $\gamma$.  Therefore, the universal property of direct limits \cite[Exercise 2.16]{AtMac} gives us maps
\[
\begin{tikzcd}
0 \arrow{r} & \varinjlim L_n(RV^\circ(G_n)) \arrow{r} \arrow{d}{\alpha} & \varinjlim RV(G_n) \arrow{r} \arrow{d}{\beta} & \varinjlim \Upsilon(G_n,L_n) \arrow{r} \arrow{d}{\gamma} & 0 \\
0 \arrow{r} & L'(RV^\circ(G')) \arrow{r} & RV(G') \arrow{r} & \Upsilon(G',L') \arrow{r}  & 0.
\end{tikzcd}
\]
Because the direct limit is an exact functor \cite[Exercise 2.19]{AtMac}, the top row of this diagram is exact.  The bottom row is exact by construction of $\Upsilon$.  We easily see that the first two vertical maps are isomorphisms since $V^\circ(G') = \bigcup_n V^\circ(G_n)$ and $V(G') = \bigcup_n V(G_n)$.  Therefore, the five-lemma implies that the third vertical map is an isomorphism.  This completes the proof of (1).

The statement (2) follows by applying the $\Hom(-,M)$ functor to (1) because whenever $M_0 \to M_1 \to \dots$ is a sequence of $R$-modules, there is a natural isomorphism
\[
\Hom(\varinjlim M_n, M) \cong \varprojlim \Hom(M_n,M).
\]

The proof of (3) is symmetrical to the proof of (1).  Instead of using the short exact sequence
\[
0 \to L(RV^\circ(G)) \to RV(G) \to \Upsilon(G,L) \to 0
\]
for each network $G = G_n$ or $G = G'$, we use the short exact sequence
\[
0 \to \c U_0(G,L,M) \to RV^\circ \otimes M \xrightarrow{L \otimes \id} (L \otimes \id)(RV^\circ \otimes M) \to 0.
\]
Here $RV^\circ \otimes M$ is viewed as the module of finitely supported functions $V^\circ \to M$.  In the last step, we apply the five lemma to show that the \emph{first} map out of three is an isomorphism rather than the \emph{last} map as in (1).
\end{proof}

\begin{proposition} \label{prop:infinitefiltration}
Suppose that a $R^\times$-network $(G',L')$ is a layerable extension of $(G,L)$ through the filtration $\{(G_j,L_j)\}_{j=0}^\infty$.  Let $S \subseteq V(G')$ be the set of vertices which are adjoined as isolated boundary vertices at some step of the filtration.  Then the inclusion map $(G,L) \to (G',L')$ induces isomorphisms
\begin{align*}
\Upsilon(G',L') &\cong \Upsilon(G,L) \oplus RS \\
\c U(G',L',M) &\cong \c U(G,L,M) \times M^S \\
\c U_0(G',L',M) &\cong \c U_0(G,L,M).
\end{align*}
\end{proposition}

\begin{proof}
Let $S_n = S \cap V(G_n)$.  Then by Proposition \ref{prop:filtration}, the maps $\Upsilon(G,L) \oplus RS_n \to \Upsilon(G_n,L_n)$ are isomorphisms.  By passing to the direct limit, we see that $\Upsilon(G,L) \oplus RS \to \Upsilon(G',L')$ is an isomorphism.  The arguments for the other two statements are similar.
\end{proof}

\section{Functorial Properties of Layer-Stripping} \label{sec:layeringfunctor}

\subsection{Unramified $\partial$-graph Morphisms and Layer-Stripping}

In this section, we will show that layer-stripping pulls back through unramified $\partial$-graph morphisms.  Here it is convenient to take the perspective of removing things from a $\partial$-graph rather than adding things, layer-stripping operations rather than layerable extensions, and decreasing rather than increasing filtrations.  Recall that the layer-stripping operations are \emph{deleting an isolated boundary vertex}, \emph{contracting a boundary spike}, and \emph{deleting a boundary vertex}.  We must also define unramified morphisms and preimages of sub-$\partial$-graphs.

\begin{definition} \label{def:unramified}
We say a $\partial$-graph morphism $f\colon G \to H$ is an {\bf unramified} if $\deg(f,x) = 1$ for all $x \in V^\circ(G)$ and $\deg(f,x) \leq 1$ for all $x \in \partial V(G)$.  We apply the same terminology to $R$-network morphisms.
\end{definition}

Lemma \ref{lem:categorydegree} implies that $\partial$-graphs and unramified morphisms form a category, which we will denote $\partial\text{\cat{-graph}}_{\text{unrm}}$.  The category of $R$-networks and unramified morphisms will be denoted $R\text{\cat{-net}}_{\text{unrm}}$.  The full subcategory of finite $\partial$-graphs or networks will be denoted by a superscript ``$0$.''  We remark that covering maps and inclusions of sub-$\partial$-graphs are unramified.  Moreover, restricting an unramified morphism to a sub-$\partial$-graph yields another unramified morphism.

\begin{definition}
If $f\colon G \to H$ is $\partial$-graph morphism and $H'$ is a sub-$\partial$-graph of $H$, then we can define the {\bf pullback} or {\bf preimage $f^{-1}(H')$} as the $\partial$-graph given by
\begin{align*}
V(f^{-1}(H')) \sqcup E(f^{-1}(H')) &= f^{-1}(V(H') \sqcup E(H')), \\
V^\circ(f^{-1}(H')) &= f^{-1}(V^\circ(H')) \cap V^\circ(G).
\end{align*}
Straightforward casework verifies that $f^{-1}(H')$ is a sub-$\partial$-graph of $G$.
\end{definition}

\begin{lemma} \label{lem:layerstrippingfunctoriality}
Suppose $f\colon G \to H$ is an unramified morphism between finite $\partial$-graphs.  Suppose $H_2 \subseteq H_1 \subseteq H$ are harmonic sub-$\partial$-graphs.  If $H_2$ is obtained from $H_1$ by a layer-stripping operation, then $f^{-1}(H_2)$ is obtained from $f^{-1}(H_1)$ by a sequence of layer-stripping operations.
\end{lemma}

\begin{proof}
Suppose that $H_2$ is obtained from $H_1$ by deleting a boundary edge $e$.  Since $f$ must map interior vertices to interior vertices, any preimage of a boundary vertex is a boundary vertex.  Thus, $f^{-1}(e)$ consists of boundary edges, so $f^{-1}(H_2)$ is obtained from $f^{-1}(H_1)$ by deleting boundary edges.

Suppose that $H_2$ is obtained from $H_1$ by deleting an isolated boundary vertex $x$.  Now $f^{-1}(x)$ may contain boundary vertices of $f^{-1}(H_2)$ as well as edges which are collapsed by the map $f$ into the vertex $x$.  Every edge in $f^{-1}(x)$ is a boundary edge since its endpoints are also in $f^{-1}(x)$ (Definition \ref{def:dgraphmorphism} (4)), and because $f$ maps interior vertices to interior vertices (Definition \ref{def:dgraphmorphism} (2)).  Thus we can obtain $f^{-1}(H_2)$ by deletlng the boundary edges in $f^{-1}(x)$, and then deleting the vertices in $f^{-1}(x)$, which are now isolated boundary vertices.

Suppose that $H_2$ is obtained from $H_1$ by contracting a boundary spike $e$ with boundary endpoint $x = e_+$ and interior endpoint $y = e_-$.  Then we obtain $f^{-1}(H_2)$ from $f^{-1}(H_1)$ in several steps.  First, any edges in $f^{-1}(x)$ are boundary edges, and we can delete them.  Second, if there are any isolated boundary vertices in $f^{-1}(x)$, we delete them.  At this point, any vertex left in $f^{-1}(x)$ is attached to an edge in $f^{-1}(e)$ and no other edges.

Now consider each edge $e' \in f^{-1}(e)$.  Note $e_+' \in f^{-1}(x)$ is a boundary vertex.  If $e_-'$ is also a boundary vertex, then we delete $e'$ as a boundary edge and then delete $e_+'$ as an isolated boundary vertex.  If $e_-'$ is an interior vertex, then we can contract $e'$ as a spike, which will automatically remove $e_+' \in f^{-1}(x)$.  These steps will allow us to obtain $f^{-1}(H_2)$ from $f^{-1}(H_1)$.  Indeed, all vertices and edges in $f^{-1}(x)$ have been removed and all edges in $f^{-1}(e)$ have been removed.  Moreover, all vertices in $f^{-1}(y)$ are now boundary vertices; indeed if $y' \in f^{-1}(y)$ was interior, then because $\mathcal{E}(y') \cap f^{-1}(E(H_1)) \to \mathcal{E}(y)$ is bijective, $y'$ must have been the endpoint of some $e'$ which mapped to $e$.  This $e'$ was contracted as a boundary spike, which means $y'$ has been changed to a boundary vertex.

These are the only possibilities; thus, by a sequence of layer-stripping operations, we can obtain $f^{-1}(H_2)$ from $f^{-1}(H_1)$.
\end{proof}

\begin{figure}
	\begin{center}
	\begin{tikzpicture}
		\begin{scope}
			\node[bd] (1A) at (0,1) [label=left:$1A$] {};
			\node[int] (2A) at (0,0) [label=left:$2A$] {};
			\node[bd] (3A) at (0,-1) [label=left:$3A$] {};
			\node[bd] (1B) at (1,1) [label=right:$1B$] {};
			\node[int] (2B) at (1,0) [label=right:$2B$] {};
			\node[bd] (3B) at (1,-1) [label=right:$3B$] {};
			
			\draw (1A) to (2A) to (3A);
			\draw (1B) to (2B) to (3B);
			\draw (1A) to (1B); \draw (2A) to (2B); \draw (3A) to (3B);
			
			\node at (-3,0) {$G = f^{-1}(H_1)$};
		\end{scope}
		
		\begin{scope}[shift={(5,0)}]
			\node[bd] (1) at (0,1) [label=left:$1$] {};
			\node[int] (2) at (0,0) [label=left:$2$] {};
			\node[bd] (3) at (0,-1) [label=left:$3$] {};
			
			\draw (1) to (2) to (3);
			
			\node at (2,0) {$H = H_1$};
		\end{scope}
		
		\begin{scope}[shift={(0,-5)}]
			\node[bd] (1A) at (0,1) [label=left:$1A$] {};
			\node[bd] (2A) at (0,0) [label=left:$2A$] {};
			\node[bd] (1B) at (1,1) [label=right:$1B$] {};
			\node[bd] (2B) at (1,0) [label=right:$2B$] {};
			
			\draw (1A) to (2A);
			\draw (1B) to (2B);
			\draw (1A) to (1B); \draw (2A) to (2B);
			
			\node at (-3,0.5) {$f^{-1}(H_2)$};
		\end{scope}
			
		\begin{scope}[shift={(5,-5)}]
			\node[bd] (1) at (0,1) [label=left:$1$] {};
			\node[bd] (2) at (0,0) [label=left:$2$] {};
			
			\draw (1) to (2);
			
			\node at (2,0.5) {$H_2$};
		\end{scope}
		
		\draw[->] (2,0) to node[auto] {$f$} (4,0);
		\draw[->] (2,-4.5) to node[auto] {$f|_{f^{-1}(H_2)}$} (4,-4.5);
		
		\draw[->] (0.5,-2) -- (0.5,-3);
		\draw[->] (5,-2) -- (5,-3);
		
		\node[align=center, text width=3cm] at (-3,-2.5) {Delete boundary edge $(3A,3B)$, then contract boundary spikes $(2A,3A)$ and $(2B,3B)$.};
		\node[align=center, text width=2cm] at (7,-2.5) {Contract boundary spike $(2,3)$.};
		
	\end{tikzpicture}
	
	\caption{One case in the proof of Lemma \ref{lem:layerstrippingfunctoriality}.  Here $H_1 = H$, and $H_2$ is obtained from $H_1$ by contracting a boundary spike.} \label{fig:functorialityexample}
	
	\end{center}
\end{figure}
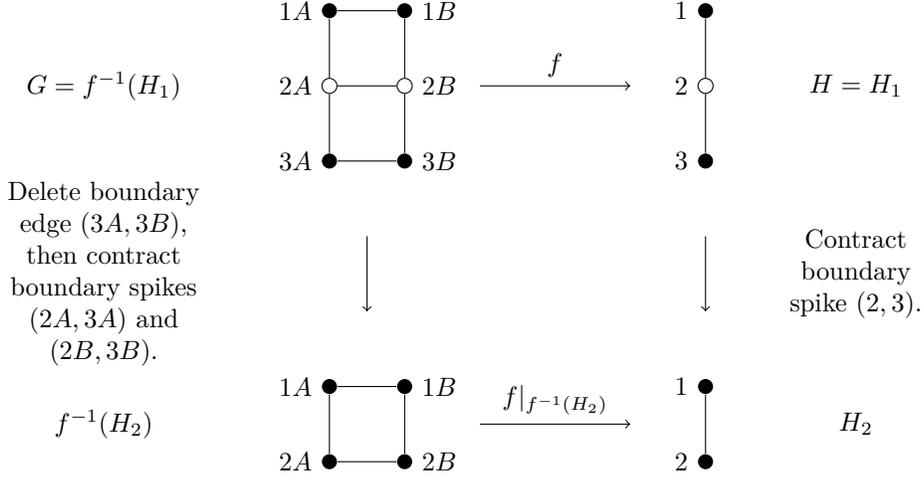

The previous lemma immediately implies
\begin{lemma} \label{lem:layerabilityPullback}
Let $G$ and $H$ be finite $\partial$-graphs.  If $H$ is layerable and there is an unramified morphism $f \colon G \to H$, then $G$ is also layerable.
\end{lemma}

Before developing our basic ideas further, we must clear up some technical issues pertaining to pulling back layer-stripping operations.  Lemma \ref{lem:layerstrippingfunctoriality} shows that each layer-stripping operation pulls back to some sequence of layer-stripping operations, but this sequence is not unique since the transformation $f^{-1}(H_1) \mapsto f^{-1}(H_2)$ could be broken up into layer-stripping operations in multiple ways.  However, with a little cleverness, we can find a type of elementary operation that pulls back uniquely to another operation of the same type.

\begin{definition}
We define a {\bf three-step layer-stripping operation} to be a sequence of boundary edge deletion, isolated boundary vertex deletion, and boundary spike contraction (in \emph{that} order).  Multiple edges or vertices are allowed to be removed at each step; however, if multiple boundary spikes are contracted, then they are not allowed to share any endpoints with each other.  We allow any one of the three steps to be trivial, and thus a simple layer-stripping operation can be considered a three-step layer-stripping operation.  We allow infinitely many vertices or edges to be removed if the $\partial$-graph is infinite.
\end{definition}

The following lemma applies even to infinite $\partial$-graphs.

\begin{lemma} \label{lem:layerstrippingfunctoriality2}
Suppose $f\colon G \to H$ is an unramified $\partial$-graph morphism, and suppose $H_2 \subseteq H_1 \subseteq H$ are sub-$\partial$-graphs.  If $H_2$ is obtained from $H_1$ by a three-step layer-stripping operation, then $f^{-1}(H_2)$ is obtained from $f^{-1}(H_1)$ by a three-step layer-stripping operation.
\end{lemma}

\begin{proof}
The steps are the same as in Lemma \ref{lem:layerstrippingfunctoriality} but we must order the operations in a non-obvious way:
\begin{itemize}
	\item Delete all boundary edges in $f^{-1}(H_1)$ that map to the deleted boundary edges in $H_1$.
	\item Delete all boundary edges in $f^{-1}(H_1)$ that map to the deleted isolated boundary vertices in $H_1$.
	\item Delete all boundary edges in $f^{-1}(H_1)$ that map to the boundary endpoints of spikes in $H_1$.
	\item Delete all boundary edges in $f^{-1}(H_1)$ that map to the contracted boundary spikes in $H_1$.
	\item Delete all isolated boundary vertices in $f^{-1}(H_1)$ that map to the deleted isolated boundary vertices in $H_1$.
	\item Delete all isolated boundary vertices in $f^{-1}(H_1)$ that map to the boundary endpoints of contracted spikes in $H_1$.
	\item Contract all boundary spikes in $f^{-1}(H_1)$ that map to the contracted boundary spikes in $H_1$.
\end{itemize}
We leave the reader to verify that this works by adapting the casework in Lemma \ref{lem:layerstrippingfunctoriality}.
\end{proof}

\begin{figure}
	\begin{center}

	\begin{tikzpicture}[scale=0.6]
		\begin{scope}
			\node[int] (1A) at (0,0) [label = below left:$1A$] {};
			\node[bd] (2A) at (0,1) [label = left:$2A$] {};
			\node[bd] (3A) at (0,2) [label = above left:$3A$] {};
			\node[bd] (1B) at (1,0) [label = below:$1B$] {};
			\node[bd] (2B) at (1,1) [label = above right:$2B$] {};
			\node[bd] (3B) at (1,2) [label = above:$3B$] {};
			\node[bd] (1C) at (2,0) [label = below right:$1C$] {};
			\node[bd] (2C) at (2,1) [label = right:$2C$] {};
			
			\draw (1A) to (2A) to (3A);
			\draw (1B) to (2B) to (3B);
			\draw (1A) to (1B) to (1C);
			\draw (2A) to (2B) to (2C);
			\draw (3A) to (3B);
		\end{scope}

		\node at (1,3.5) {$G = f^{-1}(H_1)$};

		\draw[->,dashed] (-1.4,0.8) -- (-2.4,0.2);

		\begin{scope}[shift={(-4.7,-1)},scale=0.8]
			\node[int] (1A) at (0,0) {};
			\node[bd] (2A) at (0,1) {};
			\node[bd] (3A) at (0,2) {};
			\node[bd] (1B) at (1,0) {};
			\node[bd] (2B) at (1,1) {};
			\node[bd] (3B) at (1,2) {};
			\node[bd] (1C) at (2,0) {};
			\node[bd] (2C) at (2,1) {};
			
			\draw (1A) to (2A);
			\draw (1A) to (1B) to (1C);
		\end{scope}

		\draw[->,dashed] (-3.9,-1.5) -- (-3.9,-3);

		\begin{scope}[shift={(-4.7,-4)},scale=0.8]
			\node[int] (1A) at (0,0) {};
			\node[bd] (2A) at (0,1) {};
			\node[bd] (1B) at (1,0) {};
			\node[bd] (1C) at (2,0) {};
			
			\draw (1A) to (2A);
			\draw (1A) to (1B) to (1C);
		\end{scope}

		\draw[->,dashed] (-2.6,-3.8) -- (-0.8,-4.6);

		\begin{scope}[shift={(0,-5)}]
			\node[bd] (1A) at (0,0) [label = below left:$1A$] {};
			\node[bd] (1B) at (1,0) [label = below:$1B$] {};
			\node[bd] (1C) at (2,0) [label = below right:$1C$] {};
			
			\draw (1A) to (1B) to (1C);
		\end{scope}
		
		\node at (1,-6.5) {$f^{-1}(H_2)$};
		
		\draw[->] (1,-1.5) -- (1,-3.5);
		\draw[->] (7,-1.5) -- (7,-3.5);
	
		\draw[->] (4,1) to node[auto] {$f$} (6,1);
		\draw[->] (4,-5) to node[auto] {$f|_{f^{-1}(H_2)}$} (6,-5);
	
		\begin{scope}[shift={(7,0)}]
			\node[int] (1) at (0,0) [label = below right:$1$] {};
			\node[bd] (2) at (0,1) [label = right:$2$] {};
			\node[bd] (3) at (0,2) [label = above right:$3$] {};
			
			\draw (1) to (2) to (3);
		\end{scope}

		\node at (7,3.5) {$H = H_1$};

		\draw[->,dashed] (8.1,0.7) -- (9.2,0);

		\begin{scope}[shift={(10,-1)},scale=0.8]
			\node[int] (1) at (0,0) {};
			\node[bd] (2) at (0,1) {};
			\node[bd] (3) at (0,2) {};
			
			\draw (1) to (2);
		\end{scope}

		\draw[->,dashed] (10,-1.5) -- (10,-2.5);

		\begin{scope}[shift={(10,-4)},scale=0.8]
			\node[int] (1) at (0,0) {};
			\node[bd] (2) at (0,1) {};
			
			\draw (1) to (2);
		\end{scope}

		\draw[->,dashed] (9.2,-4) -- (7.8,-4.7);

		\begin{scope}[shift={(7,-5)}]
			\node[bd] (1) at (0,0) [label = below right:$1$] {};
		\end{scope}
	
		\node at (7,-6.5) {$H_2$};
	
	\end{tikzpicture}
	
	\caption{An example of Lemma \ref{lem:layerstrippingfunctoriality2}.} \label{fig:functorialityexample2}
	
	\end{center}
\end{figure}
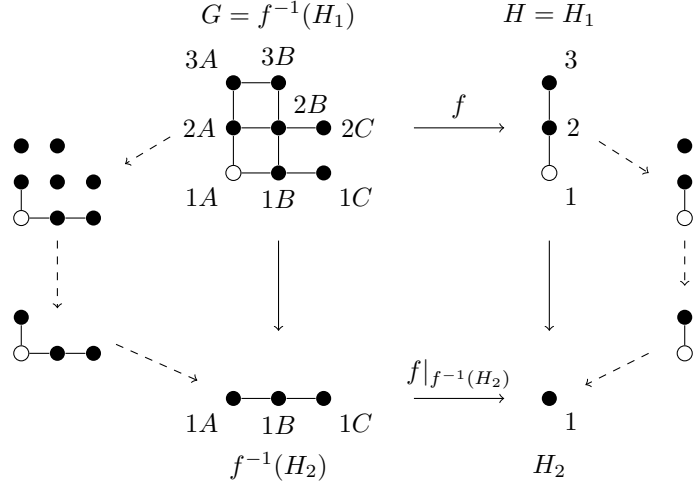

We define as {\bf three-step layer-stripping filtration} of a $\partial$-graph $G$ as a sequence of subgraphs $G = G_0 \supseteq G_1 \supseteq \dots$ such that $\bigcap G_j = \varnothing$ and $G_{j+1}$ is obtained from $G_j$ by a three-step layer-stripping operation.  The last lemma is convenient because it implies
\begin{lemma} \label{lem:functoriality}
Suppose that $f\colon G \to H$ is an unramified $\partial$-graph morphism.  If $H_0 \supseteq H_1 \supseteq \dots$ is a three-step layer-stripping filtration of $H$, then $f^{-1}(H_0) \supseteq f^{-1}(H_1) \supseteq \dots$ is a three-step layer-stripping filtration of $G$.  Therefore, one can define a contravariant functor $\partial\text{\cat{-graph}}_{\text{unrm}} \to \text{\cat{Set}}$ by mapping $G$ to the set of three-step layer-stripping filtrations of $G$.
\end{lemma}

\begin{remark}
As in Observation \ref{obs:layeringsubdgraph}, if $G'$ is obtained from $G$ by a sequence of layer-stripping operations, then $G'$ is a sub-$\partial$-graph of $G$.  Moreover, Lemmas \ref{lem:layeringUpsilon1} and \ref{lem:layeringUpsilon2} generalize to layer-stripping operations, even with infinitely many edges.
\end{remark}

\begin{remark}
Lemma \ref{lem:layerstrippingfunctoriality} fails for harmonic morphisms in general: If $x$ is the boundary endpoint of a boundary spike $e$, then a vertex in $f^{-1}(x)$ might not be a boundary spike.  It could have degree $> 1$ since there can be multiple preimages of $e$ attached to it.  This problem is illustrated by the $\partial$-graph morphism in Figure \ref{fig:modifiedprojectionGHMexample}.
\end{remark}

\subsection{The Flower Functor}

Layer-stripping can be viewed as a loose discrete analogue of a deformation retraction; it is a reduction to a smaller space that leaves our algebraic invariant $\Upsilon$ unchanged.  This inspires the following analogue of homotopy equivalence:

\begin{definition}
We say two finite $\partial$-graphs $G$ and $G'$ are {\bf layerably equivalent} if there is a finite sequence of $\partial$-graphs $G = G_0, G_1, \dots, G_n = G'$ such that for each $j$, either 1) $G_j$ is obtained from $G_{j+1}$ by a layer-stripping operation or 2) $G_{j+1}$ is obtained from $G_j$ by a layer-stripping operation.  As usual, we apply the same terminology to $R$-networks.
\end{definition}

If two finite $R^\times$-networks $(G,L)$ and $(G',L')$ are layerably equivalent, then by Lemma \ref{lem:layeringUpsilon1} we have
\[
\Upsilon(G,L) \oplus R^n \cong \Upsilon(G',L') \oplus R^{n'}
\]
for some $n$ and $n' \in \N_0$.  Thus, the torsion submodules of $\Upsilon(G,L)$ and $\Upsilon(G',L')$ are isomorphic.  Similar reasoning applies the modules $\c U$ and $\c U_0$ for the two networks.

Our next goal is to find a canonical representative for each equivalence class.  A natural candidate is a $\partial$-graph with no boundary spikes, boundary edges, or disconnected boundary vertices, which we will call a {\bf flower}.  The name ``flower'' was coined in 1992 by David Ingerman and James Morrow who were studying some planar examples which looked like flowers,\footnote{Personal communication with David Ingerman and James Morrow.} and it was first written down in \cite{Reichert}.  Examples of flowers include the $\partial$-graph in Figure \ref{fig:flowerexample}, the boundary-interior bipartite $\partial$-graphs in Example \ref{ex:completebipartite} and Figure \ref{fig:K32}, the $\CLF$ $\partial$-graphs discussed in \S \ref{sec:CLF} (see FIgure \ref{fig:CLFflower}), and the $\partial$-graph on the left hand side of Figure \ref{fig:modifiedprojectionGHMexample}.

\begin{figure}
\begin{center}
	\begin{tikzpicture}[scale=0.7]
		\node[bd] (1) at (0,1.5) {};
		\node[bd] (2) at (-1.5,0) {};
		\node[bd] (3) at (0,-1.5) {};
		\node[bd] (4) at (1.5,0) {};
		
		\node[int] (A) at (0.5,0.5) {};
		\node[int] (B) at (-0.5,0.5) {};
		\node[int] (C) at (-0.5,-0.5) {};
		\node[int] (D) at (0.5,-0.5) {};
		
		\draw (A) to (1) to (B) to (2) to (C) to (3) to (D) to (4) to (A);
		\draw (A) to (B) to (C) to (D) to (A);
	\end{tikzpicture}
	
	\caption{Example of a flower.} \label{fig:flowerexample}

\end{center}
\end{figure}
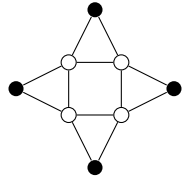

Every finite $\partial$-graph is layerably equivalent to a flower.  Indeed, we can keep removing boundary spikes, boundary edges, and isolated boundary vertices until there are no more left.  The result is a sequence of layer-stripping operations that transforms $G$ into a flower $\flower(G)$.  We claim that in fact this flower is unique and the map $G \mapsto \flower(G)$ is a functor on $\partial \text{\cat{-graph}}_{\text{unrm}}^0$.

~
\begin{theorem}~ \label{thm:flowerfunctor}
\begin{enumerate}
	\item Every finite $\partial$-graph $G$ can be layer-stripped to a unique flower $\flower(G)$.
	\item There is exactly one flower in each layerable equivalence class.
	\item If $f\colon G \to H$ is a UHM, then $\flower(G) \subseteq f^{-1}(\flower(H))$.
	\item $\flower$ is a functor $\partial \text{\cat{-graph}}_{\text{unrm}}^0 \to \partial \text{\cat{-graph}}_{\text{unrm}}^0$, and the inclusion $\flower(G) \to G$ is a natural transformation $\flower \to \id$.
\end{enumerate}
\end{theorem}

\begin{proof}
To prove the uniqueness claim of (1), suppose that we have two flowers $H$ and $H' \subseteq G$, and let
\[
G = H_0 \supseteq H_1 \supseteq \dots \supseteq H_n = H, \qquad G = H_0' \supseteq \dots \supseteq H_m' = H'
\]
be the corresponding three-step layer-stripping filtrations.  Applying Lemma \ref{lem:functoriality} to the inclusion map $H \to G$, we see that
\[
H = H \cap H_0' \supseteq \dots \supseteq H \cap H_m' = H \cap H'
\]
is another three-step layer-stripping filtration.  Since $H$ is a flower, one cannot perform any nontrivial layer-stripping operations on it, so the filtration must be trivial, so that $H = H \cap H'$.  By symmetry, $H' = H \cap H' = H$.

To prove (2), suppose $G'$ is obtained from $G$ by a layer-stripping operation.  We just showed $\flower(G)$ is independent of the sequence of layer-stripping operations, so that $\flower(G)$ is obtained by performing the layer-stripping operation $G \mapsto G'$, then reducing $G'$ to a flower.  Hence, $\flower(G) = \flower(G')$.  In general, if we have a layerable equivalence sequence $G = G_0, G_1, \dots, G_n = G'$, then $\flower(G_j)$ and $\flower(G_{j+1})$ are equal (based on the identification of $G_j$ as a subgraph of $G_{j+1}$ or vice versa) and hence $\flower(G) = \flower(G')$.

(3) follows from Lemma \ref{lem:functoriality}, and (4) follows from (3).
\end{proof}

\begin{remark}
$G$ is layerable if and only if $\flower(G) = \varnothing$.
\end{remark}

As a consequence, the study of torsion for finite $R^\times$-networks can be reduced in a functorial manner to the study of torsion for flowers: 

~

\begin{corollary}~
\begin{enumerate}
	\item There is a flower functor $\flower: R^\times\text{\cat{-net}}_{\text{unrm}}^0 \to R^\times\text{\cat{-net}}_{\text{unrm}}^0$.
	\item $\Upsilon(G,L) = \Upsilon(\flower (G,L)) \oplus R^n$ for some $n$ (depending on $(G,L)$).
	\item If $F$ is any functor on $R$-modules that commutes with direct sums and vanishes on free modules, then the inclusion natural transformation $\flower \to \id$ on $R^\times\text{\cat{-net}}_{\text{unrm}}^0$ induces a natural isomorphism $F \circ \Upsilon \circ \flower \to F \circ \Upsilon$.
\end{enumerate}
\end{corollary}

\begin{proof}
(1) follows directly from Theorem \ref{thm:flowerfunctor}.  (2) follows from Lemma \ref{lem:layeringUpsilon1} because $\flower(G,L)$ is obtained from $(G,L)$ by a sequence of layer-stripping operations.  (3) follows from (2).
\end{proof}

\section{Complete Reducibility} \label{sec:completereducibility}

Complete reducibility is a generalization of layerability which allows us to applying layer-stripping operations as well as split apart $\partial$-graphs that are glued together at one boundary vertex (Definition \ref{def:completelyreducible}).  We shall give an algebraic characterization of complete reducibility analogous to Theorem \ref{thm:layerabilitycharacterization} except that it uses Laplacians with $d = 0$ rather than generalized Laplacians.  To simplify the proof, we first define a reduced version of $\Upsilon$ suited to Laplacians with $d = 0$.  We use our algebraic characterization to prove that boundary-interior bipartite $\partial$-graphs of a certain type are not completely reducible.

\subsection{The Reduced Module $\tilde{\Upsilon}$}

Recall that a generalized Laplacian $L$ is given by $w: E \to R$ and $d: V \to R$.  If $d = 0$, we will call $L$ a {\bf Laplacian} or {\bf weighted Laplacian}.    A network given by a weighted Laplacian will be called a {\bf normalized} $R$-network.  For weighted Laplacians, constant functions are always harmonic, and dually every element of $L(RV)$ has coordinates which sum to zero.  Therefore, it will be convenient to define a reduced version of $\Upsilon$.

Let $L$ be a weighted Laplacian.  Let $\epsilon: RV \to R$ be the map which sums the coordinates, given by $x \mapsto 1$ for every $x \in V$.  Then $\im L \subseteq \ker \epsilon$ because
\[
\epsilon(Lx) = \sum_{e \in \mathcal{E}(x)} w(e)(\epsilon x -  \epsilon e_-) = 0.
\]
Therefore, we can define
\[
\tilde{\Upsilon}(G,L) = \ker \epsilon / L(RV^\circ).
\]
We remark that since $\ker \epsilon \subseteq RV$, we can regard $\tilde{\Upsilon}(G,L)$ as a submodule of $\Upsilon(G,L)$.  Moreover, for each vertex $x \in V(G)$, there is an internal direct sum
\[
\Upsilon(G,L) = \tilde{\Upsilon}(G,L) \oplus Rx.
\]

We also define
\[
\tilde{\c U}(G,L,M) = \c U(G,L,M) / (\text{constant functions})
\]

Most of the results for $\Upsilon$ and their proofs adapt in a straightforward way to $\tilde{\Upsilon}$.  We list the ones we will need for the algebraic characterization.  Using similar reasoning as in Lemma \ref{lem:hom}, one can show
\begin{lemma}
If $L$ is a weighted Laplacian on $G$, then there is a natural $R$-module isomorphism
\[
\tilde{\c U}(G,L,M) \cong \Hom_R(\tilde{\Upsilon}(G,L),M).
\]
\end{lemma}

Moreover, by similar reasoning as in \ref{subsec:degeneracy}, we have

\begin{lemma}
If $(G,L)$ is a non-degenerate normalized network, then a free resolution of $\tilde{\Upsilon}(G,L)$ is given by
\[
\dots \to 0 \to RV^\circ \xrightarrow{L} \ker \epsilon \to \tilde{\Upsilon}(G,L) \to 0.
\]
Moreover,
\[
\Tor_1^R(\tilde{\Upsilon}(G,L),M) \cong \c U_0(G,L,M)
\]
and $\Tor_j^R(\tilde{\Upsilon}(G,L),M) = 0$ for $j > 1$.
\end{lemma}

Lemma \ref{lem:layeringUpsilon1} carries over almost word for word.

\begin{lemma} \label{lem:layeringUpsilon3}
Let $(G',L')$ be a normalized $R^\times$-network.  If $(G',L')$ is obtained from $(G,L)$ by adjoining a boundary spike or boundary edge, then the induced maps $\tilde{\Upsilon}(G,L) \to \tilde{\Upsilon}(G',L')$ and $\c U(G,L,M) \to \c U(G',L',M)$ are isomorphisms.

If $(G',L')$ is obtained from $(G,L)$ by adjoining an isolated boundary vertex $x$ and $G$ is nonempty, then we have
\[
\tilde{\Upsilon}(G',L') \cong \tilde{\Upsilon}(G,L) \oplus R
\]
and
\[
\tilde{\c U}(G',L',M) \cong \tilde{\c U}(G,L,M) \times M.
\]
\end{lemma}

\subsection{Completely Reducible $\partial$-Graphs}

\begin{definition}
Given $\partial$-graphs $G_1$ and $G_2$ and specified vertices $x_i \in \partial V_i$ for $i = 1,2$, the {\bf boundary wedge sum}
\[
G_1 \vee G_2 = G_1 \vee_{x_1,x_2} G_2
\]
is obtained by identifying $x_1$ with $x_2$ in the disjoint union $G_1 \sqcup G_2$.  Note that $G_1$ and $G_2$ are sub-$\partial$-graphs of $G_1\vee G_2$.  We apply the same terminology to $R$-networks as to $\partial$-graphs.
\end{definition}

Just as with layerable extensions, the behavior of $\tilde{\Upsilon}$ under boundary wedge-sums and disjoint unions is easy to characterize.  A similar result for the critical group of graphs without boundary appears in \cite[Remark, p.\ 280]{Lor1}.

\begin{lemma} \label{lem:wedgesumUpsilon}
Let $(G,L)$ be a normalized $R$-network.   If $(G,L) = (G_1,L_1) \vee (G_2,L_2)$, then
\begin{align*}
\tilde{\Upsilon}(G,L) &\cong \tilde{\Upsilon}(G_1,L_1) \oplus \tilde{\Upsilon}(G_2,L_2) \\
\tilde{\c U}(G,L,M) &\cong \tilde{\c U}(G_1,L_1,M) \times \tilde{\c U}(G_2,L_2,M) \\
\c U_0(G,L,M) &\cong \c U_0(G_1,L_1,M) \oplus \c U_0(G_2,L_2,M).
\end{align*}
If $(G,L) = (G_1,L_1) \sqcup (G_2,L_2)$, then
\begin{align*}
\tilde{\Upsilon}(G,L) &\cong \tilde{\Upsilon}(G_1,L_1) \oplus \tilde{\Upsilon}(G_2,L_2) \oplus R \\
\tilde{\c U}(G,L,M) &\cong \tilde{\c U}(G_1,L_1,M) \times \tilde{\c U}(G_2,L_2,M) \times M \\
\c U_0(G,L,M) &\cong \c U_0(G_1,L_1,M) \oplus \c U_0(G_2,L_2,M).
\end{align*}
\end{lemma}

\begin{proof}
In the case of a boundary wedge-sum, we have an internal direct sum
\[
\ker \epsilon = \ker \epsilon_1 \oplus \ker \epsilon_2.
\]
Moreover,
\[
L(RV^\circ) = L_1(RV_1^\circ) + L_2(RV_2^\circ), \quad L_1(RV_1^\circ) \subseteq \ker \epsilon_1, \quad L_2(RV_2^\circ) \subseteq \ker \epsilon_2.
\]
Therefore, taking quotients yields
\[
\tilde{\Upsilon}(G,L) \cong \tilde{\Upsilon}(G_1,L_1) \oplus \tilde{\Upsilon}(G_2,L_2).
\]
The claim for $\tilde{\c U}$ follows by applying $\Hom(-,M)$.

To prove the claim for $\c U_0$, consider the map
\[
\Phi: \c U_0(G_1,L_1,M) \oplus \c U_0(G_2,L_2,M) \to \c U_0(G,L,M)
\]
defined by the inclusion maps $(G_j,L_j) \to (G,L)$.  Let $x$ be the common boundary vertex of $G_1$ and $G_2$.  To see that $\Phi$ is injective, suppose $\Phi(u_1 \oplus u_2) = 0$.  Then $u_1$ and $u_2$ vanish at $x$ by definition of $\c U_0(G_j,L_j,M)$, and they vanish on the rest of $G_1$ and $G_2$ respectively because $G_1$ and $G_2$ only intersect at $x$.  To show surjectivity of $\Phi$, let $u \in \c U_0(G,L,M)$, and let $u_j = u|_{V_j}$.  Clearly, $u_j = 0$ on $\partial V_j$ since $\partial V_j \subseteq \partial V$.  Moreover, $L_j u_j(y) = Lu(y) = 0$ for every $y \in V_j \setminus \{x\}$.  Because $\im (L_j \otimes \id_M) \subseteq \ker (\epsilon_j \otimes \id_M)$, this implies that $L_j u(x) = 0$ also.  Thus, $u_j \in \c U_0(G_j,L_j,M)$ and hence $u = \Phi(u_1 \oplus u_2) \in \im \Phi$.  So $\Phi$ is an isomorphism as desired.

In the case of a disjoint union, the claim for $\tilde{\Upsilon}$ is proved in a similar way after noting that
\[
\ker \epsilon \cong \ker \epsilon_1 \oplus \ker \epsilon_2 \oplus R.
\]
The claim for $\c U$ follows by applying $\Hom(-,M)$.  The argument for $\c U_0$ is similar to the boundary wedge-sum case but easier.
\end{proof}

\begin{definition}
{\bf Completely reducible finite $\partial$-graphs} are defined to be the smallest class $\c C$ of finite $\partial$-graphs that contains the empty graph and is closed under layerable extensions, disjoint unions, and boundary wedge-sums.  More informally, a graph $G$ is completely reducible if it can be reduced to nothing by layer-stripping and splitting apart boundary wedge-sums and disjoint unions.
\end{definition}

\begin{definition} \label{def:completelyreducible}
A finite $\partial$-graph is {\bf irreducible} if it has no boundary spikes, boundary edges, or isolated boundary vertices, and it is not a boundary wedge-sum or disjoint union.  Note that every irreducible $\partial$-graph is a flower.
\end{definition}

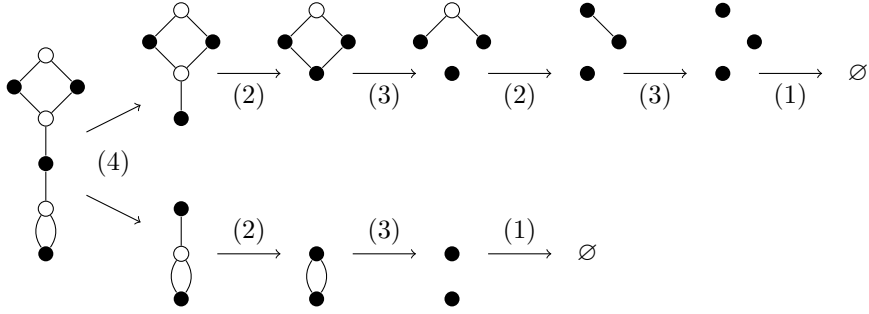
\begin{figure}

\begin{center}
\begin{tikzpicture}[scale=0.6]
	\begin{scope}
		\node[bd] (A) at (0,0) {};
		\node[int] (C) at (0,1) {};
		\node[bd] (D) at (0.7,1.7) {};
		\node[bd] (E) at (-0.7,1.7) {};
		\node[int] (F) at (0,2.4) {};
		\node[int] (G) at (0,-1) {};
		\node[bd] (H) at (0,-2) {};
		
		\draw (G) to[bend left = 35] (H);
		\draw (H) to [bend left = 35] (G);
		\draw (G) to (A);
		\draw (A) to (C); \draw (C) to (D);
		\draw (C) to (E);  \draw (D) to (F);  \draw (E) to (F);
	\end{scope}

	\draw[->] (0.9,0.7) to (2.1,1.3);
	\draw[->] (0.9,-0.7) to (2.1,-1.3);
	\node at (1.5,0) {(4)};

	\begin{scope}[shift = {(3,-1)}]
		\node[bd] (A) at (0,0) {};
		\node[int] (G) at (0,-1) {};
		\node[bd] (H) at (0,-2) {};
		
		\draw (G) to[bend left = 35] (H);
		\draw (H) to [bend left = 35] (G);
		\draw (G) to (A);
	\end{scope}

	\begin{scope}[shift = {(3,1)}]
		\node[bd] (A) at (0,0) {};
		\node[int] (C) at (0,1) {};
		\node[bd] (D) at (0.7,1.7) {};
		\node[bd] (E) at (-0.7,1.7) {};
		\node[int] (F) at (0,2.4) {};
		
		\draw (A) to (C); \draw (C) to (D);
		\draw (C) to (E);  \draw (D) to (F);  \draw (E) to (F);
	\end{scope}

	\draw[->] (3.8,-2) to (5.2,-2);
	\draw[->] (3.8,2) to (5.2,2);
	\node at (4.5,1.5) {(2)};
	\node at (4.5,-1.5) {(2)};

	\begin{scope}[shift = {(6,-1)}]
		\node[bd] (G) at (0,-1) {};
		\node[bd] (H) at (0,-2) {};
		
		\draw (G) to[bend left = 35] (H);
		\draw (H) to [bend left = 35] (G);
	\end{scope}

	\begin{scope}[shift = {(6,1)}]
		\node[bd] (C) at (0,1) {};
		\node[bd] (D) at (0.7,1.7) {};
		\node[bd] (E) at (-0.7,1.7) {};
		\node[int] (F) at (0,2.4) {};
		
		\draw (C) to (D);
		\draw (C) to (E);  \draw (D) to (F);  \draw (E) to (F);
	\end{scope}
	
	\draw[->] (6.8,-2) to (8.2,-2);
	\draw[->] (6.8,2) to (8.2,2);
	\node at (7.5,1.5) {(3)};
	\node at (7.5,-1.5) {(3)};
	
	\begin{scope}[shift = {(9,-1)}]
		\node[bd] (G) at (0,-1) {};
		\node[bd] (H) at (0,-2) {};
	\end{scope}

	\begin{scope}[shift = {(9,1)}]
		\node[bd] (C) at (0,1) {};
		\node[bd] (D) at (0.7,1.7) {};
		\node[bd] (E) at (-0.7,1.7) {};
		\node[int] (F) at (0,2.4) {};
		
		\draw (D) to (F);  \draw (E) to (F);
	\end{scope}

	\draw[->] (9.8,-2) to (11.2,-2);
	\draw[->] (9.8,2) to (11.2,2);
	\node at (10.5,1.5) {(2)};
	\node at (10.5,-1.5) {(1)};

	\begin{scope}[shift = {(12,1)}]
		\node[bd] (C) at (0,1) {};
		\node[bd] (D) at (0.7,1.7) {};
		\node[bd] (F) at (0,2.4) {};
		
		\draw (D) to (F);
	\end{scope}

	\node at (12,-2) {$\varnothing$};

	\draw[->] (12.8,2) to (14.2,2);
	\node at (13.5,1.5) {(3)};

	\begin{scope}[shift = {(15,1)}]
		\node[bd] (C) at (0,1) {};
		\node[bd] (D) at (0.7,1.7) {};
		\node[bd] (F) at (0,2.4) {};
	\end{scope}

	\draw[->] (15.8,2) to (17.2,2);
	\node at (16.5,1.5) {(1)};

	\node at (18,2) {$\varnothing$};

\end{tikzpicture}
\end{center}

\caption{A completely reducible $\partial$-graph.  The boundary vertices are black and interior vertices are white.  The operations are (1) isolated boundary vertex deletion, (2) boundary spike contraction, (3) boundary edge deletion, (4) splitting a boundary wedge-sum.}

\label{fig:CRexample}

\end{figure}

The following is an analogue of Proposition \ref{prop:layerablebehavior}:

\begin{proposition} \label{prop:reduciblebehavior}
Let $G$ be a finite nonempty completely reducible $\partial$-graph.  If $(G,L)$ is a normalized $R^\times$-network, then $(G,L)$ is non-degenerate and $\tilde{\Upsilon}(G,L)$ is a free $R$-module of rank $|\partial V(G)| - 1$.
\end{proposition}

\begin{proof}
Let $\mathcal{C}$ be the class of $\partial$-graphs for which the claims hold, together with the empty $\partial$-graph.  Lemmas \ref{lem:layeringUpsilon3} and \ref{lem:wedgesumUpsilon} imply that $\mathcal{C}$ is closed under layerable extensions, disjoint unions, and boundary wedge-sums.
\end{proof}

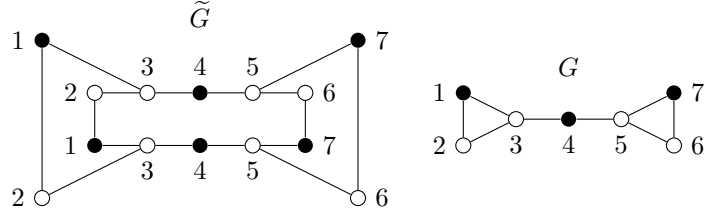
\begin{figure}

\begin{center}

\begin{tikzpicture}[scale=0.7]

\begin{scope}
	\node at (0,2) {$\tilde{G}$};

	\node[int] (2A) at (-3,-1.5) [label = left:$2$] {};
	\node[bd] (1B) at (-3,1.5) [label = left:$1$] {};

	\node[bd] (1A) at  (-2,-0.5) [label = left:$1$] {};
	\node[int] (2B) at (-2,0.5) [label = left:$2$] {};
	
	\node[int] (3A) at (-1,-0.5) [label = below:$3$] {};
	\node[int] (3B) at (-1,0.5) [label = above:$3$] {};
	\node[bd] (4A) at (0,-0.5) [label = below:$4$] {};
	\node[bd] (4B) at (0,0.5) [label = above:$4$] {};
	\node[int] (5A) at (1,-0.5) [label = below:$5$] {};
	\node[int] (5B) at (1,0.5) [label = above:$5$] {};
	
	\node[bd] (7A) at (2,-0.5) [label = right:$7$] {};
	\node[int] (6B) at (2,0.5) [label = right:$6$] {};
	
	\node[int] (6A) at (3,-1.5) [label = right:$6$] {};
	\node[bd] (7B) at (3,1.5) [label = right:$7$] {};
	
	\draw (3A) to (1A) to (2B) to (3B);
	\draw (3B) to (1B) to (2A) to (3A);
	\draw (3A) to (4A) to (5A);
	\draw (3B) to (4B) to (5B);
	\draw (5A) to (7A) to (6B) to (5B);
	\draw (5B) to (7B) to (6A) to (5A);
\end{scope}

\begin{scope}[shift={(7,0)}]
	\node at (0,1) {$G$};

	\node[bd] (1) at (-2,0.5) [label = left:$1$] {};
	\node[int] (2) at (-2,-0.5) [label = left:$2$] {};
	\node[int] (3) at (-1,0) [label = below:$3$] {};
	\node[bd] (4) at (0,0) [label = below:$4$] {};
	\node[int] (5) at (1,0) [label = below:$5$] {};
	\node[int] (6) at (2,-0.5) [label = right:$6$] {};
	\node[bd] (7) at (2,0.5) [label = right:$7$] {};
	
	\draw (1) to (2) to (3) to (1);
	\draw (3) to (4) to (5);
	\draw (5) to (6) to (7) to (5);
\end{scope}

\end{tikzpicture}

\caption{A covering map $f \colon \tilde{G} \to G$ such that $G$ decomposes as a boundary wedge-sum and $\tilde{G}$ does not.  In fact, $G$ is completely reducible and $\tilde{G}$ is irreducible.} \label{fig:reduciblefunctorialityfail}

\end{center}

\end{figure}

Unlike layer-stripping operations, the operation of spliting apart a boundary wedge-sum does \emph{not} pull back through unramified $\partial$-graph morphisms.  The problem is illustrated in Figure \ref{fig:reduciblefunctorialityfail}.  However, we do have

\begin{observation} \label{obs:wedgesumfunctoriality}
Suppose that $G$ is a sub-$\partial$-graph of $H$ and that $H$ decomposes as a boundary wedge-sum or disjoint union of $H_1$ and $H_2$.  Then $G$ decomposes as a boundary wedge-sum or disjoint union of $G \cap H_1$ and $G \cap H_2$.  Together with Lemma \ref{lem:layerstrippingfunctoriality}, this implies that a sub-$\partial$-graph of a completely reducible $\partial$-graph is also completely reducible.
\end{observation}

\subsection{Algebraic Characterization}

We shall prove an algebraic characterization of complete reducibility in the same way as we did for layerability (Theorem \ref{thm:layerabilitycharacterization}).  As in Lemma \ref{lem:layerabilityfieldcharacterization}, we first construct degenerate networks over fields.

\begin{lemma}\label{lem:reducibilityfieldcharacterization}
Let $G$ be a finite $\partial$-graph and let $F$ be an infinite field.  Then $G$ is completely reducible if and only if every normalized $F^\times$-network on $G$ is non-degenerate.
\end{lemma}

\begin{proof}
The implication $\implies$ follows from Proposition \ref{prop:reduciblebehavior}.

Let $G'$ be a minimal sub-$\partial$-graph of $G$ which is not completely reducible.  Note that $G'$ must be irreducible.  As in the proof of Lemma \ref{lem:layerabilityfieldcharacterization}, it suffices to construct degenerate edge weights on $G'$.

Our strategy is choose a potential function $u$ first with $u|_{\partial V} = 0$, and \emph{then} choose an edge-weight function $w$ that will make $L u \equiv 0$.  Let $S \subseteq E(G)$ be the union of all cycles, i.e., $S$ contains every edge that is part of any cycle.  Note that every edge in $S$ must have endpoints in distinct components of $G \setminus S$.  Define $u$ to be zero on every component of $G \setminus S$ that contains a boundary vertex of $G$, and assign $u$ a different nonzero value on each component of $G \setminus S$ that does not contain any boundary vertices.

We need to guarantee that $u$ is not identically zero.  But in fact, we claim that $u$ is nonzero at every interior vertex.  To prove this, it suffices to show that every edge $e$ with endpoints $x \in \partial V$ and $y \in V^\circ$ must be in $S$, that is, such an edge $e$ must be contained in some cycle.  By hypothesis, our edge $e$ is not a boundary spike. Thus, there is some other edge $e' \neq e$ incident to $x$.  Let $z$ be the other endpoint $e'$.  Since $G$ is not a boundary wedge-sum, deleting $x$ leaves $G$ connected. Thus, there is a path $P = \{e_1,\dots,e_k\}$ from $y$ to $z$ which avoids $x$.  Then $P \cup \{e,e'\}$ is a cycle containing $e$. Consequently, $u$ is nonzero at every interior vertex.

Now we choose the edge weights.  Choose oriented cycles $C_1, \dots, C_k$ such that $S = \bigcup_{j=1}^k (C_j \cup \overline{C}_j)$.  If $e \in C_j$, then $e \in S$ and hence $e_+$ and $e_-$ are in distinct components of $G \setminus S$, so $du(e) = u(e_+) - u(e_-) \neq 0$.  For each $j$, define
  \[
    w_j(e) = w_j(\overline{e}) = \begin{cases}
      1/du(e), \text{ for } e \in C_j \\
      0, \text{ for } e \not \in C_j \cup \overline{C}_j.
    \end{cases}
  \]
Then $w_j(e) du(e)$ is $1$ on $C_j$ and $-1$ on $\overline{C}_j$ and vanishes elsewhere.  Therefore, if we let $L_j$ be the Laplacian associated to the edge-weight function $w_j$, then we have $L_j u = 0$.  For each $e \in S$, there is a weight function $w_j$ with $w_j(e) \neq 0$.  Since $F$ is infinite and the graph is finite, we may choose $\alpha_j \in F$ such that $\sum_{j=1}^k \alpha_j w_j(e) \neq 0$ for all $e \in S$ simultaneously.

Set $w = 1_{E \setminus S} + \sum_{j=1}^k \alpha_j w_j$ and let $L$ be the associated Laplacian.  Then $w(e) \neq 0$ for each $e$.  Because $u$ is constant on each component of $G \setminus S$, we know that $u(e_+) - u(e_-) = 0$ for each $e \in E \setminus S$.  Thus, these edges do not contribute to $L u$, and so
\[
L u = \sum_{j=1}^k \alpha_j L_ju = 0.
\]
Thus, $(G,L)$ is the desired degenerate $F^\times$-network becuase $0 \neq u \in \c U_0(G,L,F)$.
\end{proof}

We proved equivalent algebraic characterizations for layerability by assigning indeterminates to the edges (see Proposition \ref{prop:genericfieldnetwork}).  The analogue for normalized networks is as follows.

\begin{definition}
Let $G$ be a $\partial$-graph and let $F$ be a field.  Then $\tilde{R} = \tilde{R}(G,F) = F[t_e^{\pm 1}: e \in E]$ will denote the Laurent polynomial algebra over $F$ with generators indexed by the edges of $G$.  Let $\tilde{L} = \tilde{L}(G,F)$ denote the weighted Laplacian over $\tilde{R}$ given by $\tilde{w}(e) = t_e$.
\end{definition}

\begin{proposition} \label{prop:genericfieldnetwork2}
Let $G$ be a finite $\partial$-graph such that each component contains at least one boundary vertex, and let $F$ be a field.  Then $(G, \tilde{L})$ is non-degenerate.  Moreover, $\tilde{\Upsilon}(\tilde{G}, \tilde{L})$ is a flat $\tilde{R}$-module if and only if every normalized $F^\times$-network on $G$ is non-degenerate.
\end{proposition}

\begin{proof}
To prove that $(G,\tilde{L})$ is non-degenerate, it suffices to prove that each connected component of $(G,\tilde{L})$ is non-degenerate.  Therefore, we may assume without loss of generality that $G$ is connected.  Since our original graph has at least one boundary vertex in each connected component, we may assume $G$ is connected and has at least one boundary vertex $x$.

Recall that $(G,\tilde{L})$ is non-degenerate if and only if $\tilde{L} \colon \tilde{R}V^\circ \to \tilde{R}V$ is injective (see proof of \ref{prop:tor}).  For our given boundary vertex $x$, let $\tilde{L}_x \colon \tilde{R}(V \setminus \{x\}) \to \tilde{R}(V \setminus \{x\})$ be the Laplacian $\tilde{L}$, with the domain restricted to chains in $\tilde{R}(V \setminus \{x\}) \subseteq \tilde{R}V$, and with the output truncated by applying the canonical projection $\tilde{R}V \to \tilde{R}(V \setminus x)$.  Then injectivity of $\tilde{L}_x$ will imply injectivity of $\tilde{L} \colon \tilde{R}V^\circ \to \tilde{R}V$ since $V \setminus x \supseteq V^\circ$ and $V \setminus x \subseteq \partial V$.  By the weighted matrix-tree theorem (see \cite[Theorem 1]{RF} and \cite[Theorem 4.2]{RK}), we have
\[
\det \tilde{L}_x = \sum_{T\in\text{Span}(G)} \prod_{e \in T} t_e \neq 0,
\]
where Span$(G)$ denotes the set of spanning trees of $G$.  Since we assumed $G$ is connected, $\det \tilde{L}_x$ is a nonzero polynomial in $(t_e)_{e \in E}$ and hence is a nonzero element of the Laurent polynomial algebra $\tilde{R}$.  Since $\tilde{R}$ is an integral domain, it follows that $\tilde{L}_x$ is injective.  This completes the proof that $(G,\tilde{L})$ is non-degenerate.

The rest of the proof is exactly the same as for Proposition \ref{prop:genericfieldnetwork}.
\end{proof}

The following Theorem is proved the same way as Theorem \ref{thm:layerabilitycharacterization}.

\begin{theorem} \label{thm:reducibilitycharacterization}
Let $G$ be a finite $\partial$-graph such that every component has at least one boundary vertex.  The following are equivalent:
\begin{enumerate}
	\item $G$ is completely reducible.
	\item For every ring $R$, every normalized $R^\times$-network on $G$ is non-degenerate.
	\item For every ring $R$, for every non-degenerate normalized $R^\times$-network $(G,L)$ on the $\partial$-graph $G$, $\tilde{\Upsilon}(G,L)$ is a free $R$-module.
	\item There exists an infinite field $F$ such that $\tilde{\Upsilon}(G, \tilde{L}(G,F))$ is a flat $\tilde{L}(G,F)$-module.
	\item There exists an infinite field $F$ such that every normalized $F^\times$-network on $G$ is non-degenerate.
\end{enumerate}
\end{theorem}

\subsection{Boundary-Interior Bipartitle $\partial$-Graphs}

The correspondence between algebraic and $\partial$-graph-theoretic conditions in Theorem \ref{thm:reducibilitycharacterization} is illustrated by the following proposition about bipartite graphs.  We present both an algebraic proof and an inductive $\partial$-graph-theoretic proof for comparison.  We say a $\partial$-graph is {\bf boundary-interior bipartite} if every edge has one interior endpoint and one boundary endpoint (similar to Example \ref{ex:completebipartite}).

\begin{proposition} \label{prop:boundaryinteriorbipartite}
Suppose that $G$ is a nonempty finite boundary-interior bipartite $\partial$-graph, $|V^\circ| \geq |\partial V|$, and every interior vertex has degree $\geq 2$.  Then $G$ is not completely reducible.
\end{proposition}

\begin{proof}[Algebraic proof]
Let $F$ be any field other than the field $F_2$ with two elements.  We will construct a degenerate $F^\times$-network on $G$.  Since each interior vertex has at least two edges incident to it and each edge is only incident to one interior vertex, we can choose $w: E \to F^\times$ such that $\sum_{e \in \mathcal{E}(x)} w(e) = 0$ for each $x \in V^\circ$.  If $u \in 0^{\partial V} \times F^{V^\circ} \subset F^V$, then $L u|_{V^\circ} = 0$ since
\[
L u(x) = \sum_{e: e_+ = x} w(e)(u(x) - u(e_-)) = \sum_{e: e_+ = x} w(e) u(x) = 0 \text{ for all } x \in V^\circ.
\]
Combining this with the fact that $\im L \subseteq \ker \epsilon$ yields
\[
L(0^{\partial V} \times F^{V^\circ}) \subseteq \left\{\phi \in F^{\partial V}: \sum_{x \in \partial V} \phi(x) = 0 \right\} \times 0^{V^\circ}.
\]
Therefore, $\dim L(0^{\partial V} \times F^{V^\circ}) \leq |\partial V| - 1 < |V^\circ|$, since we assumed $|\partial V| \leq |V^\circ|$.  Therefore, by the rank-nullity theorem,
\[
\c U_0(G,L, F) = \ker(L\colon F^{V^\circ} \to F^V) \neq 0. \qedhere
\]
\end{proof}

\begin{proof}[$\partial$-graph-theoretic proof]
By Observation \ref{obs:wedgesumfunctoriality}, it suffices to show that $G$ has a sub-$\partial$-graph which is not completely reducible.  We proceed by induction on the number of vertices.

Since $G$ is nonempty and $|V^\circ| \geq |\partial V|$, $G$ must have at least one interior vertex $x$.  By assumption $x$ has some neighbor $y$, and $y$ must be a boundary vertex since the $\partial$-graph is boundary-interior bipartite.  Therefore, $G$ must have at least one boundary vertex and one interior vertex.  If $G$ has only two vertices, it must have exactly one interior vertex and one boundary vertex with at least two parallel edges between them.  Then $G$ is irreducible.

Suppose $G$ has $n > 2$ vertices and divide into cases:
\begin{itemize}
	\item If $G$ is irreducible, we are done.
	\item Suppose $G$ has a boundary spike $(x,y)$ with $x \in \partial V$ and $y \in V^\circ$.  Let $G'$ be the $\partial$-graph obtained by contracting the space.  Then $y$ is a boundary vertex in $G'$ and by assumption all its neighbors are boundary vertices in $G$.  Thus, we can delete the boundary edges incident to $y$ and then delete the now isolated boundary vertex $y$ to obtain a harmonic sub-$\partial$-graph $G'$ which satisfies the original hypotheses.  The new $\partial$-graph $G'$ is nonempty because $|V(G)| > 2$.  By inductive hypothesis, $G'$ is not completely reducible.
	\item If $G$ can be split apart as a boundary wedge sum or a disjoint union, then each piece is boundary-interior bipartite with interior vertices that have degree $\geq 2$.  Moreover, one of the two subgraphs must have $|\partial V| \leq |V^\circ|$, and hence is not completely reducible by inductive hypothesis.
	\item $G$ has no boundary edges by assumption.  Moreover, if $G$ has an isolated boundary vertex, that can be treated as a special case of disjoint unions.
\end{itemize}
\end{proof}

\section{Network Duality} \label{sec:duality}

\subsection{Dual Circular Planar Networks, Harmonic Conjugates}

As shown in \cite[Theorem 2]{CoriRossin}, dual planar graphs have isomorphic critical groups.  In this section, we generalize this result to circular planar normalized $R^\times$-networks.  The theory here adapts the ideas of duality and discrete complex analysis found in \cite[\S 2]{Mercat}, \cite[\S 10]{CMM}, \cite{Perry}.  In this section, all the networks will be normalized (that is, they will satisfy $d = 0$).

\begin{definition}
A {\bf circular planar $\partial$-graph} $G$ is a (finite) $\partial$-graph embedded in the closed unit disk $\overline{D}$ in the complex plane such that $V\cap \partial D=\partial V$. The {\bf faces} of $G$ are the components of $D\setminus G$.
\end{definition}

\begin{definition}
A connected circular planar $\partial$-graph has a {\bf circular planar dual} $G^\dagger$ defined as follows:  The vertices of $G^\dagger$ correspond to the faces of $G$; each vertex of $G^\dagger$ is placed in the interior of the corresponding face of $G$.  The edges of $G^\dagger$ correspond to the edges of $G$.  For each oriented edge $e$ of $G$, there is a dual edge $e^\dagger$ where $e_+^\dagger$ corresponds to the face on the right of $e$ and $e_-^\dagger$ corresponds to the face on the left of $e$.   A vertex of $G^\dagger$ is considered a boundary vertex if the corresponding face has a side along $\partial D$.  For further explanation and illustration, see \cite[Definition 5.1 and Figure 1]{Perry}.
\end{definition}

\begin{remark}
The planar dual is constructed in a similar fashion for a connected planar network without boundary, and the process is well explained in \cite[\S 2.1 and Figure 2]{Mercat}. To incorporate planar networks without boundary into the circular planar framework, we may designate an arbitrary vertex to be a boundary vertex and embed the $\partial$-graph into the disk.
\end{remark}

\begin{definition}
If $(G,L)$ is a circular planar normalized $R^\times$-network, then the \emph{dual network} $(G^\dagger,L^\dagger)$ is the network on $G^\dagger$ with $w(e^\dagger) = w(e)^{-1}$.  We make the same definition for planar normalized $R^\times$-networks without boundary.
\end{definition}

\begin{theorem} \label{thm:duality}
If $(G,L)$ is a connected circular planar normalized $R^\times$-network, then
\[
\tilde{\Upsilon}(G^\dagger, L^\dagger) \cong \tilde{\Upsilon}(G,L).
\]
The same holds for planar normalized $R^\times$-networks without boundary.
\end{theorem}

Theorem \ref{thm:duality} generalizes \cite[Theorem 2]{CoriRossin} to $R^\times$-networks. Our proof combines ideas from \cite[\S 26 - 29]{Biggs} and \cite[\S 7]{CM}.

\begin{proof}
Consider the circular planar case; the proof for planar networks without boundary is the same.  The result follows from reformulating $\tilde{\Upsilon}$ in terms of oriented edges rather than vertices. Recall that $C_1(G)$ is the free $R$-module on the oriented edges $E$ modulo the relations $\overline{e} = -e$ (see \S \ref{subsec:discretedifferentialgeometry}).  Then $\ker \epsilon$ can be identified with the quotient of $C_1(G)$ by the submodule generated by oriented cycles.  The cycle submodule is in fact generated by the oriented boundaries of interior faces.  Moreover, $L(RV^\circ)$ corresponds to the submodule of $C_1(G)$ generated by $\sum_{e \in \mathcal{E}(x)} w(e) e$.  The edges bounding an interior face of $G$ correspond to the edges incident to an interior vertex in $G^\dagger$.  Therefore,
\begin{align*}
\tilde{\Upsilon}(G,L) &\cong \frac{C_1(G)}{(\sum_{e^\dagger \in \mathcal{E}(x)} e: x \in V^\circ(G^\dagger)) + (\sum_{e \in \mathcal{E}(x)} w(e) e: x \in V^\circ(G))} \\
\tilde{\Upsilon}(G^\dagger,L^\dagger) &\cong \frac{C_1(G^\dagger,L^\dagger)}{(\sum_{e^\dagger \in \mathcal{E}(x)} w(e^\dagger) e^\dagger: x \in V^\circ(G^\dagger)) + (\sum_{e \in \mathcal{E}(x)} e^\dagger: x \in V^\circ(G))}.
\end{align*}
Since $w(e^\dagger) = w(e)^{-1}$, we can define an isomorphism $\tilde{\Upsilon}(G,L) \to \tilde{\Upsilon}(G^\dagger,L^\dagger)$ by $e \mapsto w(e)^{-1} e^\dagger$.
\end{proof}

Application of $\Hom(-,M)$ yields the following discrete-complex-analytic interpretation of network duality, as in \cite[\S 2]{Mercat}, \cite[\S 7]{Perry}:

\begin{proposition} \label{prop:harmonicconjugates}
Let $(G,L)$ be a circular planar normalized $R^\times$-network.  Modulo constant functions, for every $M$-valued harmonic function $u$ on $(G,L)$, there is a unique harmonic conjugate $v$ on $(G^\dagger,L^\dagger)$ satisfying the discrete Cauchy-Riemann equation $w(e)du(e) = dv(e^\dagger)$, where $du(e) = u(e_+) - u(e_-)$ and $dv(e^\dagger) = v(e_+^\dagger) - v(e_-^\dagger)$.  Moreover, a function $u: V(G) \to M$ is harmonic if and only if there exists a function $v$ such that $w(e)du(e) = dv(e^\dagger)$.  The same holds for planar normalized $R^\times$-networks without boundary.
\end{proposition}

\begin{proof}
Given our interpretation of $\tilde{\Upsilon}(G,L)$ in the previous proof, a harmonic function modulo constants is equivalent to a map $\phi\colon E(G) \to M$ such that $\phi(e)$ sums to zero around every oriented cycle and $\sum_{e \in \mathcal{E}(x)} w(e) \phi(e) = 0$ for each interior vertex; the correspondence between $u$ and $\phi$ is given by $\phi(e) = du(e)$.  For every such $\phi$, we can define a similar function $\psi$ on the dual network by $\psi(e^\dagger) = w(e) \phi(e)$.  This proves the existence and uniqueness of harmonic conjugates.

Next, we must prove that if $u$ and $v$ satisfy $w(e) du(e) = dv(e)$, then $u$ is harmonic.  But note that for each $x \in V^\circ(G)$, we have
\[
Lu(e) = \sum_{e \in \mathcal{E}(x)} w(e) du(e) = \sum_{e \in \mathcal{E}(x)} dv(e^\dagger) = 0
\]
because $\{e^\dagger: e \in \mathcal{E}(x)\}$ is a cycle in $G^\dagger$.  The proof for the case without boundary is the same.
\end{proof}

\begin{proposition} \label{prop:dualCR}
Let $G$ be a connected circular planar $\partial$-graph.  Then $G$ is completely reducible if and only if $G^\dagger$ is completely reducible.
\end{proposition}

\begin{proof}
By Theorem \ref{thm:reducibilitycharacterization}, $G$ is completely reducible if and only if for every ring $R$, for every normalized $R^\times$-network $(G,L)$ on the $\partial$-graph, $\tilde{\Upsilon}(G,L)$ is a free $R$-module.  Clearly, $(G,L) \mapsto (G^\dagger,L^\dagger)$ defines a bijection between $R^\times$-networks on $G$ and $R^\times$-networks on $G^\dagger$.  Thus, Theorem \ref{thm:duality} implies that $G$ is completely reducible if and only if $G^\dagger$ is completely reducible.
\end{proof}

\begin{remark}
There is a direct combinatorial proof of Proposition \ref{prop:dualCR} as well, which we will merely sketch here.  It requires extending the definition of dual to circular planar $\partial$-graphs which are disconnected, which is somewhat tricky and tedious since the dual is not unique; this problem is best dealt by reformulating it using medial graphs as in \cite{WJ}.  One can then show that contracting a boundary spike on $G$ corresponds to deleting a boundary edge in $G^\dagger$ and vice versa.  A decomposition of $G$ into a boundary wedge-sum or disjoint union corresponds to a similar decomposition of $G^\dagger$.
\end{remark}

\subsection{Wheel Graphs} \label{subsec:wheel}

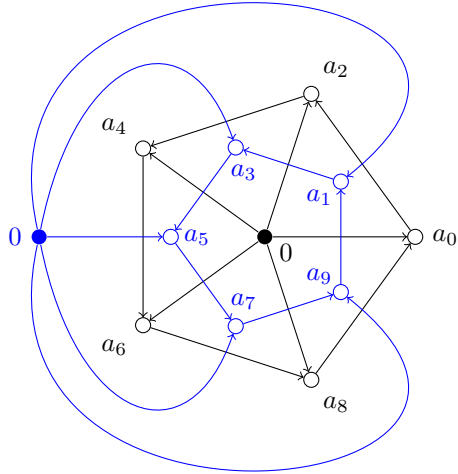
\begin{figure}

\begin{center}
\vspace{-1cm}

\begin{tikzpicture}[scale = 0.5]
 	\node[bd] (Q) at (0,0) {};
	\node at (324:0.7) {$0$};
	\node[int] (0) at (0:4) [label = 0:$a_0$] {};
	\node[int] (2) at (72:4) [label = 72:$a_2$] {};
	\node[int] (4) at (144:4) [label = 144:$a_4$] {};
	\node[int] (6) at (216:4) [label = 216:$a_6$] {};
	\node[int] (8) at (288:4) [label = 288:$a_8$] {};
	
	\begin{scope}[blue]
		\node[bd,blue] (R) at (-6,0) [label = left:$0$] {};
		\node[int,draw=blue] (1) at (36:2.5) {};
		\node at (36:1.8) {$a_1$};
		\node[int,draw=blue] (3) at (108:2.5) {};
		\node at (108:1.8) {$a_3$};
		\node[int,draw=blue] (5) at (180:2.5) {};
		\node at (180:1.8) {$a_5$};
		\node[int,draw=blue] (7) at (252:2.5) {};
		\node at (252:1.8) {$a_7$};
		\node[int,draw=blue] (9) at (324:2.5) {};
		\node at (324:1.8) {$a_9$};
	\end{scope}
	
	\begin{scope}[->]
		\draw (0) to (2); \draw[->] (2) to (4); \draw[->] (4) to (6); \draw[->] (6) to (8); \draw[->] (8) to (0);
		\draw (Q) to (0); \draw[->] (Q) to (2); \draw[->] (Q) to (4); \draw[->] (Q) to (6); \draw[->] (Q) to (8);
	\end{scope}
	\begin{scope}[->,blue]
		\draw[->] (1) to (3); \draw[->] (3) to (5); \draw[->] (5) to (7); \draw[->] (7) to (9); \draw[->] (9) to (1);
		\draw[->] (R) to (5);
		\draw (R) .. controls (-5,5) and (-2,6) .. (3);
		\draw (R) .. controls (-5,-5) and (-2,-6) .. (7);
		\draw (R) .. controls (-8,9) and (9,7) .. (1);
		\draw (R) .. controls (-8,-9) and (9,-7) .. (9);
	\end{scope}
\end{tikzpicture}
\vspace{-1cm}

\end{center}

\caption{$W_5$ and its (isomorphic) dual.  Arrows indicate the paired dual oriented edges.} \label{fig:wheel}
\end{figure}

Consider the wheel graph $W_n$ embedded in the complex plane with vertices at $\left\{e^{2\pi i k/n}\right\}_{k\in\m Z}$ and at $0$. Edges connect $0$ to $e^{2\pi ik/n}$ and $e^{2\pi ik/n}$ to $e^{2\pi i (k+1)/n}$ for all $k\in \m Z$.  Figure \ref{fig:wheel} depicts $W_5$ and its planar dual.  Note that the dual of $W_n$ is isomorphic to $W_n$.  We call the vertex $0$ the {\bf hub} and the set of vertices $\{e^{2\pi i k /n}\}$ the {\bf rim}, and we apply the same terminology to $W_n^\dagger$.  We denote the hub vertex of $W_n^\dagger$ by $0^\dagger$.

The critical group of $W_n$ is computed in \cite{Biggs2} using chip-firing, induction, and the symmetry of the graph, and a connection with Lucas sequences is uncovered.  We present an alternate approach, computing the sandpile group using harmonic continuation and planar duality.

\begin{proposition}[{\cite[Theorem 9.2]{Biggs2}}] \label{prop:wheel}
Let $W_n$ be the wheel graph and let $F_0 = 0$, $F_1 = 1$, $F_2 = 1$, $F_3 = 2$, \dots be the Fibonacci numbers.   Then
\[
\Crit(W_n) \cong \begin{cases} \Z / (F_{n-1} + F_{n+1}) \times \Z / (F_{n-1} + F_{n+1}), & n \text{ odd,} \\ \Z / F_n \times \Z / 5F_n, & n \text{ even.} \end{cases}
\]
\end{proposition}

\begin{proof}
By Proposition \ref{prop:criticalgroupnoboundary} it suffices to compute the $\Q / \Z$-valued harmonic functions modulo constants, that is,
\[
\Crit(W_n) \cong \tilde{\c U}(W_n, L_{\std}, \Q / \Z).
\]
By Proposition \ref{prop:harmonicconjugates}, it suffices to compute the $\Z$-module of pairs satisfying Cauchy-Riemann, that is,
\[
\{(u,v) \in [(\Q / \Z)^{V(W_n)}/(\text{constants}) \times (\Q / \Z)^{V(W_n^\dagger)} / (\text{constants})] \colon w(e) du(e) = dv(e^\dagger)\}.
\]
Instead of working modulo constants, we will normalize our functions so that $u$ and $v$ vanish at the hub vertices of $W_n$ and $W_n^\dagger$ respectively.  (The hub vertices are colored solid in Figure \ref{fig:wheel}).  Thus, we want to compute
\[
\{(u,v) \in [(\Q / \Z)^{V(W_n)} \times (\Q / \Z)^{V(W_n^\dagger)}] \colon u(0) = 0, v(0^\dagger) = 0, w(e) du(e) = dv(e^\dagger)\}.
\]
Let $a_0, a_1, a_2, \dots$ be the values of $u$ or $v$ on the rim vertices of $W_n$ and $W_n^\dagger$ in counterclockwise order as shown in the Figure \ref{fig:wheel}, with indices taken modulo $2n$.  The Cauchy-Riemann equations can be rewritten
\[
a_{j+1} - a_{j-1} = a_j - 0.
\]
In other words, the numbers $a_j$ satisfy the Fibonacci-Lucas recurrence $a_{j+1} = a_j + a_{j-1}$, so that
\[
\begin{pmatrix} a_{j+1} \\ a_j \end{pmatrix} = \begin{pmatrix} 1 & 1 \\ 1 & 0 \end{pmatrix} \begin{pmatrix} a_j \\ a_{j-1} \end{pmatrix}.
\]
Note that a harmonic pair $(u,v)$ is uniquely determined by $(a_1,a_0)$.  More precisely, if $A$ is the $2 \times 2$ matrix of the recursion, then $(a_1,a_0)^t \in (\Q / \Z)^2$ will produce a harmonic pair $(u,v)$ through the iteration process if and only if it is a fixed point of $A^{2n}$.  The module of harmonic pairs $(u,v)$ is thus isomorphic to the kernel of $A^{2n} - I$ acting on $(\Q/\Z)^2$.  So the invariant factors of the critical group are given by the Smith normal form of $A^{2n} - I$, which is the same as the Smith normal form of $A^n - A^{-n}$ because $A$ is invertible over $\Z$.  For $n \geq 1$,
\[
A^n = \begin{pmatrix} F_{n+1} & F_n \\ F_n & F_{n-1} \end{pmatrix}, \qquad A^{-n} = (-1)^n \begin{pmatrix} F_{n-1} & -F_n \\ -F_n & F_{n+1} \end{pmatrix}.
\]
If $n$ is odd, then
\[
A^n - A^{-n} = (F_{n+1} + F_{n-1})I,
\]
and if $n$ is even, then
\[
A^n - A^{-n} = \begin{pmatrix} F_{n+1} - F_{n-1} & 2F_n \\ 2F_n & F_{n-1} - F_{n+1} \end{pmatrix} = F_n \begin{pmatrix} 1 & 2 \\ 2 & -1 \end{pmatrix}.
\]
From here, the computation of the invariant factors is straightforward.
\end{proof}

\begin{remark}
Johnson \cite{WJ} in essence developed a system of ``discrete analytic continuation'' for harmonic conjugate pairs $(u,v)$.  Although we will not do so here, we believe future research should combine his ideas with the algebraic machinery of this paper.  Such a theory of discrete analytic continuation would have similar applications to those of Theorem \ref{thm:explicitalgorithm}.
\end{remark}

\section{Covering Maps and Symmetry} \label{sec:symmetry}

The $\partial$-graphs $\CLF(m,n)$ (\S \ref{sec:CLF}) and $W_n$ (\S \ref{subsec:wheel}) had a cyclic structure with a natural action of $\Z / m$ or $\Z / n$ by $\partial$-graph automorphisms.  In this section, we will sketch potential applications of symmetry in general, showing how symmetry imposes algebraic constraints on the group structure of $\Upsilon(G,L)$.  In particular, for $\Z$-networks, symmetry yields some information about the torsion primes of $\Upsilon(G,L)$.  In this section, we shall be brief and not develop a complete theory.  We will merely record a few simple observations for the benefit of future research.

Recall that covering maps of $\partial$-graphs were defined in Definition \ref{def:coveringmap}.  We define a {\bf covering map of $R$-networks} in the obvious way; it is an $R$-network morphism such that the underlying $\partial$-graph morphism is a covering map.  We say a covering map is {\bf finite-sheeted} if $|f^{-1}(x)|$ is finite for every $x \in V(G)$ and $|f^{-1}(e)|$ is finite for every $e \in E(G)$.  We say $f$ is {\bf $n$-sheeted} if $|f^{-1}(x)| = n$ for every $x \in V(G)$ and $|f^{-1}(e)| = n$ for every $e \in E(G)$. 

We will also the notation
\[
\c U_0^+(G,L,M) = \{u \in \c U(G,L,M) \colon u|_{\partial V(G)} = 0, Lu|_{\partial V(G)} = 0\}.
\]
This differs from $\c U_0(G,L,M)$ in that we no longer require $u$ to be finitely supported; however, for finite networks $\c U_0^+(G,L,M) = \c U_0(G,L,M)$.  Moreover, we assume familiarity with the terminology for the actions of finite groups on sets.

\begin{observation} \label{obs:coveringMap}
Let $f\colon (\tilde{G},\tilde{L}) \to (G,L)$ be a covering map.
\begin{enumerate}
	\item As in Lemma \ref{lem:upsilonfunctor} $f$ induces a surjection $\Upsilon(\tilde{G},\tilde{L}) \to \Upsilon(G,L)$ providing the following isomorphism:
\[
  \Upsilon(G,L) \cong \left.\Upsilon(\tilde{G},\tilde{L}) \middle/ \sum_{\substack{
    x,y\in V(G)\\
    f(x) = f(y)
  }} R(x - y)\right.
\]
	\item As in Lemma \ref{lem:ufunctor}, there is an injective map $f^*: \c U(G,L, M) \to \c U(\tilde{G},\tilde{L}, M)$ given by $u \mapsto u \circ f$ which identifies harmonic functions on $(G,L)$ with harmonic functions on $\tilde{G},\tilde{L})$ that are constant on each fiber of $f$.
	\item Moreover, $f^*$ restricts to an injective map $\c U_0^+(G,L,M) \to \c U_0^+(\tilde{G},\tilde{L},M)$.
	\item If $f$ is finite-sheeted, then $f^*$ restricts to an injective map $\c U_0(G,L,M) \to \c U_0(\tilde{G}, \tilde{L}, M)$.
\end{enumerate}
\end{observation}

\begin{observation} \label{obs:averaging}
Suppose $f \colon (\tilde{G},\tilde{L}) \to (G,L)$ is a finite-sheeted covering map.
\begin{enumerate}
	\item Proceeding similarly to Lemma \ref{lem:u0functor}, we can define a map
	\[
	f_*\colon \c U(\tilde{G},\tilde{L},M) \to \c U(G,L,M) \colon (f_*u)(y) = \sum_{x \in f^{-1}(y)} u(x).
	\]
	\item Moreover, $f_*$ restricts to define maps $\c U_0^+(G,L,M) \to \c U_0^+(G,L,M)$ and $\c U_0(G,L,M) \to \c U_0(G,L,M)$.
	\item If $f$ is $n$-sheeted, then $f_* \circ f^* u = n \cdot u$.
	\item Suppose $f$ is $n$-sheeted and let $M$ be an $R$-module.  Viewing $n$ as an element of $R$ via the ring morphism $\Z \to R$, we see that multiplication by $n$ defines an $R$-module morphism $n: M \to M$.  Assume $n: M \to M$ is an isomorphism and let $n^{-1}: M \to M$ denote the inverse map.  Then $n^{-1} f_* \circ f^* = \id$.  Hence, $f^*$ defines a split injection $\c U(G,L,M) \to \c U(\tilde{G}, \tilde{L}, M)$ and
	\[
	\c U(\tilde{G}, \tilde{L}, M) = f^* \c U(G,L,M) \oplus \ker f_*.
	\]
	Similarly,
	\[
	\c U_0(\tilde{G}, \tilde{L}, M) = f^* \c U_0(G,L,M) \oplus \ker f_*|_{\c U_0(\tilde{G},\tilde{L},M)}
	\]
	and the same holds for $\c U_0^+$.
\end{enumerate}
(Compare  \cite[Lemma 4.1]{bakerNor1} as well as Maschke's theorem from representation theory \cite[\S 18.1, Thm.\ 1]{DummitandFoote}.)
\end{observation}

\begin{observation} \label{obs:groupAction}
Suppose that $K$ is a group which acts by $R$-network automorphisms on the $R$-network $(\tilde{G},\tilde{L})$.  Assume the action on vertices and edges is free and that $kx \not \sim x$ for every $k \in G \setminus \{\id\}$ and every $x \in V$.
\begin{enumerate}
	\item There exists a quotient network $(G,L) = (\tilde{G},\tilde{L}) / K$ and a covering map $f \colon (\tilde{G},\tilde{L}) \to (G,L)$.
	\item There is a corresponding action of $K$ on $\c U(\tilde{G},\tilde{L}, M)$ given by $k \cdot u = k_*u$ where $k_*$ is defined as in Observation \ref{obs:averaging}.  The fixed-point submodule of this action is
	\[
	\c U(\tilde{G}, \tilde{L}, M)^K = f^* \c U(G,L,M).
	\]
	The same applies with $\c U$ replaced by $\c U_0$ or $\c U_0^+$.
	\item Suppose $K$ is a finite $p$-group for some prime $p$.  Then by a standard argument using the orbits of the $K$-action on $\c U(\tilde{G},\tilde{L},M)$, we have
	\[
	|\c U(\tilde{G},\tilde{L}, M)| \equiv |\c U(G,L,M)| \text{ mod } p,
	\]
	provided both sides are finite (for instance, assuming $\tilde{G}$ and $M$ are finite).  The same holds for $\c U_0$ and $\c U_0^+$.
\end{enumerate}
\end{observation}

While these statements hold in general, the mod $p$ counting formula seems especially useful for the case $R = \Z$.  In the following Proposition, we make use of the classification of finitely generated $\Z$-modules (see \cite[\S 12.1]{DummitandFoote}).
\begin{proposition} \label{prop:Znetworkgroupaction}
Suppose $(\tilde{G}, \tilde{L})$ is a finite non-degenerate $\Z$-network.  Suppose $K$ is a finite $p$-group which acts by $\Z$-network automorphisms on $(\tilde{G},\tilde{L})$ as in Observation \ref{obs:groupAction}, let $(G,L)$ be the quotient network, and let $f: (\tilde{G},\tilde{L}) \to (G,L)$ be the projection map.
\begin{enumerate}
	\item The $\Z$-network $(G,L)$ is finite and non-degenerate.
	\item The generalized critical group $\Upsilon(\tilde{G},\tilde{L})$ has nontrivial $p$-torsion if and only if $\Upsilon(G,L)$ has nontrivial $p$-torsion.
	\item Let $q$ be a prime distinct from $p$ and let $k \in \Z$.  Then
	\[
	\c U_0(\tilde{G}, \tilde{L}, \Z / q^k) = f^* \c U_0(G, L, \Z / q^k) \oplus M_{q^k}
	\]
	where $M_{q^k} := \ker f_*|_{\c U_0(\tilde{G}, \tilde{L}, \Z / q^k)}$.
	\item For $q \neq p$, the action of $K$ on $M_{q^k}$ has no fixed points other than zero and in particular $|M_{q^k}| \equiv 1$ mod $p$.
\end{enumerate}
\end{proposition}

\begin{proof}
(1) We assume covering maps to be surjective on the vertex and edge sets by definition; thus, since $\tilde{G}$ is finite, $G$ is also finite.  Because $f^*$ defines an injective map $\c U_0(G,L,\Z) \to \c U_0(\tilde{G},\tilde{L}, \Z) = 0$, we know $(G,L)$ is non-degenerate.

(2) Because the networks are non-degenerate, Proposition \ref{prop:tor} shows that
\[
\Tor_1(\Upsilon(\tilde{G},\tilde{L}), \Z / p) \cong \c U_0(\tilde{G},\tilde{L}, \Z / p)
\]
and the same holds for $(G,L)$.  On the other hand, by Observation \ref{obs:groupAction} (3), we have
\[
|\c U_0(\tilde{G},\tilde{L}, \Z / p)| \equiv |\c U_0(G,L, \Z / p)| \text{ mod } p.
\]
Each of the two $\Z$-modules in this equation is either zero (hence has cardinality one mod $p$) or else it has cardinality zerp mod $p$, which implies (2).

(3) Note that $f$ is a $|K|$-sheeted covering map.  Since $|K|$ is a power of $p$, multiplication by $|K|$ acts as an isomorphism on $\Z / q^k$.  Therefore, claim (3) follows from Observation \ref{obs:averaging} (4).

(4) It follows from (3) and Observation \ref{obs:groupAction} (2) that zero is the only fixed point of the $K$-action on $M_{q^k}$.  Since $K$ is a $p$-group, we thus have $|M_{q^k}| \equiv 1$ mod $p$.
\end{proof}

\begin{example}
Consider the networks $\CLF(m,n)$ from \S \ref{sec:CLF}.  There is an obvious translation action of $\Z / k$ on $\CLF(km,n)$ with the quotient $\CLF(m,n)$.  The covering map $\CLF(km,n) \to \CLF(m,n)$ induces an inclusion $\c U_0(\CLF(m,n), L_{\std}, \Q/\Z) \to \c U_0(\CLF(km,n), L_{\std}, \Q / \Z)$.  Note that when $k$ is a power of $2$, Proposition \ref{prop:Znetworkgroupaction} (2) holds because $\c U_0(\CLF(m,n), \Z / 2^\ell)$ is nontrivial for all $m \geq 2$ and $n \geq 1$ by Theorem \ref{thm:CLF}.  Moreover, for odd integers $k$, we have $\c U_0(\CLF(m,n), \Z / k) = 0$ for all $m$, so (2) also holds when $k$ is an odd prime power.  One can verify that the other claims in Proposition \ref{prop:Znetworkgroupaction} also hold rather vacuously in the case of $\CLF(m,n)$ as well.
\end{example}

\begin{remark}
Though Proposition \ref{prop:Znetworkgroupaction} falls far short of computing $\Upsilon(\tilde{G},\tilde{L})$ from $\Upsilon(G,L)$, it nonetheless gives a significant amount of information, especially in parts (3) and (4).  Indeed, one can argue from the classification of finite $\Z$-modules that the $q$-torsion component of $\Upsilon(\tilde{G},\tilde{L})$ is uniquely determined up to isomorphism by the quantities $|\Tor_1(\Upsilon(\tilde{G},\tilde{L}), \Z / q^k)|$ for $k = 0, 1, \dots$.   Moreover, by (3)
\[
|\Tor_1(\Upsilon(\tilde{G},\tilde{L}), \Z / q^k)| = |\c U_0(G,L, \Z / q^k)| \cdot |M_{q^k}|.
\]
By (4), we know $|M_{q^k}|$ is a power of $q$ which equals $1$ mod $p$ and that the group $K$ acts by automorphisms on $M_{q^k}$ with no nontrivial fixed points.  This narrows down the possibilities for $|M_{q^k}|$, especially when combined with other information such as bounds on the number of invariant factors for the torsion part of $\Upsilon(G,L)$ from Corollary \ref{cor:invariantfactorsbound} or bounds on the size of $\Tor_1(\Upsilon(\tilde{G}, \tilde{L}), \Q / \Z)$ obtained through determinantal computations.
\end{remark}

As stated, Proposition \ref{prop:Znetworkgroupaction} does not yield optimal information for the case of graphs without boundary and the critical group since it relies on non-degeneracy.  The simplest way to handle this problem is by considering $\partial$-graphs with one boundary vertex (see Proposition \ref{prop:criticalgrouponeboundary}) and allowing one branching point in our covering map.

\begin{definition}
Let $\tilde{G}$ and $G$ be $\partial$-graphs with exactly one boundary vertex each, called $\tilde{x}$ and $x$ respectively.  A {\bf pseudo-covering map} $f \colon \tilde{G} \to G$ is a $\partial$-morphism such that $f$ is surjective on the vertex and edge sets, $f$ maps $\tilde{x}$ to $x$, $f$ maps interior vertices to interior vertices, $f$ maps edges to edges, and $\deg(f,y) = 1$ for every $y \in V(\tilde{G}) \setminus \{\tilde{x}\}$.
\end{definition}

The foregoing observations all adapt to pseudo-covering maps for \emph{normalized} $R$-networks (and in particular apply to critical groups).  The verifications are straightforward once we make the following observation:  Let $G$ be a $\partial$-graph with a single vertex $x$ and let $L$ be a weighted Laplacian (recall this means $d = 0$).  If $u: V \to M$ satisfies $Lu(y) = 0$ for all $y \neq x$, then it also satisfies $Lu(x) = 0$ because $\sum_{y \in V(G)} Lu(y) = 0$.

\begin{example}
Let $W_n$ be the wheel graph from \S \ref{subsec:wheel} where $0$ is considered a boundary vertex.  For any $k \in \N$, there is a group action of $\Z / k$ on $W_{kn}$ by rotation and a corresponding quotient map $W_{kn} \to W_n$ which is a pseudo-covering map.  By combining the results from \ref{subsec:wheel} with the results from this section, we obtain the following information about the $q$-torsion components of $\Crit(W_n)$ for each prime $q$.

(1) The $q$-torsion component has at most two invariant factors.  Indeed, the harmonic continuation argument in Proposition \ref{prop:wheel} showed that $\c U_0(W_n, L_{\std}, \Q / \Z)$ is isomorphic to the submodule of $(\Q / \Z)^2$ consisting of fixed points of $A^{2n}$.  A submodule of $(\Q / \Z)^2$ can have at most two invariant factors.  Since $\c U_0(W_n, L_{\std}, \Q / \Z)$ has at most two invariant factors, so does its $q$-torsion component.

(2) For every $k$, there exists some $n$ such that $\c U_0(W_n, L_{\std}, \Z / q^k) \cong (\Z / q^k)^2$.  To prove this, it suffices to show that every $\phi \in (\Z / q^k)^2$ will be a fixed point of $A^{2n}$ for some $n$.  Note that $A$ maps $(\Z / q^k)^2$ into itself and $(\Z / q^k)^2$ is finite, so there must exist two distinct integers $k$ and $\ell$ with $A^{2k} \phi = A^{2\ell} \phi$.  Since $A$ is invertible over $\Z$, we have $A^{2(k-\ell)} \phi = \phi$, so we can take $n = k - \ell$.

(3) If $q$ is a prime other than $5$, then we know from Proposition \ref{prop:wheel} that the $q$-torsion component of $\Crit(W_n)$ has the form $(\Z / q^k)^2$ for some $k$.  Moreover, the $5$-torsion component has the form $(\Z / 5^k)^2$ for odd $n$ and $\Z / 5^k \times \Z / 5^{k+1}$ for even $n$.

(4) If $m | n$, then there is a pseudo-covering map $W_m \to W_n$ and hence by Observation \ref{obs:coveringMap} (4), we can identify the $q$-torsion component of $\Crit(W_n)$ with a submodule of the $q$-torsion component for $\Crit(W_m)$.

(5) Suppose $n$ is such that the $q$-torsion component $\Crit(W_n)$ has two invariant factors, and let $p$ be a prime other than $q$.  Then $\Crit(W_{pn})$ has the same $q$-torsion submodule as $\Crit(W_n)$.  Indeed, multiplication by $p$ acts as an isomorphism on $\Z / q^k$.  Thus, by Proposition \ref{prop:Znetworkgroupaction} (3), we have
\[
\c U_0(W_{pn}, L_{\std}, \Z / q^k) \cong \c U_0(W_n, L_{\std}, \Z / q^k) \oplus M_{q^k}.
\]
We know that $\c U_0(W_n, L_{\std}, \Z / q^k)$ has two invariant factors, while $\c U_0(W_{pn}, L_{\std}, \Z / q^k)$ has at most two invariant factors.  This implies that $M_{q^k} = 0$ and hence
\[
\c U_0(W_{pn}, L_{\std}, \Z / q^k) \cong \c U_0(W_n, L_{\std}, \Z / q^k).
\]
Since this holds for all $k$, the $q$-torsion components of $\Crit(W_{pn})$ and $\Crit(W_n)$ are isomorphic.
\end{example}

\section{Open Problems} \label{sec:openproblems}

Much like the sandpile group, the fundamental module $\Upsilon$ connects ideas from network theory, combinatorics, algebraic topology, homological algebra, and complex analysis.  We have correlated the algebraic properties of $\Upsilon$ with the combinatorial properties of $\partial$-graphs, including $\partial$-graph morphisms, layer-stripping, boundary wedge-sums, duality, and symmetry and we have given applications to the critical group and Laplacian eigenvalues.  Our results lead to the following questions:

\begin{question}
Do our algebraic invariants extend to higher-dimensional cell complexes, along the lines of \cite{DKM}?  Do they generalize to directed graphs?  What are the analogues of $\partial$-graph morphisms and layer-stripping in these settings?
\end{question}

\begin{question}
Can the techniques developed herein (particularly Theorem \ref{thm:explicitalgorithm}) be used to aid the computation of previously intractable sandpile groups?  What applications do they have for computing eigenvectors and characteristic polynomials?
\end{question}

\begin{question}
Corollary \ref{cor:invariantfactorsbound} used layer-stripping to give a bound on the number of invariant factors for $\Crit(G)$ and the multiplicity of eigenvalues.  How sharp is this bound for general graphs?  For a graph without boundary, is there an algebraic characterization of the minimal number of boundary vertices one has to assign to achieve layerability?  What is the most efficient algorithm for finding a choice of boundary vertices that achieves this minimal number?
\end{question}

\begin{question}
Are there other operations on $\partial$-graphs which interact nicely with $\Upsilon$ and with $\partial$-graph morphisms?  Can such operations be used to compute $\Upsilon$ or at least produce short exact sequences?  See Remark \ref{rem:nonunitlayering} and \cite[Proposition 2]{Lor1}, \cite[Proposition 21]{Treumann}.
\end{question}

\begin{question}
We have studied algebraic invariants which test layerability (Theorem \ref{thm:layerabilitycharacterization}).  Are there algebraic invariants of $\partial$-graphs which test whether or not the electrical inverse problem can be solved by layer-stripping?
\end{question}

\begin{question}
Do Theorems \ref{thm:layerabilitycharacterization} and \ref{thm:reducibilitycharacterization} extend to infinite $\partial$-graphs? In particular, for a fixed infinite graph $G$, if $\Upsilon(G,L)$ is flat for all unit edge-weight functions $w$, must $\Upsilon(G,L)$ also be free for all unit edge-weight functions?
\end{question}

\begin{question}
Determine the $\Z$-module of $\m Q/\m Z$-harmonic functions supported in a given subset of the $\m Z^2$ lattice. Applying Lemma \ref{lem:M1computation} to $\CLF(\infty,n)$ resolves the case of a diagonal strip with sides parallel to the lines $y=\pm x$.  An argument using harmonic continuation shows that these are the only \emph{strips} with a nonzero answer.
\end{question}

\begin{question}
Can the techniques of \S \ref{sec:CLF} be modified to handle $\partial$-graphs built from the triangular or hexagonal lattice rather than the rectangular lattice?
\end{question}

\subsection*{Acknowledgments:} The ideas in this paper were in part developed at the University of Washington REU in Electrical Inverse Problems (summer 2015), in which David Jekel and Avi Levy were graduate student TAs, and the undergraduates Will Dana, Collin Litterell, and Austin Stromme were students.  We all owe a great debt to James A.\ Morrow for organizing the REU, participating in discussions, and encouraging our interest in networks.

Furthermore, we thank the organizers of the CMO-BIRS workshop on Sandpile Groups for their support, encouragement, and hospitality. One of the authors presented a preliminary version of these results at this workshop and would like to thank Dustin Cartwright, Caroline Klivans, Lionel Levine, Jeremy Martin, David Perkinson, and Farbod Shokrieh for stimulating discussions.

We thank the referees and journal editors of SIDMA for many helpful corrections and suggestions on exposition, as well as for pointing out several references.

\bibliographystyle{siam}
\bibliography{torsion}

\begin{thebibliography}{10}

\bibitem{ALT}
{\sc J.~Alman, C.~Lian, and B.~Tran}, {\em Circular planar electrical networks:
  Posets and positivity}, J. Combin. Theory Ser. A, 132 (2015), pp.~58--101.

\bibitem{AtMac}
{\sc M.~F. Atiyah and I.~G. MacDonald}, {\em Introduction to Commutative
  Algebra}, Westview Press, 1969.

\bibitem{Bai}
{\sc H.~Bai}, {\em On the critical group of the n-cube}, Linear Algebra Appl.,
  369 (2003), pp.~251--261.

\bibitem{BakerFaber}
{\sc M.~Baker and X.~Faber}, {\em Metric properties of the tropical
  {A}bel--{J}acobi map}, J. Algebraic Combin., 33 (2011), pp.~349--381.

\bibitem{bakerNor2}
{\sc M.~Baker and S.~Norine}, {\em Riemann--{R}och and {A}bel--{J}acobi theory
  on a finite graph}, Adv. Math., 215 (2007), pp.~766 -- 788.

\bibitem{bakerNor1}
\leavevmode\vrule height 2pt depth -1.6pt width 23pt, {\em Harmonic morphisms
  and hyperelliptic graphs}, Int. Math. Res. Not. IMRN, 2009 (2009),
  pp.~2914--2955.

\bibitem{Berman}
{\sc K.~A. Berman}, {\em Bicycles and spanning trees}, SIAM Journal on
  Algebraic Discrete Methods, 7 (1986), pp.~1--12.

\bibitem{Biggs}
{\sc N.~L. Biggs}, {\em Algebraic potential theory on graphs}, Bull. Lond.
  Math. Soc., 29 (1997), pp.~641--682.

\bibitem{Biggs2}
\leavevmode\vrule height 2pt depth -1.6pt width 23pt, {\em Chip-firing and the
  critical group of a graph}, J. Algebraic Combin., 9 (1999), pp.~25--45.

\bibitem{BLS}
{\sc A.~Bj{\"o}rner, L.~Lov{\'a}sz, and P.~W. Shor}, {\em Chip-firing games on
  graphs}, European Journal of Combinatorics, 12 (1991), pp.~283--291.

\bibitem{BobenkoGunther}
{\sc A.~I. Bobenko and F.~G{\"u}nther}, {\em Discrete complex analysis on
  planar quad-graphs}, in Advances in Discrete Differential Geometry, A.~I.
  Bobenko, ed., Springer Berlin Heidelberg, Berlin, Heidelberg, 2016,
  pp.~57--132.

\bibitem{dVGV}
{\sc Y.~{Colin de Verdiere}, I.~Gitler, and D.~Vertigan}, {\em Reseaux
  \'electriques planaires ii}, Comment. Math. Helv., 71 (1996), pp.~144--167.

\bibitem{CoriRossin}
{\sc R.~Cori and D.~Rossin}, {\em On the sandpile group of dual graphs},
  European Journal of Combinatorics, 21 (2000), pp.~447 -- 459.

\bibitem{CIM}
{\sc E.~Curtis, D.~Ingerman, and J.~Morrow}, {\em Circular planar graphs and
  resistor networks}, Linear Algebra Appl., 283 (1998), pp.~115 -- 150.

\bibitem{CMM}
{\sc E.~B. Curtis, E.~Mooers, and J.~A. Morrow}, {\em Finding the conductors in
  circular networks from boundary measurements}, Mathematical Modelling and
  Numerical Analysis, 28 (1994), pp.~781--814.

\bibitem{CMdn}
{\sc E.~B. Curtis and J.~A. Morrow}, {\em The dirichlet to neumann map for a
  resistor network}, SIAM Journal on Applied Mathematics, 51 (1991),
  pp.~1011--1029.

\bibitem{CM}
\leavevmode\vrule height 2pt depth -1.6pt width 23pt, {\em Inverse Problems for
  Electrical Networks}, World Scientific, 2000.

\bibitem{Dhar}
{\sc D.~Dhar}, {\em Self-organized critical state of sandpile automaton
  models}, Phys. Rev. Lett., 64 (1990), pp.~1613--1616.

\bibitem{DummitandFoote}
{\sc D.~S. Dummit and R.~M. Foote}, {\em Abstract Algebra}, John Wiley and
  Sons, Inc., 2004.

\bibitem{DKM}
{\sc A.~M. Duval, C.~J. Klivans, and J.~L. Martin}, {\em Critical groups of
  simplicial complexes}, Ann. Comb., 17 (2013), pp.~53--70.

\bibitem{RF}
{\sc R.~Forman}, {\em Determinants of laplacians on graphs}, Topology, 32
  (1993), pp.~35 -- 46.

\bibitem{GodsilRoyle}
{\sc C.~Godsil and G.~Royle}, {\em Algebraic Graph Theory}, Springer-Verlag New
  York, 2001.

\bibitem{HLMPPW}
{\sc A.~E. Holroyd, L.~Levine, K.~M{\'e}sz{\'a}ros, Y.~Peres, J.~Propp, and
  D.~B. Wilson}, {\em Chip-firing and rotor-routing on directed graphs}, in In
  and Out of Equilibrium 2, V.~Sidoravicius and M.~E. Vares, eds.,
  Birkh{\"a}user Basel, Basel, 2008, pp.~331--364.

\bibitem{layering}
{\sc D.~Jekel}, {\em Layering $\partial$-graphs and networks: $\partial$-graph
  transformations, discrete harmonic continuation, and a generalized electrical
  inverse problem}.
\newblock Unpublished paper at arXiv:1601.00247, 2016.

\bibitem{WJ}
{\sc W.~Johnson}, {\em Circular planar resistor networks with nonlinear and
  signed conductors}.
\newblock Unpublished paper at arXiv:1203.4045, 2012.

\bibitem{RK}
{\sc R.~Kenyon}, {\em The laplacian on planar graphs and graphs on surfaces},
  Current Developments in Mathematics,  (2011).

\bibitem{LPcyl}
{\sc T.~Lam and P.~Pylyavskyy}, {\em Inverse problem in cylindrical electrical
  networks}, SIAM J. Appl. Math., 72 (2012), pp.~767--788.

\bibitem{LP}
\leavevmode\vrule height 2pt depth -1.6pt width 23pt, {\em Electrical networks
  and lie theory}, Algebra Number Theory, 9 (2015), pp.~1401--1418.

\bibitem{lamRing}
{\sc T.-Y. Lam}, {\em A First Course in Noncommutative Rings}, Springer-Verlag
  New York, 2013.

\bibitem{levine}
{\sc L.~Levine}, {\em The sandpile group of a tree}, European Journal of
  Combinatorics, 30 (2009), pp.~1026--1035.

\bibitem{Lor2}
{\sc D.~J. Lorenzini}, {\em Arithmetical graphs}, Math. Ann., 285 (1989),
  pp.~481--501.

\bibitem{Lor1}
\leavevmode\vrule height 2pt depth -1.6pt width 23pt, {\em A finite group
  attached to the laplacian of a graph}, Discrete Math., 91 (1991), pp.~277 --
  282.

\bibitem{MacLane}
{\sc S.~MacLane}, {\em Homology}, Springer-Verlag New York, 1975.

\bibitem{Mercat}
{\sc C.~Mercat}, {\em Discrete riemann surfaces}, Comm. Math. Phys., 218
  (2001), pp.~77--216.

\bibitem{Perry}
{\sc K.~Perry}, {\em Discrete complex analysis}.
\newblock Unpublished paper from Univ. of Wash. Math REU, 2003.

\bibitem{Reichert}
{\sc N.~Reichert}, {\em The smallest recoverable flower}.
\newblock Unpublished paper from Univ. of Wash. Math REU, 2004.

\bibitem{Solomyak}
{\sc R.~Solomyak}, {\em On the coincidence of entropies for two types of
  dynamical systems}, Ergodic Theory and Dynamical Systems, 18 (1998),
  pp.~731--738.

\bibitem{Spencer}
{\sc J.~Spencer}, {\em Balancing vectors in the max norm}, Combinatorica, 6
  (1986), pp.~55--65.

\bibitem{Treumann}
{\sc D.~Treumann}, {\em Functoriality of critical groups}.
\newblock Undergraduate thesis at Univ. of Minn., 2002.

\bibitem{Urakawa}
{\sc H.~Urakawa}, {\em A discrete analogue of the harmonic morphism and green
  kernel comparison theorems}, Glasg. Math. J., 42 (2000), pp.~319--334.

\bibitem{Vermani}
{\sc L.~R. Vermani}, {\em An elementary approach to homological algebra}, CRC
  Press, 2003.

\bibitem{Weibel}
{\sc C.~A. Weibel}, {\em An introduction to homological algebra}, Cambridge
  University Press, 1995.

\end{thebibliography}

\end{document}